\documentclass[3p]{elsarticle}

\usepackage{amssymb}
\usepackage{amsmath}
\usepackage{amsthm}
\usepackage[colorlinks=true, allcolors=blue]{hyperref} %%

\usepackage{xspace}
\usepackage{multirow}
\usepackage{bbm}
\usepackage{Macros/mydef}
\usepackage{Macros/remarks}
\usepackage{algorithm}
\usepackage[algo2e]{algorithm2e} 
\usepackage{algpseudocode}
\usepackage[numbers]{natbib}
\usepackage{graphicx}
\usepackage{subcaption}
\usepackage{enumitem}
\usepackage{mathtools}
\mathtoolsset{showonlyrefs}
\usepackage{textcomp,mathcomp}
\usepackage{booktabs}

\usepackage[dvipsnames]{xcolor}

\usepackage{longtable}

\usepackage[utf8]{inputenc}
\usepackage[english]{babel}

\usepackage{float}
\usepackage[colorinlistoftodos]{todonotes}
\usepackage[utf8]{inputenc}
\usepackage[export]{adjustbox}
\usepackage{expl3}[2012-07-08]
\ExplSyntaxOn
\ExplSyntaxOff
\newcommand{\uu}{  \bu \SCAL \bu } 
\newcommand{\uuDG}{ \overline{ \bu \SCAL \bu } } 
\newcommand{\uuDGn}{  \overline{ \bu^n \SCAL \bu^{n+1} } } 
\newcommand{\uum}{  \bu^{n+1/2} \SCAL \bu^{n+1/2} } 
\newcommand{\uumm}{  \bu^{n} \SCAL \bu^{n+1} }

\usepackage{ulem}
\theoremstyle{plain} 

\newtheorem{theorem}{Theorem}
\numberwithin{theorem}{section}

\numberwithin{lemma}{section}

\newtheorem{proposition}{Proposition}
\numberwithin{proposition}{section}

\theoremstyle{definition}
\newtheorem{remark}{Remark}
\numberwithin{remark}{section}

\theoremstyle{definition}

\numberwithin{definition}{section}

\theoremstyle{definition}

\numberwithin{example}{section}

\mathchardef\ordinarycolon\mathcode`\:
\mathcode`\:=\string"8000
\begingroup \catcode`\:=\active
  \gdef:{\mathrel{\mathop\ordinarycolon}}
\endgroup

\let\cite=\citet
\newcommand{\commentA}[1]{\textcolor{black}{#1}}

\newcommand{\commentB}[1]{\textcolor{black}{#1}}

\newcommand{\commentC}[1]{\textcolor{black}{#1}}
\newcommand{\commentD}[1]{\textcolor{black}{#1}}
\newcommand{\commentE}[1]{\textcolor{black}{#1}}

\newcommand{\new}[1]{\textcolor{black}{#1}}
\newcommand{\neww}[1]{\textcolor{black}{#1}}
\addto\captionsenglish{\renewcommand{\appendixname}{Appendix }}

\journal{Journal of Computational Physics}

\begin{document}

%\renewcommand{\thefootnote}{\fnsymbol{footnote}} \footnotetext[1]{This material is based upon work supported
%  by the Swedish Research Council (VR) and the Swedish Foundation for
%  Strategic Research (SSF)...}
%\footnotetext[2]{Computational Technology Laboratory, Numerical
%  Analysis Department, KTH, 100 44 Stockholm, Sweden, \
%  E-mail address:}
%\renewcommand{\thefootnote}{\arabic{footnote}}

\date{}
\begin{frontmatter}
\title{A fully conservative and shift-invariant formulation for Galerkin discretizations of incompressible variable density flow}
\author[1]{Lukas Lundgren\corref{cor1}}
\ead{lukas.lundgren@math.su.se}
\author[1]{Murtazo Nazarov}
\ead{murtazo.nazarov@uu.se}
\cortext[cor1]{Corresponding author}
\address[1]{Department of Information Technology, Division of Scientific Computing, Uppsala University, Sweden}
\begin{abstract}
  This paper introduces a formulation of the variable density incompressible Navier-Stokes equations by modifying the nonlinear terms in a consistent way. For Galerkin discretizations, the formulation \neww{leads to favorable discrete conservation properties} without the divergence-free constraint being strongly enforced. In addition, the formulation is shown to make the density field invariant to global shifts. The effect of viscous regularizations on conservation properties is also investigated. Numerical tests validate the theory developed in this work. \new{The new formulation shows superior performance compared to other formulations from the literature, both in terms of accuracy for smooth problems and in terms of robustness.}

\end{abstract}
\end{frontmatter}

% with distinct phases, \ie air and carbon dioxide, oil and water and so on.

% This also occurs 

% This occurs in nature when pairs of fluids interact. These pairs can be distinct, such as oil-water and so on. Or they can mix such as air-co2, salt water - fresh water.

% Pairs of fluids can be described by the variable density navier-stoeks equarions and can be used to model two-phase flows [XYZ] and mixing of fluids as demonstrated by for example [Birman and company].

% It is also a good model when modeling mixing of fluids with the same phase such as air-co2, salt water-fresh water and so on.

% Two-phase flows can either mix by diffusion or be distinct and not mix. Additional terms can be added to model gravity, chemical reactions and surface tension. The application of multiphase flows is very broad is important

% are very important when modeling combustion, the spread of pollution, 

%As a starting point they can be used to model a pair of fluids. Later, additional terms can be added to model gravity, diffusion, chemical reactions, surface tension and so on.
% As shown by XYZ they can be used to model two-phase flows. And as shown by Birman et al they can be used to model the mixing of fluids.

\section{Introduction.} 

The simulation of incompressible flow is important for industrial and scientific applications when modeling subsonic flow. In this work, we consider incompressible flow when the density field is variable, which \commentB{arises when multiple fluids interact. A pair of fluids can remain immiscible as seen in oil/water, gas/liquids and so on. Alternatively, fluids can diffuse and mix, as observed in pairs of gases and certain pairs of liquids such as salt water/fresh water, alcohol/water and more. These phenomena can be modeled using the incompressible Navier-Stokes equations with variable density, along with an advection equation for the density field. Complexity is often added to the model by introducing additional terms to account for factors such as gravity, chemical reactions and surface tension. At a continuous level, when viscosity effects are neglected,} these equations can be shown to conserve mass, squared density, momentum, angular momentum and kinetic energy. The density field can also be shown to be shift-invariant which means that density remains invariant if it is shifted by a constant. These properties are fundamental to the flow model and are closely linked to the so-called divergence-free constraint which is pointwise satisfied. It is favorable for numerical approximations of variable density flow to fulfill as many of these physical properties as possible. 

Unfortunately, numerical approximations of the incompressible Navier-Stokes equations are rarely pointwise divergence-free \citep{Lundgren2022} which leads to issues with conservation. In particular, most Galerkin methods are known to satisfy the divergence-free condition only weakly and much work has been performed to improve the conservation properties of these methods. Recently, \cite{Charnyi2017} showed that, by consistently modifying the nonlinear term in the constant density Navier-Stokes equations, a fully conservative Galerkin method could be derived. The novelty of the formulation is that a modified pressure is solved for and the formulation has been named the energy, momentum and angular momentum conserving (EMAC) formulation. The EMAC formulation has later been extended to be more efficient by \cite{Charnyi2019}, further analyzed by \cite{Olshanskii2020} and extended to projection methods by \cite{Ingimarson2023}. 

Unlike constant density flow, variable density flow has additional conserved properties and no fully conservative formulation is currently known in the literature. The current state-of-the-art formulations for variable density flow are based on skew-symmetric formulations that can be shown to conserve squared density and kinetic energy \citep{Guermond2000} or a formulation that can be shown to conserve mass, kinetic energy, momentum and angular momentum \citep{Zhang_2023,Manzanero_2020}. In this work, we extend the modified pressure technique \citep{Charnyi2017} to variable density flow which leads to a new formulation that, when discretized by a Galerkin method, \neww{is shift-invariant and also conserves mass, squared density, momentum and angular momentum. We show that the leading error term for discrete kinetic energy conservation is independent of the divergence-free constraint. The modified Galerkin method by \citep{Gawlik2020} is also shift-invariant but instead leads to conservation of mass, squared density and kinetic energy with a leading error term for discrete momentum and angular momentum conservation independent of the divergence-free constraint.} The analysis is first performed on the inviscid problem, \ie the incompressible Euler equations with variable density. Later, since viscous regularizations affect the conservation properties of the model, we also investigate this. 

The theoretical findings are summarized in Theorems \ref{theorem div} and \ref{theorem:momentum-viscous-flux}, and these are verified using numerical experiments. The results show that the new formulation leads to improved accuracy and robustness compared to existing formulations in the variable density literature. Furthermore, the new formulation can easily be implemented into existing Galerkin-based variable density Navier-Stokes solvers. 

% In this work, we extend the modified pressure technique \citep{Charnyi2017} to variable density flow which leads to a new formulation that, when discretized by a Galerkin method, is shift-invariant and also conserves mass, squared density, kinetic energy, momentum and angular momentum. The analysis is first performed on the inviscid problem, \ie the incompressible Euler equations with variable density. Later, since viscous regularizations affect the conservation properties of the model, we also investigate this. 

% Shift invariance is new to the literature. But many use a level-set approach to achieve shift invariance...

The paper is organized the following way: In Section \ref{Sec:prelim} we give the governing equations, present our notation and explain the properties of the governing equations that we later want our proposed numerical method to mimic. Then, in Section \ref{Sec:div_neq_properties} we introduce an alternative formulation of the governing equations which, when discretized using a Galerkin method, leads to improved conservation properties in the semi-discrete form. In Section \ref{Sec:viscous_regularization} we present a family of viscous regularizations that lead to various conserved properties of the model. In Section \ref{Sec:fully_discrete} a conservative second-order accurate time discretization is presented which leads to a fully discrete method. In Section \ref{Sec:gresho_sections} we perform numerical validations. In Section \ref{Sec:conclusion} we give concluding remarks.

\section{Preliminaries.}\label{Sec:prelim}
In this section, we introduce the governing equations that model variable density flow. We start with the inviscid case, \ie we only consider the incompressible Euler equations with variable density. We also discuss the properties of the governing equations and present our notation.

\subsection{The inviscid model problem.}\label{Sec:Equation}
We first consider the incompressible Euler equations with variable density in a domain $\Omega \subset \mb R^d$ where $d$ is the dimension and with finite time interval $[0, T]$
\begin{equation}\label{eq:cons_law_primitive}
      \begin{aligned}
        \p_t \rho + \bu \SCAL \GRAD \rho  & = 0,\\
 \p_t \bbm + \bu \SCAL \GRAD  \bbm  + \GRAD p  &= \bef, \quad &(\bx,t) \in  \Omega \CROSS (0,T],\\
           \DIV \bu &= 0,\\
           \bu(\bx,0) &= \bu_{0}(\bx), \quad  &\bx\in \Omega,\\
           \rho(\bx,0) &= \rho_{0}(\bx),    
      \end{aligned}
\end{equation}
where the density $\rho > 0$, the velocity field $\bu$ and the pressure $p$ are the unknowns. We define $\bbm := \rho \bu$ as the momentum vector, $\bef(\bx,t)$ represents an external force, $\rho_{0}(\bx)$, $\bu_{0}(\bx)$ are initial conditions for density and velocity. We assume that the governing equations are supplied with well-posed boundary conditions. To simplify the analysis we set $\bef = 0$ and consider periodic boundary conditions. % \lukas{and that $\bef = 0$}. %We later consider viscous regularizations of \eqref{eq:cons_law_primitive} in Section \ref{Sec:viscous_regularization}.

\begin{remark}
It is possible to write the governing equations in primitive form, \ie $\rho \new{\p_t} \bu + \bbm \SCAL \GRAD \bu$ instead of $\new{\p_t} \bbm + \bu \SCAL \GRAD \bbm$. All of the results presented in this manuscript can straightforwardly be extended to this form. See \ref{Sec:alternative_formulations} for more details. 
\end{remark}

\subsection{Notation and useful identities.} 
Let $\rho, w \in H^1(\Omega)$ and $\bu, \bw, \bv \in \bH^1(\Omega) := [H^1(\Omega)]^d$. \commentA{We note that $ \rho \bu = \bbm \in [H^1(\Omega)]^d$.} In this work the \commentE{$L^2$-inner product is written as $( \SCAL , \SCAL )$ with associated $L^2$-norm} $\| \cdot \|$ and we define $(\GRAD \bu) := \partial_{x_i} u_j $ \commentB{and $(\bu \otimes \bv)_{ij} = u_i v_j$}. We also define a trilinear form

\begin{equation} \label{eq:trilinear definition}
b( \bu, \bv, \bw) :=  (\bu \SCAL \GRAD \bv, \bw) = \l( ( \GRAD \bv)^\top \bu, \bw \r) = ( (\GRAD \bv) \bw , \bu).
\end{equation}

% Let $\bbm := \rho \bu$, we then have the following vector identities

% Repetition?
% Provided that $\bu, \bw, \bv \in \bH^1_0(\Omega)$ the following relations hold due to integration by parts

% \begin{align}
% (\bu \SCAL \GRAD \bv, \bw) = -( (\DIV \bu) \bv, \bw  ) - ( \bu \SCAL \GRAD \bw, \bv ), \label{eq:IBP} \\
% (\bu \SCAL \GRAD \bv, \bv) = - \frac{1}{2} ( (\DIV \bu) \bv, \bv  ), \label{eq:IBP2} \\
% ( \GRAD \CROSS \bu, \bv ) = ( \bu, \GRAD \CROSS \bv ). \label{eq:IBP cross}
% \end{align}
% Repetition?

% We don't need the thing below anymore...
% We denote symmetric part of $\GRAD \bu$ as $\GRAD^s \bu := \frac{1}{2} ( \GRAD \bu + (\GRAD \bu)^\top ) $ and anti-symmetric part as $\GRAD^n \bu := \frac{1}{2} (\GRAD \bu - (\GRAD \bu)^\top) $

% \begin{equation} \label{eq:anti sym cross}
% (\GRAD^n \bu) \bv  = \frac{1}{2} ( \GRAD \CROSS \bu ) \CROSS \bv
% \end{equation}

% given a function $a(\bx)$ we have that

% \begin{equation} \label{eq:cross_norm}
% \| a \GRAD \CROSS \bu \|^2 = 2 \| a  \GRAD^n \bu \|^2
% \end{equation}
% We don't need the thing below anymore...

%Let $\bn$ be the outward-pointing normal vector on the boundary $\p \Omega$. \lukas{Provided that $\bu \SCAL \bn |_{ \p \Omega} = 0$} the following relations hold due to integration by parts 

The following identities hold: 
\begin{align}
 &b(\bu , \bbm, \bv)  =b(\bbm, \bu, \bv) + (\bu \SCAL \GRAD \rho ,  \bu \SCAL \bv ), \label{eq:mom_ident_1} \\
 &b(\bu, \bu, \bbm) = b(\bbm, \bu, \bu),  \label{eq:move_rho}
\end{align}
where \eqref{eq:mom_ident_1} follows from $\bu \SCAL \GRAD ( \rho \bu)  = \rho \bu \SCAL \GRAD \bu + (\bu \SCAL \GRAD \rho) \bu $ and \eqref{eq:move_rho} follows from the definition of the trillinear form \eqref{eq:trilinear definition}. 

Since the domain is periodic the following relations hold due to integration by parts \noeqref{eq:advection_IBP1}
\begin{align}
\label{eq:IBP1} &b ( \bu , \bv, \bw ) = - b( \bu, \bw , \bv ) - ( (\DIV \bu) \bv , \bw), \\
\label{eq:IBP2} &b ( \bu , \bw, \bw ) =  - \frac{1}{2} ( (\DIV \bu) \bw , \bw), \\
\label{eq:advection_IBP1} & ( \bu \SCAL \GRAD \rho , w ) = -( \bu \SCAL \GRAD w, \rho ) - ( ( \DIV \bu ) \rho , w) , \\
\label{eq:advection_IBP2} & ( \bu \SCAL \GRAD \rho , \rho ) = - \frac{1}{2} ( ( \DIV \bu ) \rho , \rho ) .
\end{align}

\commentA{We also define $\be_i$ to be a unit vector pointing in the $i$th cordinate direction}, \ie in 3D $\be_1 = (1, 0, 0)^\top$, $\be_2 = (0, 1, 0)^\top$ and $\be_3 = (0, 0, 1)^\top$. We define $\bphi_i \coloneqq  \bx \CROSS \be_i $ and note that $\bphi_i$ has the property that \commentA{$\DIV \bphi_i = 0$} and $ \GRAD \bphi_i + (\GRAD \bphi_i)^\top = 0 $. To get a correct definition of the cross product in 2D, the last component of all the vectors is extended by 0. \commentB{The contraction operator : , is defined as
\begin{equation}
A : B = \sum_{i,j = 1}^d A_{ij} B_{ij},
\end{equation}
for any two $d \times d$ matrices $A$ and $B$.}

\subsubsection{Finite element preliminaries.} \label{Sec:fem_approx}

% \sout{\lukas{The global shape functions $\{ \varphi_i \}_{i=1}^{N}$ form a basis for the space $\calM$ , where $N$ is the total number of nodes in $\calM$.}}

We denote a shape regular computational mesh by $\mathcal{T}_h$ which is a triangulation of $\Omega$ into a finite number of disjoint elements $K$. \commentA{The finite element spaces we use for the density, velocity and pressure are respectively given by}
\begin{equation}
\begin{split}
 % &{\calM} : = \{ w(\bx) : w \in \Bar{\calM}, w \in \mathcal{C}^0( \Omega )  \}, \\ 
 &\calM : = \{ w(\bx) : w \in \mathcal{C}^0(\Omega), w|_K \in \polP_{k_\rho}, \forall K \in \mathcal{T}_h  \}, \\ 
&{\bcalV} := \Big[ \{ v(\bx) : v \in \mathcal{C}^0( \Omega ), v|_K \in \polP_{k_\bu}, \forall K \in \mathcal{T}_h  \} \Big]^d, \\
&\calQ : = \l \{ q(\bx) : q \in \mathcal{C}^0( \Omega ), q|_K \in \polP_{k_P}, \forall K \in \mathcal{T}_h, \int_\Omega q \ud \bx = 0 \r \}, 
\end{split}
\end{equation}
where $\polP_{k_\rho},\polP_{k_\bu},\polP_{k_P}$ are the set of multivariate polynomials of total degree at most $k_\rho,k_\bu,k_P \geq 1$ defined over $K$. It is well-known that to satisfy the so-called inf-sup condition \citep{Girault_1986} we require $k_P < k_\bu$. The corresponding discrete inner products are defined as
\begin{equation} \label{eq:inner product}
\begin{split}
&(v, w) := \sum_{K \in \mathcal{T}_h} \int_K v  w  \ud \bx, \quad (\bv, \bw) := \sum_{K \in \mathcal{T}_h} \int_K \bv \cdot \bw  \ud \bx, \quad (\GRAD \bv, \GRAD \bw) := \sum_{K \in \mathcal{T}_h} \int_K \GRAD \bv : \GRAD \bw  \ud \bx, %\quad (v, w)_{\partial \Omega} := \sum_{K \in \mathcal{T}_h} \int_{\partial K \backslash \partial \Omega } v  w   \ud s ,%
\end{split}
\end{equation}
with associated $L^2$ norms $\| \cdot \|$. Discrete and continuous inner products are used interchangeably for brevity. The discrete mesh size function $h(\bx)$ is constructed as the following continuous piecewise linear function: for every nodal point $N_i$
\begin{equation}
h(N_i) := \min \l( h_K / \max(\commentA{k_\bu}, k_\rho), \quad \forall K \in \mathcal{T}_h \text{ such that } N_i \in K \r)
\end{equation}
where $h_K$ is the smallest edge of $K$.

% \commentA{To facilitate the analysis, we define the following finite element space}
% \begin{equation} \label{eq:dg_space_definition}
% % \commentA{ \Bar{\calM}_s : = \{ \overline{w}(\bx) : \overline{w} \in L_2(\Omega), \overline{w}|_K \in \polP_{s}, \forall K \in \mathcal{T}_h  \}, 
% }\end{equation}
% \commentA{where we note that $\calM \subset \Bar{\calM}_s$ provided that $k_\rho \leq s$.
% }
\subsection{Properties of the governing equations.} \label{Sec:div_u_properties}
The governing equations can be shown to conserve mass, squared density, kinetic energy, momentum and angular momentum. These are defined as

\begin{equation} \label{eq:conservation_properties}
\begin{alignedat}{6}
& \text{Mass}  && \int_{\Omega} \rho \ud \bx; \quad && \text{Kinetic energy} && \frac{1}{2} \int_{\Omega} \rho \bu \SCAL \bu \ud \bx;   \quad && \text{Angular momentum}  && \int_{\Omega}  \bbm \CROSS \bx \ud \bx;  \\ &  \text{Momentum} && \int_{\Omega}  \bbm \ud \bx;   \quad 
&&  \text{Squared density} \quad && \frac{1}{2} \int_{\Omega}  \rho^2  \ud \bx.  
\end{alignedat}
\end{equation}

Another property that is rarely (if at all) mentioned in the variable density flow literature is that the mass equation is invariant to shifts in the density field. \commentE{To be more precise, we say that the advection equation is shift-invariant since}
 \begin{equation} \label{eq:shift_invariance}
  \commentE{\p_t \rho + \bu \SCAL \GRAD \rho   = 0 \quad \Leftrightarrow \quad \p_t ( \rho + c) + \bu \SCAL \GRAD (  \rho + c )  = 0, \quad \forall c \in \mR \setminus \{  0 \}.}
 \end{equation}
 \commentE{On the other hand, the advection equation in conservative form is not shift-invariant since}
 \begin{equation} \label{eq:shift_invariance}
  \commentE{ \p_t \rho + \DIV ( \rho \bu  )   = 0 \quad \nLeftrightarrow \quad \p_t ( \rho + c) +  \DIV ( ( \rho + c) \bu  )  = 0, \quad \forall c \in \mR \setminus \{  0 \}.}
 \end{equation}

\section{Properties when  \texorpdfstring{$\DIV \bu \neq 0. $} s} \label{Sec:div_neq_properties}
For most numerical methods, $\DIV \bu = 0$ does not hold pointwise at the discrete level. This is also true for most Galerkin methods where the numerical solution usually satisfies
\begin{equation}
(\DIV \bu, q) = 0, \quad \forall q \in \commentB{\calQ}.
\end{equation}
We note that there are some exceptions to this. In particular, some exotic elements are pointwise divergence-free such as iso-geometric B-splines \citep{Evans2013}, Raviart-Thomas elements \citep{Raviart_Thomas, Gawlik2020}, Scott-Vogelius elements \citep{Scott1985} and so on, but some of these are not included in all FEM software libraries and some introduce constraints on the computational mesh and polynomial degree. Additionally, not all divergence-free elements, such as \citep{Raviart_Thomas, Gawlik2020}, are $H^1$-conforming which is a beneficial property to have.

In this work, we analyze many variations of the nonlinear terms. Depending on how these are chosen, favorable conservation properties can be obtained which we investigate thoroughly in this work. To this end, we write the governing equations \eqref{eq:cons_law_primitive} in a more general form
\begin{equation}\label{eq:generic_nse}
\begin{aligned}
        \p_t \rho + \bu \SCAL \GRAD \rho + \alpha_\rho (\DIV \bu) \rho  & = 0,\\
 \p_t \bbm + \bu \SCAL \GRAD  \bbm + \alpha_\bbm (\DIV \bu) \bbm + \GRAD P +  \alpha_P (\GRAD \bu) \bbm + \alpha_P (\GRAD \bbm) \bu  &= 0,  \\
\end{aligned}
\end{equation}
where $P = p - \alpha_P \bu \SCAL \bbm$ is a modified pressure and $\alpha_\rho, \alpha_\bbm, \alpha_P  \in \mR$. The terms involving $\DIV \bu$ are consistent with the governing equations and have previously been used for variable density flow in different variations to derive numerical methods, see \eg \citep{Guermond2000,Shen2007,Bartholomew_2019,almgren98,Zhang_2023}. The modified pressure technique ($\alpha_P \neq 0$) has been successfully applied for constant density flow \citep{Charnyi2017} but has currently not been applied in the variable density context. Depending on how $\alpha_\rho, \alpha_\bbm, \alpha_P$ are chosen different conservation properties are obtained. By using that $ (\GRAD \bbm) \bu = (\GRAD (\rho \bu)) \bu  = ( \GRAD \rho \otimes \bu ) \bu +  (\GRAD \bu) \bu \rho   =   \GRAD \rho  (\bu \SCAL \bu) +  (\GRAD \bu) \bbm  $, the system \eqref{eq:generic_nse} can equivalently be expressed as
\begin{equation}\label{eq:generic_nse_tricks}
\begin{aligned}
        \p_t \rho + \bu \SCAL \GRAD \rho + \alpha_\rho (\DIV \bu) \rho  & = 0,\\
 \p_t \bbm + \bu \SCAL \GRAD  \bbm + \alpha_\bbm (\DIV \bu) \bbm  + \GRAD P  +  \l( \alpha_P  + \frac{1}{2} \r) (\GRAD \bu) \bbm \\ + \l(  \alpha_P  - \frac{1}{2} \r) (\GRAD \bbm) \bu + \frac{1}{2} ( (\bu \SCAL \bu)  \GRAD \rho    - \alpha_\rho \GRAD (  \rho   \bu \SCAL \bu )  ) 
 + \frac{1}{2}     \alpha_\rho \GRAD (  \rho \bu \SCAL \bu )    &= 0,  \\
\end{aligned}
\end{equation}
Lastly, we introduce $\Bar{\rho} \in \mR $ which, if properly chosen, will make the system below shift-invariant
\begin{equation}\label{eq:generic_nse_bar}
\begin{aligned}
        \p_t \rho + \bu \SCAL \GRAD \rho + \alpha_\rho (\DIV \bu) (\rho - \Bar{\rho}) & = 0,\\
 \p_t \bbm + \bu \SCAL \GRAD  \bbm + \alpha_\bbm (\DIV \bu) \bbm  + \GRAD P  +  \l( \alpha_P  + \frac{1}{2} \r) (\GRAD \bu) \bbm \\ + \l(  \alpha_P  - \frac{1}{2} \r) (\GRAD \bbm) \bu + \frac{1}{2} (   \GRAD \rho  (\bu \SCAL \bu)   - \alpha_\rho \GRAD (  (\rho - \Bar{\rho} ) \bu \SCAL \bu )  ) 
 + \frac{1}{2}     \alpha_\rho \GRAD (  \rho \bu \SCAL \bu )    &= 0,  \\
\end{aligned}
\end{equation}
where $P = p - \alpha_P \rho \bu \SCAL \bu - \frac{1}{2} \alpha_\rho \Bar{\rho} \bu \SCAL \bu $ is a modified pressure which now includes $\Bar{\rho}$. To the best of the authors' knowledge, using $\Bar{\rho}$ to achieve shift-invariant formulations is new to the literature for both variable density flow and scalar transport problems in general.

 % and are summarized in Table \ref{table:emac_mom}.

\commentB{
\begin{remark}
% In this work, we present strong numerical evidence that shift invariance is a property that improves robustness and accuracy for variable density flow. If a numerical method does not have this property it means that the method will artificially behave differently if density is shifted. In practice, this means that the numerical method will artificially produce larger discretization errors for the initial condition $\rho_0(\bx)+ 100$ compared to $\rho_0(\bx)$.
In this work, we provide strong numerical evidence supporting that shift invariance is a valuable property that enhances both the robustness and accuracy of variable density flow simulations. \commentE{Ideally, we would expect a numerical method solving the advection equation to produce the same error for the initial conditions $\rho_0(\bx)$ and $\rho_0(\bx) + c$. However, if a numerical method for the advection equation is not shift-invariant this is not the case, \ie the method behaves artificially differently when the initial condition is shifted. In simpler terms, such a method would produce larger discretization errors in its results for the initial condition $\rho_0(\bx) + 100$ compared to $\rho_0(\bx)$.}
\end{remark}}

\subsection{Galerkin finite element approximation.}
The finite element approximation of \eqref{eq:generic_nse_bar} is derived by testing the mass equation with $w$, the momentum equations with $\bv$ and the divergence-free constraint with $q$. The finite element method reads: Find $(\rho, \bu, P) \in \commentE{ \calC^1( [0,T]; } \ \calM \CROSS \mathbf{\bcalV} \CROSS \calQ \commentE{)} $ such that
\begin{equation}\label{eq:mom_update_general}
\begin{alignedat}{2} 
( \new{\p_t} \rho , w) + (\bu \SCAL \GRAD \rho, w) + \alpha_\rho ((\DIV \bu) ( \rho - \Bar{\rho} ) , w) & = 0, \quad && \forall w \in \calM, \\
( \new{\p_t} \bbm , \bv)  + b( \bu, \bbm, \bv ) + \alpha_\bbm ( ( \DIV \bu ) \bbm, \bv) + (\GRAD P, \bv) \\   + \l( \alpha_P - \frac{1}{2}  \r) b(\bv,\bbm,\bu) +  \l( \alpha_P + \frac{1}{2}  \r) b(\bv,\bu,\bbm) + \frac{1}{2} \alpha_\rho (  \GRAD ( \bbm \SCAL \bu )  , \bv )  \\ + \frac{1}{2}  \Big( \l( \bv \SCAL \GRAD \rho, \commentC{ \alpha_{\uu} \uuDG + (1 - \alpha_{\uu}) \uu } \r)  \\  -  \alpha_\rho \l(  \GRAD \l(  (\rho - \Bar{\rho})  \commentC{ \l( \alpha_{\uu} \uuDG + (1 - \alpha_{\uu}) \uu \r) } \r) , \bv \r) \Big)  &= 0, \quad && \forall \bv \in  \bcalV, \\
(\DIV \bu, q) &= 0, \quad && \forall q \in \calQ,
\end{alignedat}
\end{equation}
\commentC{where $\alpha_{\uu} \in \{0 , 1\}$ and $\uuDG \in \calM$ is the projection of $\bu \SCAL \bu$ onto $\calM$, see \eg \citep[Remark 3.1]{Gawlik2020}: Find $ \overline{\uu} \in \calM$ such that
\begin{equation} \label{eq:uu_projection}
 (  \overline{ \bu \SCAL \bu } , w ) = ( \bu \SCAL \bu, w ), \quad \forall w \in \calM.
\end{equation}}
% Note that $\alpha_{\uu}=0$ corresponds to a standard Galerkin discretization of \eqref{eq:generic_nse_bar}, whereas $\alpha_{\uu}=1$ corresponds to a modified Galerkin method as proposed by \cite{Gawlik2020}.}

%
% \begin{remark}
\commentC{The reason why $\uuDG$ is used is that we later need to set $w = \bu \SCAL \bu$ to obtain a kinetic energy estimate. This means that we either need to project $\bu \SCAL \bu$ to $\calM$ \eqref{eq:uu_projection} or we require $k_\rho \geq 2 k_\bu$. We note that $\alpha_{\uu} = 0$ corresponds to a standard Galerkin discretization of \eqref{eq:generic_nse_bar}, whereas $\alpha_{\uu} = 1$ corresponds to a modified Galerkin method as proposed by \citep[Remark 3.1]{Gawlik2020}.}
% \end{remark}

% We require additional constraints to achieve all properties in Table \ref{table:emac_mom}.
\subsection{Semi-discrete properties of the Galerkin method.} \label{Sec:semi-discrete properties}
\neww{
We summarize the properties of the semi-discrete method \eqref{eq:mom_update_general} in Table \ref{table:emac_mom} and Theorem \ref{theorem div}, which are the main results of this work. In the constant density case, a unique choice of parameters allows for the conservation of kinetic energy, momentum and angular momentum, known as the EMAC formulation \citep{Charnyi2017}. However, the situation becomes more complex in the variable density context. Ignoring $\alpha_{\uu}$, all properties are obtained by setting $\alpha_\rho = \frac{1}{2}, \alpha_\bbm = 1, \alpha_P = 0.25, \Bar{\rho} = \frac{1}{| \Omega |} \int_\Omega \rho\ud \bx$ and choosing the polynomial degree of density to be less than or equal to that of pressure ($k_\rho \leq k_P$). Considering different parameter choices for $\alpha_\rho$, $\alpha_\bbm$, $\alpha_P$, $\Bar{\rho}$, and $k_\rho$ leads to both common and several new formulations from the variable density flow literature, as detailed in Table \ref{table:formulations_of_interest}.}

% We summarize the properties of the semi-discrete method \eqref{eq:mom_update_general} in Table \ref{table:emac_mom} and Theorem \ref{theorem div}, which is the main result of this work. In the constant density case, a unique parameter choice allows for the conservation of kinetic energy, momentum, and angular momentum, known as the EMAC formulation \citep{Charnyi2017}. However, the situation is more complex in the variable density context. Ignoring $\alpha_{\uu}$, all properties are obtained by setting $\alpha_\rho = \frac{1}{2}, \alpha_\bbm = 1, \alpha_P = 0.25, \Bar{\rho} = \frac{1}{| \Omega |} \int_\Omega \rho\ud \bx$ and choosing the polynomial degree of density to be less than or equal to that of pressure ($k_\rho \leq k_P$). Considering other parameter choices $\alpha_\rho, \alpha_\bbm, \alpha_P, \Bar{\rho}, k_\rho$ results in common and several new formulations from the variable density flow literature, as detailed in Table \ref{table:formulations_of_interest}. 

\neww{In Table \ref{table:formulations_of_interest}, the SI-MEDMAC formulation, when ignoring $\alpha_{\uu}$, is described as shift-invariant and mass, kinetic energy, squared density, momentum and angular momentum conserving. We follow a similar naming convention for the other formulations in the table; for example, MEMAC conserves mass, kinetic energy, momentum, and angular momentum, whereas EDMAC conserves kinetic energy, squared density, momentum, and angular momentum. All formulations with $\alpha_\rho=0$ are termed 'locally' shift-invariant (LSI), meaning they remain invariant to domain stretching, unlike the other formulations. The convective form is obtained when \eqref{eq:cons_law_primitive} is discretized directly without any modifications.  We note a limitation in all formulations listed in Table \ref{table:formulations_of_interest}: the influence of $\alpha_{\uu}$. Specifically, setting $\alpha_{\uu}=1$ allows the Galerkin method to conserve kinetic energy \citep{Gawlik2020} but not momentum and angular momentum. Conversely, with $\alpha_{\uu}=0$, the method can conserve momentum and angular momentum but not kinetic energy. The proof of Theorem \ref{theorem div} highlights that the leading conservation error term for SI-MEDMAC is independent of the divergence-free condition, marking a significant improvement over the other formulations in the table.}

\renewcommand{\arraystretch}{1.3}

\begin{table}[H]
\centering     
\caption{Semi-discrete properties of the method \eqref{eq:mom_update_general} when $\bef = 0$ for the inviscid case when $\DIV \bu \neq 0$.}
\begin{tabular}{|c|c|}    
\hline Property & Condition   \\  \hline 
  $ \partial_t  \int_\Omega \rho \bu \SCAL \bu  \ud\bx$ = 0 & \commentC{$\alpha_{\uu} = 1$ and} $\alpha_\bbm -\alpha_P - \alpha_\rho/2 = 1/2$\\[0.1cm]
    $ \partial_t \int_\Omega \bbm  \ud\bx$ = 0 & \commentC{$\alpha_{\uu} = 0$ and} $\alpha_\bbm = 1$ \\  
     $ \partial_t \int_\Omega \bbm \CROSS \bx   \ud\bx = 0$ & \commentC{$\alpha_{\uu} = 0$ and} $\alpha_\bbm = 1$ \\ 
      $ \partial_t \frac{1}{2} \int_\Omega   \rho^2\ud \bx  -  \int_\Omega  \Bar{\rho}\partial_t \rho   \ud\bx = 0 $ & $\alpha_\rho = 1/2$ \\  
      $ \partial_t  \l(  \int_\Omega   \rho^2\ud \bx  - \frac{1}{ | \Omega | } \l(  \int_\Omega \rho  \ud\bx  \r)^2 \r)   = 0 $ & $\alpha_\rho = 1/2$ and $\Bar{\rho} = \frac{1}{| \Omega |} \int_\Omega \rho\ud \bx $ \\  
      \new{$ \partial_t    \int_\Omega   \rho^2\ud \bx = 0 $ } & \new{ \big[ $\alpha_\rho = 1/2$ and $k_\rho \leq k_P$ \big] or $2 k_\rho \leq k_P$  } \\  
     $\partial_t \int_\Omega \rho   \ud\bx$ = 0 & \new{$k_\rho \leq k_P$ or $\alpha_\rho = 1$}   \\ 
     Shift invariance &  $\alpha_\rho = 0$ or  $\Bar{\rho} = \frac{1}{| \Omega |} \int_\Omega \rho\ud \bx $   \\ \hline 
\end{tabular}                                                                   
\label{table:emac_mom}                                                      
\end{table}

\begin{table}[H]
\centering     
\caption{Existing formulations in the literature and newly derived formulations for variable density flow. The formulations are later tested and compared in numerical experiments.}
\begin{tabular}{|c|c|c|c|c|c|}    
\hline Name & $\alpha_\rho$ & $\alpha_\bbm$ & $\alpha_P$  &$ \bar{\rho}$ & $k_\rho$  \\ \hline
 LSI-EMAC & 0 & 1 & 0.5 & N/A & $\forall k_\rho$ \\
 MEMAC \citep{Zhang_2023,Manzanero_2020} & 1 & 1 & 0 & 0 & $\forall k_\rho$\\
 EDMAC & 0.5 & 1 & 0.25 & 0 & $\forall k_\rho$\\
 SI-MEMAC & 1 & 1 & 0 & $\Bar{\rho} = \frac{1}{| \Omega |} \int_\Omega \rho\ud \bx$ & $\forall k_\rho$ \\ 
 SI-EDMAC & 0.5 & 1 & 0.25 & $\Bar{\rho} = \frac{1}{| \Omega |} \int_\Omega \rho\ud \bx$ & $\forall k_\rho$ \\
 SI-MEDMAC & 0.5 & 1 & 0.25 & $\Bar{\rho} = \frac{1}{| \Omega |} \int_\Omega \rho\ud \bx$ & $k_\rho \leq k_P$ \\
 LSI-EC \citep{Lundgren_2023} & 0 & 0.5 & 0 & N/A & $\forall k_\rho$ \\ 
 Convective \citep[Sec 4]{Guermond_Salgado_2009}  & 0 & 0 & 0 & N/A & $\forall k_\rho$ \\ \hline
\end{tabular}                                                                   
\label{table:formulations_of_interest}                                                      
\end{table}

% \begin{remark}
% All formulations in Table \ref{table:formulations_of_interest} with $\alpha_\rho=0$ are, for a lack of a better term, named \textit{locally} shift-invariant. If the domain is stretched, these formulations are invariant to this change but the other formulations are not.% We were unable to find a formulation that is locally shift-invariant and also mass and/or squared density conservative. We considered letting $\bar{\rho} \in \bar{\calM}_0$, where $\bar{\calM}_0 = \l \{ w \in L^2 (\Omega) : w|_K \in \polP_0,   \forall K \in \mathcal{T}_h   \r   \}$, and setting $\bar{\rho}|_{K_i} = \frac{1}{| K_i |}  \int_{K_i} \rho\ud \bx$. This makes the formulation locally shift-invariant, but unfortunately, since $\bar{\rho}$ is a discontinuous function, integration by parts produces edge terms which makes the method lose conservation. In the authors' opinions, finding a locally shift-invariant formulation of the density update which is also conservative would be an interesting subject of further research.
% \end{remark}

\begin{theorem} \label{theorem div}
The properties of Table \ref{table:emac_mom} hold provided that the conditions inside it are met.
\end{theorem}

\begin{proof}

We divide the proof into several parts. To simplify the proof we assume a periodic domain and set $\bef = 0$. Extending the proofs to other boundary conditions such as no-slip ($\bu = 0|_{\p \Omega}$) can be done by using the proof technique in \citep[Sec 3.1.2]{Charnyi2017} or \citep[Sec 3.1.2]{Ingimarson2023}.
\noindent \\
\textbf{Mass:}
By using the definition \commentE{of mass} we have that \eqref{eq:mom_update_general} satisfies
\begin{equation} \label{eq:mass_bar1}
\p_t \int_\Omega \rho\ud \bx =  \int_\Omega \p_t \rho \ud \bx  = - ( \bu , \GRAD \rho ) - \alpha_\rho ( \DIV \bu  ,  \rho - \Bar{\rho}  ) .
\end{equation}

Since $\GRAD \Bar{\rho} = 0$, \eqref{eq:mass_bar1} is equivalent to
\begin{equation}
\begin{split}
\p_t \int_\Omega \rho\ud \bx = & - ( \bu , \GRAD ( \rho - \Bar{\rho} )  ) - \alpha_\rho (  \DIV \bu  ,  \rho - \Bar{\rho} )  \\
= \  &  ( \DIV \bu ,  \rho - \Bar{\rho}   )   - \alpha_\rho (  \DIV \bu ,  \rho - \Bar{\rho} ) \\
= \ & \ ( 1 -\alpha_\rho  ) ( \DIV \bu,  \rho - \Bar{\rho} )  ,\\
\end{split}
\end{equation}
where integration by parts was used on the first term. Therefore mass is conserved if $\alpha_\rho = 1$. If $\Bar{\rho} = \int_\Omega \rho \ud \bx$ then
\begin{equation}
\int_\Omega \new{ (} \rho - \Bar{\rho}  \new{  )} \ud \bx = 0,
\end{equation}
which means that $(\rho - \Bar{\rho} ) \in \calQ$ provided that $k_\rho \leq k_P$. Invoking the weak divergence-free condition \eqref{eq:mom_update_general} shows that
\begin{equation}
( \DIV \bu, \rho - \Bar{\rho} ) = 0,
\end{equation}
which means that mass is conserved if $k_\rho \leq k_P$.

% \new{Next, since $\GRAD A = 0$, where $A =   \frac{ \int_\Omega \rho \ud \bx -  \alpha_\rho \l( \int \rho \ud \bx - |\Omega| \bar{ \rho } \r) }{ | \Omega | } \in \mR $, \eqref{eq:mass_bar1} is equivalent to
% \begin{equation}
% \begin{split}
% \p_t \int_\Omega \rho\ud \bx = & - ( \bu , \GRAD ( \rho -  A )  ) - \alpha_\rho (  \DIV \bu  ,  \rho - \Bar{\rho} )  \\
% = \ & \ ( \DIV \bu, \rho -  A - \alpha_\rho ( \rho - \bar{ \rho }  ) )  .
% \end{split}
% \end{equation}
% Since the average of $ \rho - A - \alpha_\rho ( \rho - \bar{ \rho }  )$ is zero, \ie 
% \begin{equation}
% \begin{split}
%  \int_\Omega  \rho - A - \alpha_\rho ( \rho - \bar{ \rho }  ) \ud \bx = &  \int_\Omega \rho -   \frac{ \int_\Omega \rho \ud \bx -  \alpha_\rho \l( \int \rho \ud \bx - |\Omega| \bar{ \rho } \r) }{ | \Omega | } - \alpha_\rho ( \rho - \bar{ \rho } ) \ud \bx 
%  \\ = & \int_\Omega \rho \ud \bx -\int_\Omega \rho \ud \bx + \alpha_\rho \l( \int_\Omega \rho \ud \bx  - | \Omega | \Bar{ \rho} \r) - \alpha_\rho  \int_\Omega ( \rho - \bar{ \rho } ) \ud \bx = 0 ,
% \end{split}
% \end{equation}
% this means that $ \l( \rho - A - \alpha_\rho ( \rho - \bar{ \rho }  ) \r) \in \calQ$ provided that $k_\rho \leq k_P$. Invoking the weak divergence-free constraint \eqref{eq:mom_update_general} shows that
% \begin{equation}
% ( \DIV \bu  , \rho - A - \alpha_\rho ( \rho - \bar{ \rho }  )  ) = 0, 
% \end{equation}
% which means that mass is conserved if $k_\rho \leq k_P$.}

% Definition \ref{definition shift invariance}
\noindent \\
\textbf{Shift-invariance:}
If $\alpha_\rho = 0$ then the proof is evident from \eqref{eq:shift_invariance}. Otherwise, if $\Bar{\rho} = \frac{1}{| \Omega |} \int_\Omega \rho\ud \bx$, the following is obtained by shifting $\rho$ by a constant $c$ in \eqref{eq:generic_nse_bar}
\begin{equation}
\begin{split}
\p_t ( \rho + c) + \bu \SCAL \GRAD (\rho + c) + \alpha_\rho ( \DIV \bu ) \l( \rho + c - \frac{1}{|\Omega |}   \int_\Omega (\rho +  c ) \ud \bx \r) &= 0, \\ 
\p_t \rho + \bu \SCAL \GRAD \rho  + \alpha_\rho ( \DIV \bu ) \l( \rho + c - \frac{1}{|\Omega |}   \int_\Omega ( \rho + c ) \ud \bx \r) &= 0, \\ 
% \p_t \rho + \bu \SCAL \GRAD \rho  + \alpha_\rho ( \DIV \bu ) \l( \rho + c -  \frac{|\Omega |}{|\Omega |} c - \frac{1}{|\Omega |}   \int_\Omega \rho \ud \bx \r) &= 0 \\ 
\p_t \rho + \bu \SCAL \GRAD \rho  + \alpha_\rho ( \DIV \bu ) \l( \rho  - \frac{1}{|\Omega |}   \int_\Omega \rho \ud \bx \r) &= 0, \\ 
\p_t \rho + \bu \SCAL \GRAD \rho  + \alpha_\rho ( \DIV \bu ) \l( \rho  - \Bar{\rho} \r) &= 0, \\ 
\end{split}
\end{equation}
which shows that the formulation is shift-invariant.
\noindent \\
\textbf{Squared density:}
By using that $\GRAD \Bar{\rho} = 0$, one can show that \eqref{eq:mom_update_general} satisfies
\begin{equation} \label{eq:rho_energy1}
\partial_t \frac{1}{2} \int_\Omega   \rho^2\ud \bx  -  \int_\Omega  \Bar{\rho}\partial_t \rho   \ud\bx = ( \new{\p_t} \rho , \rho - \Bar{\rho}) =  - ( \bu \SCAL \GRAD ( \rho - \Bar{\rho} ), \rho - \Bar{\rho} ) -  \alpha_\rho ( (\DIV \bu) (\rho - \Bar{\rho} ) , \rho - \Bar{\rho} )  .
\end{equation}
Using \eqref{eq:advection_IBP2} inside \eqref{eq:rho_energy1} yields

\begin{equation} \label{eq:sq_proof1}
\partial_t \frac{1}{2} \int_\Omega   \rho^2\ud \bx  -  \int_\Omega  \Bar{\rho}\partial_t \rho   \ud\bx = \l(  \frac{1}{2} - \alpha_\rho \r) ( (\DIV \bu) (\rho - \Bar{\rho} ) , \rho - \Bar{\rho} )  .
\end{equation}
If $\alpha_\rho =  \frac{1}{2} $ then $\partial_t \frac{1}{2} \int_\Omega   \rho^2\ud \bx  -  \int_\Omega  \Bar{\rho}\partial_t \rho   \ud\bx = 0$. This shows that line 4 in Table \ref{table:emac_mom} is true.

If $\Bar{\rho} =  \frac{1}{| \Omega |} \int_\Omega \rho\ud \bx$, then we can perform the following simplification
\begin{equation}
\begin{split} \label{eq:sq_proof2}
\partial_t \frac{1}{2} \int_\Omega   \rho^2\ud \bx  -  \int_\Omega  \Bar{\rho}\partial_t \rho   \ud\bx &= \partial_t \frac{1}{2} \int_\Omega   \rho^2\ud \bx  -  \Bar{\rho} \int_\Omega  \partial_t \rho   \ud\bx \\ 
&= \partial_t \frac{1}{2}\int_\Omega   \rho^2\ud \bx  - \frac{1}{| \Omega |} \l(  \int_\Omega \rho\ud \bx \r)  \int_\Omega   \partial_t \rho   \ud\bx  \\
&= \partial_t \frac{1}{2} \int_\Omega   \rho^2\ud \bx  -  \frac{1}{| \Omega |} \l(  \int_\Omega \rho\ud \bx \r)\partial_t \l( \int_\Omega    \rho   \ud\bx \r)  \\
&= \partial_t \frac{1}{2} \l(  \int_\Omega   \rho^2\ud \bx  -  \frac{1}{| \Omega |} \l(  \int_\Omega \rho\ud \bx \r)^2 \r),  
\end{split}
\end{equation}
to show that line 5 in Table \ref{table:emac_mom} is true.

\new{Lastly, we consider line 6 in Table \ref{table:emac_mom}. By recalling that $ \partial_t \int_\Omega  \rho = 0$ if $k_\rho \leq k_P$, \eqref{eq:sq_proof2} simplifies to
\begin{equation}
\partial_t \frac{1}{2} \int_\Omega   \rho^2\ud \bx  -  \int_\Omega  \Bar{\rho}\partial_t \rho   \ud\bx = \partial_t \frac{1}{2} \int_\Omega   \rho^2\ud \bx,
\end{equation}
which shows that squared density is conserved if $\alpha_\rho = \frac{1}{2}$ and $k_\rho \leq k_P$.
Next, for $A \in \mR $ we also have
\begin{equation} 
\begin{split}
\partial_t \frac{1}{2} \int_\Omega   \rho^2\ud \bx  -  \int_\Omega  A \partial_t \rho   \ud\bx &= ( \new{\p_t} \rho , \rho - A) =  - ( \bu \SCAL \GRAD ( \rho - A ), \rho - A ) -  \alpha_\rho ( (\DIV \bu) (\rho - \Bar{\rho} ) , \rho - A )   \\
&=  \l( \DIV \bu , \frac{1}{2}  ( \rho - A )^2 -  \alpha_\rho  (\rho - \Bar{\rho} ) ( \rho - A )  \r).   \\
\end{split}
\end{equation}
There exist two roots of $A$ that satisfies
\begin{equation} \label{eq:average condition}
\int_\Omega \l( \frac{1}{2}  ( \rho - A )^2 -  \alpha_\rho  (\rho - \Bar{\rho} ) ( \rho - A ) \r) \ud \bx = 0.
\end{equation}
We only need one $A$ to satisfy \eqref{eq:average condition} and one of the roots (found symbolically using Matlab) are
\begin{equation}
\begin{split}
A =   |\Omega|^{-1} \l(  \alpha_\rho|\Omega|\Bar{\rho} + (1 - \alpha_\rho) \int_\Omega \rho \ud \bx + \l(\alpha_\rho^2|\Omega|^2\Bar{\rho}^2 - 2\alpha_\rho^2|\Omega| \Bar{\rho} \int_\Omega \rho \ud \bx + \alpha_\rho^2 \l( \int_\Omega\rho \ud \bx \r)^2 + 2 \alpha_\rho|\Omega| \int_\Omega \rho^2 \ud \bx \r. \r. \\ \l. \l. - 2 \alpha_\rho \l( \int_\Omega \rho \ud \bx \r)^2  - |\Omega| \int_\Omega \rho^2 \ud \bx + \l(\int_\Omega \rho \ud \bx \r)^2\r)^{1/2}   \r).
\end{split}
\end{equation}
This means that, if $2 k_\rho \leq k_P $ then there exists an $A$ such that $ \l( \frac{1}{2}  ( \rho - A )^2 -  \alpha_\rho  (\rho - \Bar{\rho} ) ( \rho - A ) \r) \in  \calQ$. Because of this, we can invoke the weak divergence-free condition \eqref{eq:mom_update_general} to conclude that squared density is conserved.}

% Since $\rho_t \in \calM  \subset \bar{\calM}_{k_\bu + k_\rho - 1}$, we can use \eqref{eq:projection_special} also infer that 

\noindent \\
\textbf{Kinetic energy:}
% \commentA{We define the following $L^2$-projection. Find $ \uuDG \in \calM  $ such that}
% \begin{equation} \label{eq:projection_special}
% % \commentA{ \l( \uuDG  , \overline{w} \r) = \l( \bu \SCAL \bu , \overline{w} \r) \quad \forall \overline{w} \in \bar{\calM}_{k_\bu + k_\rho - 1}},
% \end{equation}
% \commentA{where $\bar{\calM}_{k_\bu + k_\rho - 1}$ is defined by \eqref{eq:dg_space_definition}.} Next, , the definition of the $L^2$-projection \eqref{eq:projection_special} and setting $w = \uuDG, \bv = \bu$ 
\commentC{Using the definition of kinetic energy, the definition of the $L^2$-projection \eqref{eq:uu_projection} and setting $w= \overline{\uu}$, $\bv = \bu$ inside the density and momentum updates \eqref{eq:mom_update_general} yields }
\begin{equation}
\begin{split}
\frac{1}{2} \p_t \int_\Omega \rho \bu \SCAL \bu\ud \bx  = & \   \frac{1}{2} ( \bu \new{\p_t \rho} , \bu) +  ( \rho \new{\p_t} \bu, \bu) = ( \bu \new{\p_t} \rho , \bu) + (\rho \new{\p_t} \bu, \bu) - \frac{1}{2} (  \bu \new{\p_t \rho } , \bu)   \\ = & \ ( \new{\p_t} \bbm  , \bu )  - \frac{1}{2} \l( \new{\p_t} \rho , \bu \SCAL \bu \r) = ( \new{\p_t} \bbm  , \bu )  - \frac{1}{2} \l( \new{\p_t} \rho , \uuDG \r) \\
=& - b( \bu, \bbm, \bu ) - \alpha_\bbm ( ( \DIV \bu ) \bbm, \bu) - (\GRAD P, \bu) \\  
& - \l( \alpha_P - \frac{1}{2}  \r)  b(\bu,\bbm,\bu) -  \l( \alpha_P + \frac{1}{2}  \r) b(\bu,\bu,\bbm) - \frac{1}{2} \alpha_\rho (  \GRAD ( \bbm \SCAL \bu )  , \bu )  \\ 
&- \frac{1}{2} \l( \l( \bu \SCAL \GRAD \rho, \commentC{ \alpha_{\uu} \uuDG + (1 - \alpha_{\uu}) \uu } \r) -  \alpha_\rho \l(  \GRAD \l(  (\rho - \Bar{\rho}) \commentC{ \l( \alpha_{\uu} \uuDG + (1 - \alpha_{\uu}) \uu \r) } \r) , \bu \r) \r) \\
& + \frac{1}{2} \l( \bu \SCAL \GRAD \rho  + \alpha_\rho  (\DIV \bu) ( \rho - \Bar{\rho} ) , \uuDG  \r).
\end{split}
\end{equation}

Integration by parts on the pressure term and invoking the weak divergence-free constraint inside \eqref{eq:mom_update_general}, using \eqref{eq:IBP1} on $\l( \alpha_P - \frac{1}{2}  \r) b(\bu,\bbm,\bu)$ and integration by parts on $ \commentC{ \alpha_\rho \l(  \GRAD \l(  (\rho - \Bar{\rho})  \l( \alpha_{\uu} \uuDG + (1 - \alpha_{\uu}) \uu \r) \r), \bu \r) }  $ yields
\begin{equation}
\begin{split}
\frac{1}{2} \p_t \int_\Omega \rho \bu \SCAL \bu\ud \bx  =& - b( \bu, \bbm, \bu ) + \l(   \alpha_P - \frac{1}{2}   - \alpha_\bbm \r) ( ( \DIV \bu ) \bbm, \bu) \\  
& + \l( \alpha_P - \frac{1}{2}  \r)  b(\bu,\bu,\bbm)   -  \l( \alpha_P + \frac{1}{2}  \r) b(\bu,\bu,\bbm) - \frac{1}{2} \alpha_\rho (  \GRAD ( \bbm \SCAL \bu )  , \bu )  \\ 
&- \frac{1}{2} \l( \l( \bu \SCAL \GRAD \rho, \commentC{ \alpha_{\uu} \uuDG + (1 - \alpha_{\uu}) \uu } \r) +  \alpha_\rho \l(  \DIV \bu ,   (\rho - \Bar{\rho}) \commentC{ \l( \alpha_{\uu} \uuDG + (1 - \alpha_{\uu}) \uu \r) }   \r) \r) \\ 
& + \frac{1}{2} \l( \bu \SCAL \GRAD \rho  + \alpha_\rho  (\DIV \bu) ( \rho - \Bar{\rho} ) , \uuDG  \r) \\ \\
=& - b( \bu, \bbm, \bu ) + \l(   \alpha_P - \frac{1}{2}   - \alpha_\bbm \r) ( ( \DIV \bu ) \bbm, \bu) \\  
&  -  b(\bu,\bu,\bbm) - \frac{1}{2} \alpha_\rho (  \GRAD ( \bbm \SCAL \bu )  , \bu )  \\
 &\commentA{- \frac{1}{2} \l( \bu \SCAL \GRAD \rho  + \alpha_\rho  (\DIV \bu) ( \rho - \Bar{\rho} ) , \commentC{ \alpha_{\uu} \uuDG + (1 - \alpha_{\uu}) \uu } \r)} \\ 
&\commentA{ + \frac{1}{2} \l( \bu \SCAL \GRAD \rho  + \alpha_\rho  (\DIV \bu) ( \rho - \Bar{\rho} ) , \uuDG  \r).}
\end{split}
\end{equation}
% \commentA{Since $ \l( \bu \SCAL \GRAD \rho  + \alpha_\rho  (\DIV \bu) ( \rho - \Bar{\rho} ) \r) \in \bar{\calM}_{k_\bu + k_\rho - 1} $, we can use \eqref{eq:projection_special} to show that}
% \begin{equation}
% \begin{split}
% % \commentA{\frac{1}{2} \p_t \int_\Omega \rho \bu \SCAL \bu\ud \bx  
% =} & \commentA{ - b( \bu, \bbm, \bu ) + \l(   \alpha_P - \frac{1}{2}   - \alpha_\bbm \r) ( ( \DIV \bu ) \bbm, \bu)} \\  
% &  \commentA{-  b(\bu,\bu,\bbm) - \frac{1}{2} \alpha_\rho (  \GRAD ( \bbm \SCAL \bu )  , \bu ).}
% \end{split}
% \end{equation}
\commentA{Next,} using \eqref{eq:IBP1} on $b( \bu, \bbm, \bu )$ and integration by parts on $\frac{1}{2} \alpha_\rho (  \GRAD ( \bbm \SCAL \bu )  , \bu )$ yields
\begin{equation}
\begin{split}
\frac{1}{2} \p_t \int_\Omega \rho \bu \SCAL \bu\ud \bx  =&   \l(   \alpha_P + \frac{1}{2}   - \alpha_\bbm + \frac{1}{2} \alpha_\rho \r) ( ( \DIV \bu ) \bbm, \bu), \\ 
&\commentA{- \frac{1}{2} \l( \bu \SCAL \GRAD \rho  + \alpha_\rho  (\DIV \bu) ( \rho - \Bar{\rho} ) , \commentC{ \alpha_{\uu} \uuDG + (1 - \alpha_{\uu}) \uu } \r)} \\ 
&\commentA{ + \frac{1}{2} \l( \bu \SCAL \GRAD \rho  + \alpha_\rho  (\DIV \bu) ( \rho - \Bar{\rho} ) , \uuDG  \r).} 
\end{split}
\end{equation}
which means that kinetic energy is conserved if $\alpha_\bbm -\alpha_P - \alpha_\rho/2 = 1/2$ \commentC{and if $\alpha_{\uu} = 1$}.
\noindent \\
\textbf{Momentum:} 
% Since $\be_i \notin \bcalV$, a strict conservation of momentum proof cannot be obtained. Instead we follow the proof technique of \cite[Sec 3.1.2]{Charnyi2017}. We define $\widehat{\Omega}$ as some strictly interior subdomain. We then define the restriction $\chi (\bg) = \bg$ of an arbitrary function $\bg$ by setting $\chi(\bg) = \bg $ in $\widehat{\Omega}$ and $\chi(\bg)$ arbitrary defined on $S = \Omega \setminus \widehat{\Omega} $ to satisfy the boundary condition $\bn \SCAL \bu = 0$. Since the solution 
% we define $\be_{ih} \in \bcalV$ by letting $\be_{ih} = \be_i$ in the interior of the domain and letting $\be_{ih}$ satisfy the boundary conditions on the final strip of the mesh along the boundary \citep[Sec 3.1.2]{Ingimarson2023}, \ie $\bn \SCAL \be_{ih} |_{\p \Omega} = 0$. Find $\be_{ih} \in \bcalV$ such that
% \begin{equation}
% \begin{split}
% (\be_{ih}, \bv) = (\be_i, \bv), \\
% \be_{ih} \SCAL \bn |_{\p \Omega} = 0,
% \end{split}
% \end{equation}
% for all $\bv \in \bcalV$. We set $\bv = \be_i$ inside the momentum equations and obtain \lukas{I will later replace $\p_t \int_\Omega \bbm_i  \ud\bx$ with $\p_t \int_\Omega \bbm_{ih}  \ud\bx = (\bbm, \be_{ih})$.}
We set $\bv = \be_i$ inside the momentum equations \eqref{eq:mom_update_general} and obtain
\begin{equation}
\begin{split}
\p_t \int_\Omega \bbm_i  \ud\bx = & \ ( \new{\p_t} \bbm , \be_{i}) =  - b( \bu, \bbm, \be_i ) - \alpha_\bbm ( ( \DIV \bu ) \bbm, \be_i) - (\GRAD P, \be_i) \\   &- \l( \alpha_P - \frac{1}{2}  \r) b(\be_i,\bbm,\bu) -  \l( \alpha_P + \frac{1}{2}  \r) b(\be_i,\bu,\bbm) +\frac{1}{2} \alpha_\rho ( (\DIV \be_i) \bbm , \bu)  \\ &- \frac{1}{2} \l( \l( \be_i \SCAL \GRAD \rho,  \commentC{ \alpha_{\uu} \uuDG + (1 - \alpha_{\uu}) \uu } \r) +  \alpha_\rho \l( ( \DIV \be_i )  (\rho - \Bar{\rho}),  \commentC{ \alpha_{\uu} \uuDG + (1 - \alpha_{\uu}) \uu }  \r) \r)  \\ \\
= &- b( \bu, \bbm, \be_i ) - \alpha_\bbm ( ( \DIV \bu ) \bbm, \be_i) - (\GRAD P, \be_i) \\ 
 & - \l( \alpha_P - \frac{1}{2}  \r) b(\be_i,\bbm,\bu) -  \l( \alpha_P + \frac{1}{2}  \r) b(\be_i,\bu,\bbm)    - \frac{1}{2} \l( \be_i \SCAL \GRAD \rho,  \commentC{ \alpha_{\uu} \uuDG + (1 - \alpha_{\uu}) \uu }  \r)  ,
\end{split}
\end{equation}
since $\be_i$ is a constant. Using \eqref{eq:IBP1} on $\l( \alpha_P - \frac{1}{2}  \r) b(\be_i,\bbm,\bu)$ and $b( \bu, \bbm, \be_i )$ and integration by parts on the pressure term yield

\begin{equation}
\begin{split}
\p_t \int_\Omega \bbm_i  \ud\bx = & \  b( \bu, \be_i, \bbm ) + ( (\DIV \bu) \bbm, \be_i) - \alpha_\bbm ( ( \DIV \bu ) \bbm, \be_i) + ( P, \DIV \be_i) \\  & + \l( \alpha_P - \frac{1}{2}  \r) b(\be_i,\bu,\bbm) -  \l( \alpha_P + \frac{1}{2}  \r) b(\be_i,\bu,\bbm)    - \frac{1}{2} \l(  \be_i \SCAL \GRAD \rho,  \commentC{ \alpha_{\uu} \uuDG + (1 - \alpha_{\uu}) \uu }  \r)   \\ 
=& \ (1 -\alpha_\bbm  ) ( ( \DIV \bu ) \bbm, \be_i)  - b(\be_i,\bu,\bbm)    - \frac{1}{2} \l(  \be_i \SCAL \GRAD \rho, \commentC{ \alpha_{\uu} \uuDG + (1 - \alpha_{\uu}) \uu }  \r)  \\ 
=& \ (1-\alpha_\bbm  ) ( ( \DIV \bu ) \bbm, \be_i)  - b( \rho \be_i,\bu,\bu)    - \frac{1}{2} \l(  \be_i \SCAL \GRAD \rho,  \commentC{ \alpha_{\uu} \uuDG + (1 - \alpha_{\uu}) \uu }  \r)  \\ 
=& \ (1-\alpha_\bbm  ) ( ( \DIV \bu ) \bbm, \be_i)  + \frac{1}{2}  ( ( \DIV (\rho \be_i)) \bu , \bu)    - \frac{1}{2}  \l( \be_i \SCAL \GRAD \rho,  \commentC{ \alpha_{\uu} \uuDG + (1 - \alpha_{\uu}) \uu } \r)    \\ 
=& \ (1-\alpha_\bbm  ) ( ( \DIV \bu ) \bbm, \be_i) + \frac{1}{2}  \l( \be_i \SCAL \GRAD \rho, \uu \commentC{- \alpha_{\uu} \uuDG - (1 - \alpha_{\uu}) \uu } \r)   ,  
\end{split}%
\end{equation}
where \eqref{eq:IBP2} was used in the third equality \commentA{and \eqref{eq:move_rho} was used in the second equality}. Thus, momentum is conserved if $\alpha_\bbm = 1$ \commentC{  and if $\alpha_{\uu} = 0$.}
\noindent \\
\textbf{Angular momentum:} 
Conservation of angular momentum is derived by setting $\bv = \bphi_i$ inside \eqref{eq:mom_update_general} to obtain % (and follow similar steps as in the proof for Proposition \ref{proposition momentum}) 
\begin{equation}
\begin{split}
\p_t \int_\Omega (\bbm \CROSS \bx)_i\ud \bx = & \ ( \new{\p_t} \bbm , \bphi_i) =  - b( \bu, \bbm, \bphi_i ) - \alpha_\bbm ( ( \DIV \bu ) \bbm, \bphi_i) - (\GRAD P, \bphi_i) \\  & - \l( \alpha_P - \frac{1}{2}  \r) b(\bphi_i,\bbm,\bu) -  \l( \alpha_P + \frac{1}{2}  \r) b(\bphi_i,\bu,\bbm) + \frac{1}{2} \alpha_\rho ( (\DIV \bphi_i) \bbm , \bu)  \\ &- \frac{1}{2} \l( \l( \bphi_i \SCAL \GRAD \rho, \commentC{ \alpha_{\uu} \uuDG + (1 - \alpha_{\uu}) \uu } \r) +  \alpha_\rho \l( ( \DIV \bphi_i )  (\rho - \Bar{\rho}), \commentC{ \alpha_{\uu} \uuDG + (1 - \alpha_{\uu}) \uu }  \r) \r)  \\ \\
= & -b( \bu, \bbm, \bphi_i ) - \alpha_\bbm ( ( \DIV \bu ) \bbm, \bphi_i) - (\GRAD P, \bphi_i) \\  & - \l( \alpha_P - \frac{1}{2}  \r) b(\bphi_i,\bbm,\bu) -  \l( \alpha_P + \frac{1}{2}  \r) b(\bphi_i,\bu,\bbm)    - \frac{1}{2} \l(  \bphi_i \SCAL \GRAD \rho, \commentC{ \alpha_{\uu} \uuDG + (1 - \alpha_{\uu}) \uu }  \r),  
\end{split}
\end{equation}
since $ \DIV \bphi_i = 0$. Using \eqref{eq:IBP1} on $\l( \alpha_P - \frac{1}{2}  \r) b(\bphi_i,\bbm,\bu)$ and $b( \bu, \bbm, \bphi_i )$ and integration by parts on the pressure term yield%
\begin{equation}
\begin{split}
\p_t \int_\Omega (\bbm \CROSS \bx)_i\ud \bx = &   \ b( \bu, \bphi_i, \bbm ) + ( (\DIV \bu) \bbm, \bphi_i) - \alpha_\bbm ( ( \DIV \bu ) \bbm, \bphi_i) + ( P, \DIV \bphi_i) \\ &  + \l( \alpha_P - \frac{1}{2}  \r) b(\bphi_i,\bu,\bbm) -  \l( \alpha_P + \frac{1}{2}  \r) b(\bphi_i,\bu,\bbm)    - \frac{1}{2} \l( \bphi_i \SCAL \GRAD \rho, \commentC{ \alpha_{\uu} \uuDG + (1 - \alpha_{\uu}) \uu }  \r)    \\ 
= & \  b( \bu, \bphi_i, \bbm )  + (1 -\alpha_\bbm ) ( ( \DIV \bu ) \bbm, \bphi_i)  - b(\bphi_i,\bu,\bbm)    - \frac{1}{2} \l(  \bphi_i \SCAL \GRAD \rho, \commentC{ \alpha_{\uu} \uuDG + (1 - \alpha_{\uu}) \uu }  \r)   \\ 
= & \ b( \bu, \bphi_i, \bbm ) + (1 -\alpha_\bbm ) ( ( \DIV \bu ) \bbm, \bphi_i)  - b( \rho \bphi_i,\bu,\bu)   - \frac{1}{2} \l(  \bphi_i \SCAL \GRAD \rho, \commentC{ \alpha_{\uu} \uuDG + (1 - \alpha_{\uu}) \uu }  \r)   \\ 
=&  \ b( \bu, \bphi_i, \bbm ) + (1 -\alpha_\bbm ) ( ( \DIV \bu ) \bbm, \bphi_i)  + \frac{1}{2}  ( ( \DIV (\rho \bphi_i)) \bu , \bu)    - \frac{1}{2} \l(  \bphi_i \SCAL \GRAD \rho, \commentC{ \alpha_{\uu} \uuDG + (1 - \alpha_{\uu}) \uu }  \r)   \\ 
= & \ b( \bu, \bphi_i, \bbm ) + (1 -\alpha_\bbm ) ( ( \DIV \bu ) \bbm, \bphi_i) + \frac{1}{2}  \l( \bphi_i \SCAL \GRAD \rho  , \uu   \commentC{ - \alpha_{\uu} \uuDG - (1 - \alpha_{\uu}) \uu }  \r),    
\end{split}
\end{equation}
where \eqref{eq:IBP2} was used in the third equality \commentA{and \eqref{eq:move_rho} was used in the second equality}. Next, by using \eqref{eq:trilinear definition} it can be shown that
\begin{equation} \label{eq:trillinear angular momentum}
\begin{split}
 b(\bu,\bphi_i, \bbm) & = b(\bu, \bphi_i, \bu \rho)  = \frac{1}{2}b(\bu, \bphi_i, \bu \rho) + \frac{1}{2}b(\bu, \bphi_i, \bu \rho) \\ &  = \frac{1}{2} \l( (\GRAD \bphi_i)^\top \bu, \bu \rho \r) + \frac{1}{2} \l( (\GRAD \bphi_i)^\top \bu, \bu \rho \r) = \frac{1}{2} \l( (\GRAD \bphi_i)^\top \bu, \bu \rho \r) + \frac{1}{2} ( (\GRAD \bphi_i) \bu, \bu \rho ) = 0,
\end{split}
\end{equation}
since $\GRAD \bphi_i + (\GRAD \bphi_i)^\top = 0$. Since $b( \bu, \bphi_i, \bbm )=0$, angular momentum is conserved if $\alpha_\bbm = 1$ \commentC{ and $\alpha_{\uu} = 0$}.

%  By using \eqref{eq:trilinear definition} it can be shown that \lukas{This is a repetition of a result in preliminaries. Perhaps refer to that one?}
% \begin{equation}
% \begin{split}
% b(\bu,\bphi_i, \bbm) &= b(\bu, \bphi_i, \bu \rho)  = \frac{1}{2}b(\bu, \bphi_i, \bu \rho) + \frac{1}{2}b(\bu, \bphi_i, \bu \rho) = \frac{1}{2} ( (\GRAD \bphi_i)^\top \bu, \bu \rho ) + \frac{1}{2} ( (\GRAD \bphi_i)^\top \bu, \bu \rho ) \\ & = \frac{1}{2} ( (\GRAD \bphi_i)^\top \bu, \bu \rho ) + \frac{1}{2} ( (\GRAD \bphi_i) \bu, \bu \rho ) = 0,
% \end{split}
% \end{equation}
% since $\GRAD \bphi_i + (\GRAD \bphi_i)^\top = 0$. 

\end{proof}

\section{Properties of viscous regularizations.} \label{Sec:viscous_regularization}
In this section, we investigate how viscous regularizations of the incompressible Euler equations \eqref{eq:cons_law_primitive} affect conservation properties. We note that the purpose of this section is to investigate the properties of the model and not to investigate the properties of the numerical method. \commentC{Accordingly, we simplify the analysis by assuming that $\DIV \bu = 0$ holds pointwise and that the parameter $\alpha_{\uu}$ is unimportant since it is a property of the numerical method, not the model. We also assume that the domain is periodic.} We consider a viscous regularization of the following form

% Add mass diffusivity to the mass equation. Add normal kinematic/dynamic viscosity to the momentum equations. Table showing that kinematic/dynamic viscosity is energy dissipative and angular momentum conserving. Add general mass diffusivity term to momentum equations and show that a unique choice leads to momentum and kinetic energy conservation, but not conservation of angular momentum. A family of angular momentum mass diffusivity fluxes exists. Probably best to assume $\DIV \bu = 0$ here.
  % \bu |_{\p \Omega} &= 0, \\
           % \bn \SCAL \GRAD \rho |_{\p \Omega} &= 0, \\
% \lukas{How to definte BC in a pretty way...}
\begin{equation}\label{eq:nse_regularized}
      \begin{aligned}
        \p_t \rho + \bu \SCAL \GRAD \rho  & = \DIV   \bef_{\rho}  ,\\
 \p_t \bbm + \bu \SCAL \GRAD  \bbm  + \GRAD p  &= \bef + \DIV \l( \mu \l( \GRAD \bu + (\GRAD \bu)^\top \r)  \r)  + \bef_{\bbm} , \quad &(\bx,t) \in  \Omega \CROSS (0,T],\\
           \DIV \bu &= 0,\\
           \bu(\bx,0) &= \bu_{0}(\bx), \\
           \rho(\bx,0) &= \rho_{0}(\bx), \quad  &\bx\in \Omega,  \\
      \end{aligned}
\end{equation}
where $\mu \geq 0 $ is the dynamic viscosity coefficient which in general is space-dependent, $\bef_{\rho} := \kappa \GRAD \rho$, where $\kappa \geq 0$ is a mass diffusivity coefficient which also is space-dependent. In this work, we define $\bef_\bbm$ as

 % and $\bef_{\bbm}$ is a viscous function that depends on $\bu$ and $\bef_\rho$ and should compensate for the mass diffusivity in the mass equation. In this work, we investigate different choices of $\bef_{\bbm}$ and we aim to be as general as possible. The different choices for $\bef_{\bbm}$ we investigate include

\begin{equation} \label{eq:general_viscous_flux}
\bef_{\bbm} ( \bu , \bef_\rho ) := A_1 ( \GRAD \bef_{ \rho} ) \bu + A_2 (\GRAD \bef_{ \rho} )^\top \bu + A_3 (\GRAD \bu) \bef_{ \rho} + A_4 (\GRAD \bu)^\top \bef_{ \rho}  + A_5 (\DIV \bef_{ \rho} ) \bu,
\end{equation}
where $A_i \in \mR$. Depending on how the constants $A_i$ are chosen, different properties of the model are obtained and this is summarized in Theorem \ref{theorem:momentum-viscous-flux}.

\begin{theorem} \label{theorem:momentum-viscous-flux}
If $\bef = 0$, the system \eqref{eq:nse_regularized}:

\begin{enumerate}
\item Dissipates kinetic energy if $A_1 + A_2 - A_3 = 0 \text{ and } A_5 - \frac{1}{2} A_4 = \frac{1}{2}$.%Kinetic energy is non-increasing in time if 
\item Conserves momentum if $A_1 - A_3 = 0 \text{ and } A_5 - A_4 = 0$.
\item Conserves angular momentum if $A_1 - A_3 = 0 \text{ and } A_5 - A_4 = 0 \text{ and } A_2   =   A_4$.

% \item Kinetic energy if $\mu = 0$ and if $A_1 + A_2 - A_3 = 0 \text{ and } A_5 - \frac{1}{2} A_4 = \frac{1}{2}$.
% \item Momentum if $A_1 - A_3 = 0 \text{ and } A_5 - A_4 = 0$.
% \item Angular momentum if $A_1 - A_3 = 0 \text{ and } A_5 - A_4 = 0 \text{ and } A_2   =   A_4$.
\end{enumerate}

\end{theorem}

\begin{proof}
We divide the proof into several sections. Extending the proofs to other boundary conditions such as no-slip ($\bu = 0|_{\p \Omega}$) can be done by using the proof technique in \citep[Sec 3.1.2]{Charnyi2017} or \citep[Sec 3.1.2]{Ingimarson2023}.
% \lukas{Shouldn't be necessary to write this hopefully...} The corresponding weak form:

% \begin{equation}
% (\bef_{\bbm} , \bv) := A_1 b(\bv,\bef_{ \rho},\bu) + A_2 b(\bu,\bef_{ \rho},\bv) + A_3 b(\bv,\bu,\bef_{ \rho}) + A_4 b(\bef_{ \rho},\bu,\bv)  + A_5 ( (\DIV \bef_{ \rho} ) \bu , \bv)
% \end{equation}

%Taking an inner product of the momentum equations in \eqref{eq:nse_regularized} with $\bu$ and setting $\bef = 0$ and $\mu = 0$ yields. Set $\bv = \bu$ and set $w = -1/2|\bu|^2$

% \begin{equation} \label{eq:thm_part1}
% (\bbm_t  , \bu) + b(\bu,\bbm,\bu) + (\GRAD p, \bu) = ( \bef_{\bbm} , \bu )
% \end{equation}

% Taking an inner product with the mass equation and $- 1/2 |\bu|^2$ and adding it to \eqref{eq:thm_part1} yields

% \begin{equation} \label{eq:thm_part2}
% (\bbm_t  , \bu) - \frac{1}{2} ( \rho_t, |\bu|^2 ) - \frac{1}{2} ( \bu \SCAL \GRAD \rho , |\bu|^2 ) + b(\bu,\bbm,\bu) + (\GRAD p, \bu) = ( \bef_{\bbm} , \bu ) - \frac{1}{2} ( (\DIV \bef_{ \rho}) \bu , \bu )
% \end{equation}
\noindent \\
\textbf{Kinetic energy:} Repeating similar steps on \eqref{eq:nse_regularized} as in the proof of Theorem \ref{theorem div} and assuming $\DIV \bu = 0$ yields
\begin{equation} \label{eq:thm_part3}
\begin{split}
\frac{1}{2} \p_t \int_\Omega \rho \bu \SCAL \bu\ud \bx   = & \ ( \new{\p_t} \bbm , \bu ) - \frac{1}{2} \l( \new{\p_t} \rho , \bu \SCAL \bu \r)  =   \l( \bef_{\bbm} , \bu \r) - \frac{1}{2} \l( \l( \DIV \bef_{ \rho} \r) \bu , \bu \r) + \l( \DIV \l( \mu \l( \GRAD \bu + (\GRAD \bu)^\top \r)  \r) , \bu \r), \\
\end{split}
\end{equation}
where only the viscous terms determine if kinetic energy is dissipated. Next, we obtain
\begin{equation} 
\begin{split} \label{eq:thm_part4}
\frac{1}{2} \p_t \int_\Omega \rho \bu \SCAL \bu\ud \bx =& \ A_1 b(\bu,\bef_{ \rho},\bu) + A_2 b(\bu,\bef_{ \rho},\bu) + A_3 b(\bu,\bu,\bef_{ \rho})   + A_4 b(\bef_{ \rho},\bu,\bu) \\ & + A_5 ( (\DIV \bef_{ \rho} ) \bu , \bu) - \frac{1}{2} ( (\DIV \bef_{\rho} ) \bu ,\bu ) + \l( \DIV \l( \mu \l( \GRAD \bu + (\GRAD \bu)^\top \r)  \r) , \bu \r) , \\
= & \ A_1 b(\bu,\bef_{ \rho},\bu) + A_2 b(\bu,\bef_{ \rho},\bu) - A_3 b(\bu,\bef_{ \rho}, \bu) - \frac{1}{2} A_4 ( ( \DIV \bef_{ \rho})\bu,\bu) \\ &  + A_5 ( (\DIV \bef_{ \rho} ) \bu , \bu) - \frac{1}{2} ( (\DIV \bef_{\rho} ) \bu ,\bu ) - \l(  \mu \l( \GRAD \bu + (\GRAD \bu)^\top \r)  , \GRAD \bu \r),
\end{split}
\end{equation}
by using integration by parts on the $\mu$ term and using \eqref{eq:IBP1} on the $A_3$ term and \eqref{eq:IBP2} on the $A_4$ term. Using that the contraction between a symmetric and anti-symmetric matrix is zero \citep[Ch 11.2.1]{Larson_2013}, \ie $(A + A^\top ): (A - A^\top) = 0$ where $A$ is a square matrix, one can show that
\begin{equation} \label{eq:sym}
\begin{split}
\l(  \mu \l( \GRAD \bu + \l( \GRAD \bu \r)^\top \r) , \GRAD \bu \r) =& \frac{1}{2}  \l(  \mu \l( \GRAD \bu +  \l( \GRAD \bu \r)^\top \r), \GRAD \bu + (\GRAD \bu)^\top \r) \\ &+  \frac{1}{2} \l(  \mu \l( \GRAD \bu + \l( \GRAD \bu \r)^\top \r), \GRAD \bu - (\GRAD \bu)^\top \r) \\ =& \frac{1}{2}  \l(  \mu \l( \GRAD \bu + \l( \GRAD \bu \r)^\top \r) , \GRAD \bu + (\GRAD \bu)^\top \r).
\end{split}
\end{equation}
Inserting \eqref{eq:sym} into \eqref{eq:thm_part4} gives
\begin{equation} 
\begin{split} \label{eq:thm_part5}
\frac{1}{2} \p_t \int_\Omega \rho \bu \SCAL \bu\ud \bx =& \ A_1 b(\bu,\bef_{ \rho},\bu) + A_2 b(\bu,\bef_{ \rho},\bu) - A_3 b(\bu,\bef_{ \rho}, \bu) - \frac{1}{2} A_4 ( ( \DIV \bef_{ \rho})\bu,\bu) \\ &  + A_5 ( (\DIV \bef_{ \rho} ) \bu , \bu) - \frac{1}{2} ( (\DIV \bef_{\rho} ) \bu ,\bu ) -  \frac{1}{2} \| \sqrt{\mu} \l( \GRAD \bu + (\GRAD \bu)^\top \r) \|^2.
\end{split}
\end{equation}

From \eqref{eq:thm_part5}, the conditions necessary for kinetic energy to be dissipated are
\begin{equation}
A_1 + A_2 - A_3 = 0 \text{ and } A_5 - \frac{1}{2} A_4 = \frac{1}{2}.
\end{equation}
\noindent \\
\textbf{Momentum:} Repeating similar steps on \eqref{eq:nse_regularized} as in the proof of Theorem \ref{theorem div} and assuming $\DIV \bu = 0$ yields
\begin{equation} \label{eq:thm_mom_part1}
\p_t \int_\Omega \bbm_i\ud \bx = ( \new{\p_t} \bbm  , \be_i) =   \l( \DIV \l( \mu \l( \GRAD \bu + (\GRAD \bu)^\top \r)  \r) , \be_i \r) + ( \bef_{\bbm} , \be_i ),
\end{equation}
where only the viscous terms determine if momentum is conserved. Performing integration by parts on the $\mu$ term yields
\begin{equation} \label{eq:mom_step1}
\p_t \int_\Omega \bbm_i\ud \bx =   A_1 b(\be_i,\bef_{ \rho},\bu) + A_2 b(\bu,\bef_{ \rho},\be_i) + A_3 b(\be_i,\bu,\bef_{ \rho}) + A_4 b(\bef_{ \rho},\bu,\be_i)  + A_5 ( (\DIV \bef_{ \rho} ) \bu , \be_i).
\end{equation}

Using \eqref{eq:IBP1} on the $A_3$, $A_2$ and $A_4$ terms inside \eqref{eq:mom_step1} gives
\begin{equation}
\p_t \int_\Omega \bbm_i\ud \bx = A_1 b(\be_i,\bef_{ \rho},\bu)  - A_3 b(\be_i,\bef_{ \rho},\bu) - A_4 ( (\DIV\bef_{ \rho}) \bu,\be_i)  + A_5 ( (\DIV \bef_{ \rho} ) \bu , \be_i),
\end{equation}
which leads to the conditions necessary for momentum to be conserved
\begin{equation}
A_1 - A_3 = 0 \text{ and } A_5 - A_4 = 0.
\end{equation}
\noindent \\
\textbf{Angular momentum:} Repeating similar steps on \eqref{eq:nse_regularized} as in the proof of Theorem \ref{theorem div} and assuming $\DIV \bu = 0$ yields
\begin{equation} \label{eq:thm_ang_part1}
 \p_t \int_\Omega (\bbm \CROSS \bx)_i\ud \bx = \l(\DIV \l( \mu \l( \GRAD \bu + (\GRAD \bu)^\top \r)  \r) , \bphi_i \r) + ( \bef_{\bbm} , \bphi_i ).
\end{equation}

Performing integration by parts on the $\mu$ term yields
\begin{equation} 
\p_t \int_\Omega (\bbm \CROSS \bx)_i\ud \bx  = -(  \mu \l( \GRAD \bu + (\GRAD \bu)^\top \r)   , \GRAD \bphi_i) + ( \bef_{\bbm} , \bphi_i ).
\end{equation}

Next, using that the contraction between a symmetric and anti-symmetric matrix is zero \citep[Ch 11.2.1]{Larson_2013}, one can show that
\begin{equation} \label{eq:symmetry_trick}
\begin{split}
\l( \mu   \l( \GRAD \bu +  (\GRAD \bu)^\top \r) , \GRAD \bphi_i \r) =&   \l( \frac{1}{2} \mu  \l( \GRAD \bu +  \l( \GRAD \bu \r)^\top \r), \GRAD \bphi_i + (\GRAD \bphi_i)^\top \r) \\ &+   \l( \frac{1}{2} \mu  \l( \GRAD \bu + \l( \GRAD \bu \r)^\top \r), \GRAD \bphi_i - (\GRAD \bphi_i)^\top \r) \\ =&   \l( \frac{1}{2} \mu  \l( \GRAD \bu + \l( \GRAD \bu \r)^\top \r) , \GRAD \bphi_i + (\GRAD \bphi_i )^\top \r).
\end{split}
\end{equation}
Since $\GRAD \bphi_i = - \GRAD \bphi_i^\top$ we are left with
\begin{equation} 
\begin{split}
\p_t \int_\Omega (\bbm \CROSS \bx)_i\ud \bx = & \ ( \bef_{\bbm} , \bphi_i ) =  A_1 b(\bphi_i,\bef_{ \rho},\bu) + A_2 b(\bu,\bef_{ \rho},\bphi_i) \\ & + A_3 b(\bphi_i,\bu,\bef_{ \rho}) + A_4 b(\bef_{ \rho},\bu,\bphi_i)  + A_5 ( (\DIV \bef_{ \rho} ) \bu , \bphi_i).
\end{split}
\end{equation}

Using \eqref{eq:IBP1} on the $A_3$, $A_2$ and $A_4$ terms and using that $\DIV \bu = \DIV \bphi_i = 0$ gives
\begin{equation} 
\begin{split}
\p_t \int_\Omega (\bbm \CROSS \bx)_i\ud \bx =& \ A_1 b(\bphi_i,\bef_{ \rho},\bu) -A_2 b ( \bu, \bphi_i,  \bef_{ \rho} )  - A_3 b(\bphi_i,\bef_{ \rho},\bu) \\  &- A_4 b(\bef_{ \rho},\bphi_i,\bu) - A_4 ( (\DIV\bef_{ \rho}) \bu,\bphi_i)  + A_5 ( (\DIV \bef_{ \rho} ) \bu , \bphi_i).
\end{split}
\end{equation}

Next, using the definition of the trilinear form \eqref{eq:trilinear definition} yields
\begin{equation}
\begin{split}
\p_t \int_\Omega (\bbm \CROSS \bx)_i\ud \bx   = & \ A_1 b(\bphi_i,\bef_{ \rho},\bu) -A_2  \l( (\GRAD \bphi_i)^\top \bu,  \bef_{ \rho} \r)  - A_3 b(\bphi_i,\bef_{ \rho},\bu) \\& - A_4 ( (\GRAD  \bphi_i)\bu, \bef_{ \rho}) - A_4 ( (\DIV\bef_{ \rho}) \bu,\bphi_i)  + A_5 ( (\DIV \bef_{ \rho} ) \bu , \bphi_i),
\end{split}
\end{equation}
which leads to the conditions necessary for angular momentum to be conserved
\begin{equation}
A_1 - A_3 = 0 \text{ and } A_5 - A_4 = 0 \text{ and } A_2   =   A_4,
\end{equation}
since $\GRAD \bphi_i + (\GRAD \bphi_i)^\top  = 0 $.

\end{proof}

\subsection{Discussion.}
One implication of Theorem \ref{theorem:momentum-viscous-flux} is that there is no viscous regularization $(\kappa > 0)$ that simultaneously dissipates kinetic energy and conserves angular momentum. Another implication is that a unique viscous regularization dissipates kinetic energy and conserves momentum. This can be seen by combining the kinetic energy and the momentum conditions in Theorem \ref{theorem:momentum-viscous-flux} to arrive at $A_5 = A_4 = 1$, $A_2 = 0$ and $A_1 = A_3$. Setting these conditions inside \eqref{eq:general_viscous_flux} and performing some simplifications lead to the following viscous flux
\begin{equation} \label{eq:guermond_popov_flux}
\bef_{GP} :=  (\GRAD \bu)^\top \bef_\rho + (\DIV \bef_\rho) \bu + A_1 (\GRAD \bef_\rho) \bu + A_1 (\GRAD \bu ) \bef_\rho =   \DIV ( \bef_\rho \otimes \bu ) + A_1  \GRAD ( \bef_\rho \SCAL \bu ),
\end{equation}
where the second term does not change the model since it only results that a modified pressure is solved for, \ie $P = p - A_1 \bef_\rho \SCAL \bu $, and does not affect the behavior of density and velocity. Since that term is not important, we are left with the so-called Guermond-Popov viscous flux \citep{Guermond_Popov2014}. Compared to \eqref{eq:guermond_popov_flux}, the Guermond-Popov flux \citep{Guermond_Popov2014} is simplified since it now is in the incompressible setting instead of the compressible setting. This means that, if we only consider the density and momentum contribution of \citep{Guermond_Popov2014}, it is exactly the same as \eqref{eq:guermond_popov_flux} and \eqref{eq:nse_regularized}. %Remarkably, the condition that kinetic energy is dissipated and momentum is conserved leads to a unique viscous flux.

% It is however not possible to conserve kinetic energy, momentum and angular momentum at the same time using \eqref{eq:general_viscous_flux}. 
% The is no viscous regularization $\bef_\bbm$ that simultaneously dissipates kinetic energy and conserves momentum and angular momentum. 

% Instead, there exists a family of viscous regularizations that conserve momentum and angular momentum. Two examples of these fluxes are given below
Combining the last two conditions of Theorem \ref{theorem:momentum-viscous-flux} leads to a family of viscous regularizations that conserve momentum and angular momentum. Below two examples of these are given
\begin{align}
&\bef_0 := 0, \\
% \bef_{Birman?} = WIP \\
&  \new{ \bef_{KS} :=     \DIV (   \bef_\rho \otimes \bu + \bu \otimes \bef_\rho  ), } \label{eq:GP_S}
\end{align}
where $\bef_0$ has previously been used in the variable density flow literature \citep{Bartholomew_2019}. \new{From Ficks law, \cite{Kaz_Smag_1977} derived and proposed \eqref{eq:GP_S}. They showed kinetic energy dissipation assuming that $\kappa$ is constant and
\begin{equation} \label{eq:dissipation_condition}
\kappa < \frac{2 \mu}{ \max_\Omega \rho - \min_\Omega \rho }. % \Leftrightarrow \kappa < \frac{2 \rho \nu }{ \max_\Omega \rho - \min_\Omega \rho }.
\end{equation}
It is not clear if \eqref{eq:dissipation_condition} is true for all fluid pairs at all temperatures. For example, consider hydrogen gas $\text{H}_2$ and sulfur hexafluoride gas $\text{SF}_6$ at room temperature. Values of dynamic viscosity, density and pairwise mass diffusivity are given in Table \ref{table:counter_example}. Inserting these values into \eqref{eq:dissipation_condition} shows an order of magnitude violation:
\begin{equation}
\begin{split}
% \kappa &= 0.412 *10^{-4}(m^2/sec) \\ 
\frac{2 \mu}{ \max_\Omega \rho - \min_\Omega \rho } &= \frac{2 \cdot 15.3 \cdot 10^{-6}}{6.07 - 0.0838 } \text{ [m$^2$/sec]} \\
 & = 5.1 \cdot 10^{-6} \text{ [m$^2$/sec],} \\ \\
\kappa =  4.12 \cdot 10^{-5}\text{ [m$^2$/sec]} &\nless 5.1 \cdot 10^{-6} \text{ [m$^2$/sec].}
\end{split}
\end{equation}
}
\begin{table}[H]
\caption{\new{Experimental values from \citep{ref_values_2015} of dynamic viscosity at 300 K, 1 kPa and pairwise mass diffusion at 20 \textcelsius, 1 atm. Density from ideal gas law.}} \label{table:counter_example} \centering
\begin{tabular}{c|ccc}
\toprule & $\mu$ [Pa $\cdot$ sec] & $\rho$ [kg/m$^3$] & $\kappa$ [m$^2$/sec] \\ \midrule
$\text{H}_2$ & $8.9 % 8.9 6.8 %
\cdot 10^{-6}$ &  0.0838% 0.08988 [Wiki] / 0.08375 at NTP  
& $4.12 \cdot 10^{-5}$ \\
$\text{SF}_6$ & $15.3 \cdot 10^{-6}$ & 6.07 %Ideal gas law. 6.17 [Wiki] / 6.139 / 6.16 at NTP https://www.electrical4u.com/sulfur-hexafluoride-sf6-gas-properties/ 
& $4.12 \cdot 10^{-5}$ \\ \bottomrule
\end{tabular}
\end{table} %The simulations performed by Birman used  XYZ which does not fulfill any of the conditions of Theorem X.

% We do not claim that any of the viscous regularizations \eqref{eq:guermond_popov_flux}-\eqref{eq:GP_S} correspond to a more correct or less correct physical model. Instead, the conclusion is that if mass diffusivity is included in the physical model, depending on which viscous flux is chosen in the momentum equations \eqref{eq:general_viscous_flux}, the model can, at most, either conserve momentum and kinetic energy or momentum and angular momentum. 

%If $\mu = \kappa$ and $\mu$ is assumed to be constant then the Guermond-Popov flux simplifies to a constant coefficient monolithic flux which is also angular momentum conserving, see \cite[Sec 5.3.1]{Svard2018}.

% If mass diffusivity is present in the physical model real-world physical experiments need to be performed to determine which one of these is correct. \lukas{Maybe better in conclusion?: If the simulation is underresolved \ie $\kappa << h$ or $\mu << h$ we recommend that a stabilization technique is used (for example the RV method [REF]).}

% Kinetic energy + momentum $\Rightarrow$ Guermond-Popov flux 

% Momentum conditions + kinetic energy conditions lead to: $A_5 = A_4 = 1$ which iS GP. Also that $A_2 = 0$ and $A_1 = A_3$. 

% Kinetic energy + momentum + angular momentum $\Rightarrow$ not possible

% Angular momentum + momentum $\Rightarrow$ Family of viscous fluxes

% GP-flux is more general due to potential term. Only results in modified pressure term, no difference in density or velocity. Different models and refer to papers:

% Refer to the appendix.

\section{Fully discrete conservative method.} \label{Sec:fully_discrete}
In this section, we describe the fully discrete method approximating the incompressible Navier-Stokes equations \eqref{eq:nse_regularized}. Let $(\rho^n, \bu^n, P^n) \in \commentA{ \calM \CROSS \mathbf{\bcalV} \CROSS \calQ }$ be solutions at time $t_n$, where the time-levels are defined as $0 = t_0 <t_1 <\dots < t_{N_t} = T$ and ${N_t}$ denotes the total number of time-levels.  Let $\Delta t_{n+j} := t_{n+j+1} - t_{n+j}$ denote the time step. We also define $\bu^{n+1/2} :=  \frac{\bu^n + \bu^{n+1}}{2} $. % and $\omega_i := \Delta t_i / \Delta t_{i-1}$ denote the time step ratio.

% \lukas{Should I include the viscous boundary term? Probably a good idea...}

\commentC{We propose using a second-order accurate modified Crank-Nicolson method in time.} Applying this method on \eqref{eq:mom_update_general} yields: Given $\l(\rho^{n}, \bu^{n}\r)$, find $\l(\rho^{n+1}, \bu^{n+1}, P^{n+1/2}\r) \in\commentA{ \calM \CROSS \mathbf{\bcalV} \CROSS \calQ }$ such that

% \begin{equation} 
% \begin{alignedat}{2} \label{eq:bdf_update}
% \l( d_t \l( \rho_{h}^{n+1}  , w \r)  + \bu^{n+1} \SCAL \GRAD \rho^{n+1} + \alpha_\rho \l(\DIV \bu^{n+1}\r) \l(  \rho^{n+1} - \Bar{\rho}^{n+1}\r) , w \r) &\\ + \l( \kappa \GRAD \rho^{n+1} , \GRAD w\r) - \l( \kappa \bn \SCAL \GRAD \rho^{n+1} , w  \r)_{\p \Omega} = 0 &, \quad &&\forall w \in \calM, \\
%  \l(  d_t \l(  \bbm^{n+1} \r) , \bv \r) +  b\l(\bu^{n+1}, \bbm^{n+1}, \bv\r) + \alpha_\bbm  \l(\l( \DIV \bu^{n+1}  \r) \bbm^{n+1}, \bv\r)  + \l( \GRAD P^{n+1}, \bv \r)& \\ + \l( \alpha_P - \frac{1}{2} \r) b\l(\bv, \bbm^{n+1}, \bu^{n+1}\r) + \l( \alpha_P + \frac{1}{2} \r) b\l(\bv, \bu^{n+1}, \bbm^{n+1}\r)    + \frac{\alpha_\rho}{2} \l(  \GRAD \l(    \bbm^{n+1} \SCAL \bu^{n+1} \r) , \bv \r)&\\ + \frac{1}{2} \l(   \l( \bv \SCAL \GRAD \rho^{n+1}, \uun \r)  - \alpha_\rho \l(   \GRAD \l(  \l(  \rho^{n+1} - \Bar{\rho}^{n+1} \r)   \uun \r) , \bv \r) \r)& \\ + \l( \mu \l( \GRAD \bu^{n+1} + \l( \GRAD \bu^{n+1} \r)^\top  \r), \GRAD \bv \r) - \l( \mu \bn \SCAL \l( \GRAD \bu^{n+1} + \l( \GRAD \bu^{n+1} \r)^\top  \r), \bv    \r)_{\p \Omega} \\ = \l(\bef^{n+1}, \bv\r) + \l(\bef_{\bbm} \l( \bu^{n+1} , \bef_\rho^{n+1} \r) , \bv\r),& \quad && \forall \bv \in \bcalV, \\
%  \l( \DIV \bu^{n+1}, q \r) &= 0, \quad && \forall q \in \calQ, \\
% \end{alignedat}
% \end{equation}

\begin{equation} \label{eq:full_time_disctetization}
\begin{alignedat}{2}
\l ( \frac{ \rho^{n+1} - \rho^n}{\Delta t_n} + \bu^{n+1/2} \SCAL \GRAD \rho^{n+1/2} + \alpha_\rho \l( \DIV \bu^{n+1/2} \r) \l(  \rho^{n+1/2} - \Bar{\rho}^{n+1/2} \r) , w \r) \\  + \l( \kappa \GRAD \rho^{n+1/2} , \GRAD w \r)  &= 0, \quad &&  \forall w \in \calM ,%- \l( \kappa \bn \SCAL \GRAD \rho^{n+1/2} , w  \r)_{\p \Omega}
\\
\l( \frac{\bbm^{n+1} - \bbm^{n}}{\Delta t_n}, \bv \r) +  b\l(\bu^{n+1/2}, \bbm^{n+1/2}, \bv\r) 
\\ 
+ \alpha_\bbm  \l(\l( \DIV \bu^{n+1/2}  \r) \bbm^{n+1/2}, \bv\r)  + \l( \GRAD P^{n+1/2}, \bv \r) 
\\ 
+ \l( \alpha_P - \frac{1}{2} \r) b\l(\bv, \bbm^{n+1/2}, \bu^{n+1/2}\r) 
\\
 + \l( \alpha_P + \frac{1}{2} \r) b\l(\bv, \bu^{n+1/2}, \bbm^{n+1/2}\r)    + \frac{\alpha_\rho}{2} \l(    \GRAD \l(  \bbm^{n+1/2} \SCAL \bu^{n + 1/2} \r) , \bv \r) 
\\ 
+ \frac{1}{2} \l(  \l(  \bv \SCAL \GRAD \rho^{n+1/2} , \commentC{ \alpha_{\uu} \uuDGn + (1 - \alpha_{\uu}) \uum } \r) \r . \\ \l . - \alpha_\rho \l(   \GRAD \l( \l(  \rho^{n+1/2} - \Bar{\rho}^{n+1/2} \r)   \commentC{  \l( \alpha_{\uu} \uuDGn + (1 - \alpha_{\uu}) \uum \r) } \r) , \bv \r) \r) 
\\ 
+ \l( \mu \l( \GRAD \bu^{n+1/2} + \l( \GRAD \bu^{n+1/2} \r)^\top  \r), \GRAD \bv \r) %- \l( \mu \bn \SCAL \l( \GRAD \bu^{n+1/2}  + \l( \GRAD \bu^{n+1/2} \r)^\top \r) , \bv \r)_{\p \Omega} 
\\ 
-\l(\bef^{n+1/2}, \bv\r) - \l(\bef_{\bbm} \l( \bu^{n+1/2} , \bef_\rho^{n+1/2} \r) , \bv\r)
& =0 , \quad && \forall \bv \in \bcalV ,
\\
\l( \DIV \bu^{n+1/2} , q \r) &= 0, \quad && \forall q \in \calQ.
\end{alignedat}
\end{equation}
\commentC{Note that the term $\alpha_{\uu} \uuDG + (1 - \alpha_{\uu}) \uu $ has been discretized as $\alpha_{\uu} \uuDGn + (1 - \alpha_{\uu}) \uum $. We note that $\alpha_{\uu}=0$ correspond to a standard Crank-Nicolson discretization of \eqref{eq:mom_update_general} whereas $\alpha_{\uu}=1$ corresponds to a modified Crank-Nicolson method as proposed by \cite{Gawlik2020}.}

% the modified Crank-Nicolson method by \cite{Gawlik2020} which has favorable fully discrete \commentC{conservation} properties.

% $\mu = \kappa = 0$ In this section, we show that the properties of Table \ref{table:emac_mom} hold in the fully discrete setting. As such, we set $\kappa = \mu = 0$ and assume periodic boundary conditions. %that $\bu \SCAL \bn = 0 |_{\p \Omega}$. 
\subsubsection{Properties of the fully discrete method.}
This section is a fully discrete analog of Section \ref{Sec:semi-discrete properties} and aims to verify that the fully discrete method satisfies the properties of Table \ref{table:emac_mom}. \commentC{Provided that the conditions of Table \ref{table:emac_mom} are fulfilled, all of the properties are retained at the fully discrete level.}

% \commentC{Remove:except conservation of momentum and angular momentum. To conserve momentum and angular momentum, the term $\uumm$ should be discretized as $\uum$ instead but this will lead to a loss of conservation of kinetic energy. Additional techniques are required to conserve both kinetic energy and momentum and angular momentum. For example, recently \cite{Zhang_2023} used the so-called scalar auxiliary variable (SAV) technique to enhance their time-stepping scheme to be able to conserve both kinetic energy and momentum.} %In the authors' opinion, this technique should be able to be extended to this scheme.
% \commentC{Think how to rewrite.}

% This section is a fully discrete analog of Section \ref{Sec:semi_discrete_proof} and aims to verify that the fully discrete method satisfies the properties of Table \ref{table:emac_mom}

\begin{proposition}
\commentC{The properties of Table \ref{table:emac_mom} hold at the fully discrete level for \eqref{eq:full_time_disctetization} provided that the conditions inside the table are fulfilled.}
\end{proposition}
\noindent \\
\textbf{Mass:}
Repeating the steps of the proof of Theorem \ref{theorem div} on \eqref{eq:full_time_disctetization} yields
\begin{equation}
\int_\Omega \rho^{n+1}\ud \bx = \int_\Omega \rho^{n}\ud \bx + \Delta t_n \l( 1 -\alpha_\rho  \r) \l( \DIV \bu^{n+1/2} ,  \rho^{n+1/2} - \Bar{\rho}^{n+1/2}  \r),
\end{equation}
and similar steps can be performed to show that $\l( \rho^{n+1/2} -  \Bar{\rho}^{n+1/2} \r) \in \calQ$ provided that $k_\rho \leq k_P$. Mass is conserved if $k_\rho \leq k_P$ or if $\alpha_\rho = 1$.

% \new{which shows that mass is conserved if $\alpha_\rho = 1$. Similar steps can also be performed to show that for $A \in \mR$
% \begin{equation}
% \int_\Omega \rho^{n+1}\ud \bx = \int_\Omega \rho^{n}\ud \bx + \Delta t_n  \l( \DIV \bu^{n+1/2} , \rho - A - \alpha_\rho \l( \rho^{n+1/2}  - \Bar{\rho}^{n+1/2}  \r) \r).
% \end{equation}
% There exists an $A$ so that $\l( \rho - A - \alpha_\rho \l( \rho^{n+1/2}  - \Bar{\rho}^{n+1/2}  \r) \r) \in \calQ $ provided that $k_\rho \leq k_P$. This means that mass is conserved if $k_\rho \leq k_P$ by using the weak divergence-free constraint \eqref{eq:full_time_disctetization}.}
\noindent \\
\textbf{Squared density:}
Setting $w = \rho^{n+1/2} -  \Bar{\rho}^{n+1/2}$ inside \eqref{eq:full_time_disctetization} and repeating the steps of the proof of Theorem \ref{theorem div} yields
\begin{equation}
\l(\rho^{n+1} - \rho^n, \rho^{n+1/2} - \Bar{\rho}^{n+1/2} \r) = \Delta t_n  \l( \frac{1}{2} - \alpha_\rho \r) \l( \l(\DIV \bu^{n+1/2}\r) \l(\rho^{n+1/2} - \Bar{\rho}^{n+1/2} \r) , \rho^{n+1/2} - \Bar{\rho}^{n+1/2} \r) .
\end{equation}
Simplifying the left-hand side finally yields
\begin{equation}
\begin{split}
& \frac{1}{2} \| \rho^{n+1} \|^2 - \l(\rho^{n+1},   \Bar{\rho}^{n+1/2} \r)    =  \frac{1}{2} \| \rho^n \|^2 - \l(\rho^n,   \Bar{\rho}^{n+1/2} \r) \\ & + \Delta t_n  \l( \frac{1}{2} - \alpha_\rho \r) \l( \l(\DIV \bu^{n+1/2}\r) \l(\rho^{n+1/2} - \Bar{\rho}^{n+1/2} \r) , \rho^{n+1/2} - \Bar{\rho}^{n+1/2} \r) ,
\end{split}
\end{equation}
which is a fully discrete equivalent of \eqref{eq:sq_proof1}. If $\Bar{\rho}^{n+1/2} = \frac{1}{|\Omega|} \int_\Omega \rho^{n+1/2}\ud \bx  $ then
\begin{equation}
\frac{1}{2} \| \rho^{n+1} \|^2 - \l(\rho^{n+1},   \Bar{\rho}^{n+1/2} \r)  - \l(  \frac{1}{2} \| \rho^n \|^2 - \l(\rho^n,   \Bar{\rho}^{n+1/2} \r) \r),
\end{equation}
simplifies to
\begin{equation}
\frac{1}{2} \l( \| \rho^{n+1} \|^2 -  \l( \int_\Omega  \rho^{n+1}\ud \bx \r)^2 \r)    - \frac{1}{2} \l(   \| \rho^n \|^2 -  \l( \int_\Omega  \rho^{n}\ud \bx \r)^2 \r),
\end{equation}
which is a fully discrete equivalent of \eqref{eq:sq_proof2}.

\new{Lastly, similar steps can be performed to show that for $A \in \mR$
\begin{equation}
\begin{split}
& \frac{1}{2} \| \rho^{n+1} \|^2 - \l(\rho^{n+1},   A \r)    =  \frac{1}{2} \| \rho^n \|^2 - \l(\rho^n,   A \r) \\ & + \Delta t_n   \l( \DIV \bu^{n+1/2}, \frac{1}{2} \l(\rho^{n+1/2} - A \r)^2  - \alpha_\rho \l( \rho^{n+1/2} - \Bar{\rho}^{n+1/2} \r) \l( \rho^{n+1/2} - A \r) \r) .
\end{split}
\end{equation}
There exists an $A$ such that $\l( \frac{1}{2} \l(\rho^{n+1/2} - A \r)^2  - \alpha_\rho \l( \rho^{n+1/2} - \Bar{\rho}^{n+1/2} \r) \l( \rho^{n+1/2} - A \r) \r) \in \calQ$ provided that $2 k_\rho \leq k_P$. This means that squared density is conserved if $2 k_\rho \leq k_P$ by using the weak divergence-free condition \eqref{eq:full_time_disctetization} and that mass is conserved.}

\noindent \\
\textbf{Shift invariance:}
Substituting $\rho^{n+j}$ with $\rho^{n+j} + c$ and repeating the steps of the proof of Theorem \ref{theorem div} shows that the density update \eqref{eq:full_time_disctetization} is invariant if $\alpha_\rho = 0$ or if $\Bar{\rho}^{n+1/2} = \frac{1}{|\Omega|} \int_\Omega \rho^{n+1/2}\ud \bx  $.
\noindent \\
\textbf{Kinetic energy:}
We begin by using the identity from \citep[Sec 3]{Gawlik2020} and the definition of the \commentC{$L^2$-projection \eqref{eq:uu_projection}} to obtain
\begin{equation} \label{eq:gawlik_identity_}
\begin{split}
\frac{1}{2} \| \sqrt{\rho^{n+1}} \bu^{n+1} \|^2 - \frac{1}{2} \| \sqrt{\rho^{n}} \bu^{n} \|^2 &= \l( \rho^{n+1} \bu^{n+1} - \rho^n \bu^n, \bu^{n+1/2} \r) - \frac{1}{2} \l(\rho^{n+1} - \rho^n, \bu^n \SCAL \bu^{n+1}\r) \\
&= \l( \rho^{n+1} \bu^{n+1} - \rho^n \bu^n, \bu^{n+1/2} \r) - \frac{1}{2} \l(\rho^{n+1} - \rho^n, \overline{\uumm} \r).
\end{split}
\end{equation}
Next, repeating the steps of the proof of Theorem \ref{theorem div} on \eqref{eq:full_time_disctetization} leads to
\begin{equation}
\begin{split}
\frac{1}{2} \l( \| \sqrt{\rho^{n+1}} \bu^{n+1} \|^2 -  \| \sqrt{\rho^{n}} \bu^{n} \|^2 \r) =  \Delta t_n \l(\commentA{ \alpha_P + \frac{1}{2} - \alpha_\bbm + \frac{1}{2} \alpha_\rho  } \r) \l( \l( \DIV \bu^{n+1/2} \r) \bbm^{n+1/2}, \bu^{n+1/2}\r) \\
\\ 
\commentA{- \frac{1}{2} \l( \bu^{n+1/2} \SCAL \GRAD \rho^{n+1/2}  + \alpha_\rho  (\DIV \bu^{n+1/2}) ( \rho^{n+1/2} - \Bar{\rho}^{n+1/2} ) , \commentC{ \alpha_{\uu} \uuDGn + (1 - \alpha_{\uu}) \uum - \uuDGn } \r)} \\ 
% \commentA{ + \frac{1}{2} \l( \bu \SCAL \GRAD \rho  + \alpha_\rho  (\DIV \bu) ( \rho - \Bar{\rho} ) , \uuDGn  \r),}
\end{split}
\end{equation}
which shows that kinetic energy is conserved if $\alpha_\bbm - \alpha_P - \alpha_\rho/2 = 1/2$ \commentC{ and if $\alpha_{\uu}=1$}.
\noindent \\
\textbf{Momentum:}
Setting $\bv = \be_i$ inside \eqref{eq:full_time_disctetization} and repeating the steps of the proof of Theorem \ref{theorem div} yields

\begin{equation} \label{eq:momentum_full_proof}
\begin{split}
 \int_\Omega \bbm_i^{n+1}  \ud\bx = \int_\Omega \bbm_i^{n}  \ud\bx + \commentA{\Delta t_n} \l(1-\alpha_\bbm  \r) \l( \l( \DIV \bu^{n+1/2} \r) \bbm^{n+1/2}, \be_i\r) \\ +  \frac{1}{2} \commentA{\Delta t_n}  \l( \be_i \SCAL \GRAD \rho^{n+1/2},   \commentC{ \uum  - \alpha_{\uu} \uuDGn - (1 - \alpha_{\uu}) \uum } \r)    .
\end{split}
\end{equation}
\commentC{Momentum is conserved if $\alpha_{\bbm} = 1$ and if $\alpha_{\uu} = 0$.}
\noindent \\
\textbf{Angular momentum:}
Setting $\bv = \bphi_i$ inside \eqref{eq:full_time_disctetization} and repeating the steps of the proof of Theorem \ref{theorem div} yields

\begin{equation} \label{eq:angular momentum_full_proof}
\begin{split}
 \int_\Omega \l(\bbm  \CROSS \bx\r)_i^{n+1}  \ud\bx = \int_\Omega \l(\bbm  \CROSS \bx\r)_i^{n}  \ud\bx + \commentA{\Delta t_n} \l(1-\alpha_\bbm  \r) \l( \l( \DIV \bu^{n+1/2} \r) \bbm^{n+1/2}, \bphi_i\r) \\ +  \frac{1}{2} \commentA{\Delta t_n}  \l( \bphi_i \SCAL \GRAD \rho^{n+1/2}, \commentC{ \uum  - \alpha_{\uu} \uuDGn - (1 - \alpha_{\uu}) \uum } \r)    .  
\end{split}
\end{equation}
\commentC{Angular momentum is conserved if $\alpha_{\bbm} = 1$ and if $\alpha_{\uu} = 0$.}

\subsection{Adaptive time-stepping}

% \lukas{Rewrite}

% Let us denote the ratio of time steps as $\omega_i := \Delta t_i / \Delta t_{i-1}$. Denote by $[s_{\min}, s_{\max}]$ the interval in which $\omega_i$ lives and ensures the zero-stability of the BDF scheme. 

We use a simple adaptive time-stepping algorithm based on the CFL condition. Given user-defined parameters \commentA{CFL}, $s_{\min}$ and $s_{\max}$, the next time step is computed using the following adaptive algorithm:

\begin{algorithm}[H]
\caption{A simplified version of the time step control algorithm from \citep[Sec 5.4]{Guermond_Minev_2019} used to compute the next time step $\Delta t_n$ given user-defined parameters \commentA{CFL,} $s_\text{max}, s_\text{min}$.}\label{algorithm_constant_density}
\tcc{Compute time step increment based on CFL condition}
$s_\text{cfl} = \textrm{CFL} \min_{\bx \in \Omega} h(\bx) / \l(\| \bu^n \|_{l^2} \Delta t_{n-1} \r)$
% \\
% $e_{\text{loc}} = ...$ \tcc{An arbitrary estimation of the local relative error}
% \tcc{Compute time step increment based on user-specified tolerance}
% $s_{\bu} = \text{TOL}/(e_{\text{loc}} \Delta t_{n-1} )$ 

\tcc{Make sure the next time step $\Delta t_n$ is bounded by $s_\text{max} \Delta t_{n-1}$}
$s = \min{ \l( s_\text{cfl},  s_\text{max} \r) }$

$\Delta t_{n} = s \Delta t_{n-1}$

\If{$s < s_\text{min}$}{Repeat previous time step with $\Delta t_{n}$ instead}

\Return{Time step $\Delta t_{n}$ and flag whether to repeat the time step or not}
\end{algorithm}

\section{Numerical validation.} \label{Sec:gresho_sections}
The goal of this section is to numerically compare the formulations from Table \ref{table:formulations_of_interest}. The table includes newly developed formulations and formulations previously known from the literature. Another goal is to numerically verify the viscous regularization from Section \ref{Sec:viscous_regularization}. The nonlinear system is solved with the relative tolerance of $10^{-14}$. Our code is implemented in FEniCS \citep{fenics2015}, an open-source finite element library. We perform experiments using high-order Lagrange elements in space. More specifically, the SI-MEDMAC uses $(\rho, \bu, P)\in\commentA{  \polP_2 \CROSS \polP_3 \CROSS \polP_2 }$  and the other formulations uses $(\rho, \bu, P)\in\commentA{  \polP_3 \CROSS \polP_3 \CROSS \polP_2 }$. \commentC{In all experiments, we use $\alpha_{\uu} = 1$.}

\newcommand \greshoFigWidth{.5}
\newcommand \errorScale{.5}
\newcommand \greshoScale{.56}
\newcommand \legendScale{0.9}

% In time a constant time-step $\Delta t = 0.025 \sqrt{2} \min_\Omega(h) $ is used which corresponds to CFL = 0.025, since $\max_{(\Omega \times [0,T])} \| \bu \|_{l^2} = \sqrt{2}$. A small time-step is chosen so that the error in time is negligible.

\subsection{Accuracy test.} \label{Sec:accuracy test}
In this section, we verify the accuracy of the proposed formulations using a manufactured solution on a unit disk. The time step is computed using CFL = 0.025 and the initial time step is used throughout the whole simulation. A small time step is chosen so that the error in time is negligible. Since the error will be dominated by spatial errors, the expected convergence rate is then 4 for velocity and 3 for pressure. The expected rate for density is 4 for all formulations except SI-MEDMAC where we expect a rate of 3. We follow the setup from \citep{Guermond_Salgado_2009} where the forcing function $\bef$ is chosen to obtain the following exact solution
\begin{equation}
\begin{split}
&\rho(\bx,t) = 2 + x \cos( \sin(t) ) + y \sin ( \sin(t) ),\\
&\bu(\bx,t) =  \begin{bmatrix}
-y \cos(t) \\
x \cos(t)
\end{bmatrix}, \\
&p(\bx,t) = \sin(x) \sin(y) \sin(t), \\
& P(\bx, t) = p - \alpha_\rho \rho \bu \SCAL \bu - \frac{1}{2} \alpha_\rho \Bar{\rho} \bu \SCAL \bu - \frac{1}{ |\Omega| } \int_\Omega  \l( p - \alpha_\rho \rho \bu \SCAL \bu - \frac{1}{2} \alpha_\rho \Bar{\rho} \bu \SCAL \bu \r)\ud \bx.
\end{split}
\end{equation}
Note that we have included the modified pressure $P$ above which is the one that is solved for in \eqref{eq:full_time_disctetization}. \commentA{Dirichlet boundary conditions are imposed for the velocity. The boundary conditions are imposed strongly through the linear system using the routine implemented in FEniCS \citep{fenics2015}.} We perform a convergence study using a series of unstructured meshes. The dynamic viscosity is set to $\mu = 0.01$ and we set $\kappa = 0$. The termination time is set to $T = 1$.

The convergence results are presented in Figure \ref{fig:manufactured} where the $L^2$-error is plotted against the number of degrees of freedom (DOF). All the errors are computed using a high-order quadrature and are relative, \ie they are normalized with their corresponding norm. Overall, the velocity errors converge with the expected convergence rate 4 and all formulations have similar errors. Similarly, the pressure errors converge with the expected rate of 3 with similar errors for all formulations. However, for density, the convergence rate drops to around 3.5 for all formulations except the LSI-EMAC formulation. The reason for this is that the added term in the density update $\alpha_\rho ( \DIV \bu) ( \rho - \Bar{\rho} )$ pollutes the accuracy of $\rho$. If $\bu$ has an error $\mathcal{O}(h^4)$, then $\DIV \bu = \mathcal{O}(h^3)$. This means that a $\mathcal{O}(h^3)$ perturbation has been added to the density update and this explains why density converges with a suboptimal rate of $ \approx 3.5$ instead of 4. Even though the SI-MEDMAC uses $\rho \in \polP_2$, we observe that it converges with the superoptimal rate $\approx 3.5 $. SI-MEDMAC is the most accurate formulation for density even though it uses a lower polynomial degree.

 Another important observation is that the shift-invariant formulations are more accurate than their counterpart, \ie SI-EDMAC is more accurate than EDMAC, and SI-MEMAC is more accurate than MEMAC. More specifically, the error for density is roughly 1 order of magnitude lower for SI compared to the non-SI formulations.

\begin{figure}[H]
\begin{subfigure}{\greshoFigWidth \linewidth}
\centering 
\includegraphics[scale=\errorScale]{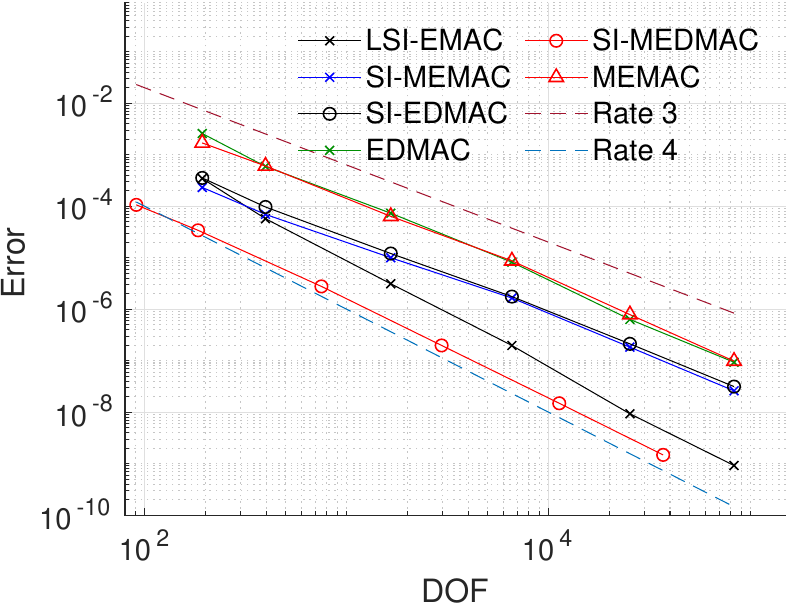}
\caption{Density}
\end{subfigure}% 
\begin{subfigure}{\greshoFigWidth \linewidth} 
\centering
\includegraphics[scale=\errorScale]{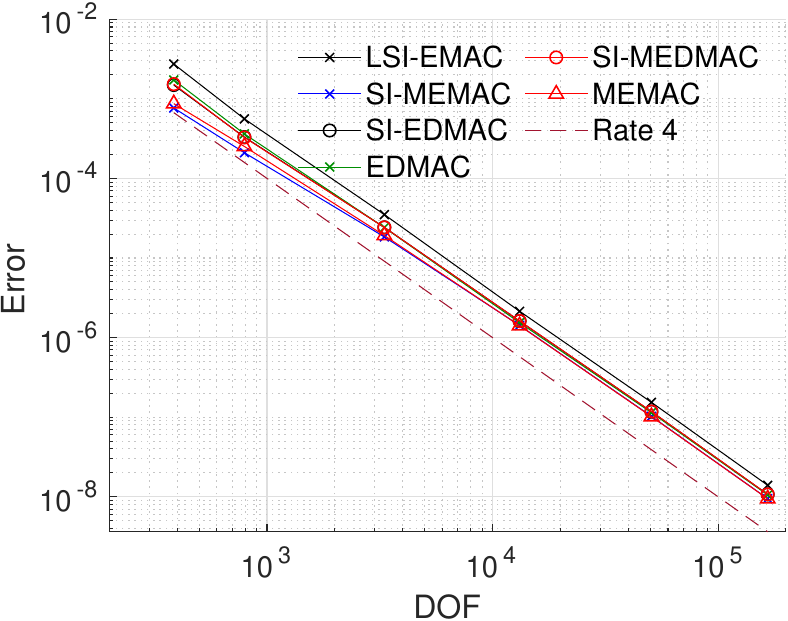}
\caption{Velocity}
\end{subfigure} \\ \centering
\begin{subfigure}{\greshoFigWidth \linewidth}
\centering 
\includegraphics[scale=\errorScale]{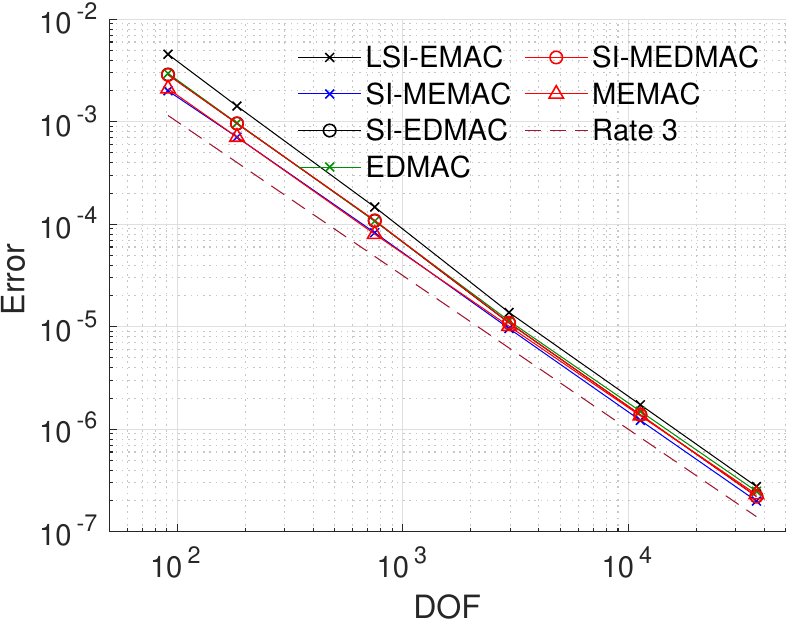}
\caption{Modified pressure}
\end{subfigure}
% \begin{subfigure}{\greshoFigWidth \linewidth}
% \centering 
% \includegraphics[scale=\greshoScale]{Figures/manufactured/error_pressure_modified-eps-converted-to.pdf}
% \caption{Modified pressure}
% \end{subfigure}% 
\caption{Manufactured solution. Relative $L^2$-error versus degrees of freedom. Convergence study for $T=1$, $\mu=0.01$, CFL = 0.025. All formulations used $(\rho, \bu, P)\in\commentA{  \polP_3 \CROSS \polP_3 \CROSS \polP_2 }$, except SI-MEDMAC which used $(\rho, \bu, P)\in\commentA{  \polP_2 \CROSS \polP_3 \CROSS \polP_2 }$. }
\label{fig:manufactured}
\end{figure}

\subsection{Gresho problem with variable density.} \label{Sec:gresho}
In this section, we numerically verify the conserved properties of the formulations under investigation (see Table \ref{table:formulations_of_interest}) for the inviscid case. To this end, we consider the so-called Gresho problem \commentA{on the domain, $\Omega = \{ (x,y) \in (-0.5 , -0.5) \CROSS (0.5, 0.5) \} $. The problem was originally} conceived for constant density incompressible flow \citep{Gresho1990} and we modify the problem so that density is variable and obtain the following exact solution
\begin{equation} \label{eq:gresho} 
\begin{split}
&\rho = 5 + 0.1  \l( 1 - \tanh \l(  \frac{x^2  + y^2}{r_0^2} - 1\r) \r), \\
&r \leq 0.2 : \begin{dcases}
\bu = \begin{pmatrix}
-5y \\ 5x
\end{pmatrix} \\
p = 12.5r^2 + C_1
\end{dcases}, \\
&r > 0.4 : \begin{dcases}
\bu = \begin{pmatrix}
0 \\ 0
\end{pmatrix} \\
p = 0
\end{dcases},
\\
0.2 \leq & r \leq 0.4 : \begin{dcases}
\bu = \begin{pmatrix}
\frac{-2y}{r} + 5y \\
\frac{2x}{r} -5x
\end{pmatrix} \\
p = 12.5r^2 -20r +4 \log(r) + C_2
\end{dcases},
\end{split}
\end{equation}
where $r = \sqrt{x^2 + y^2}$ and we set $r_0 = 1/8$ and
\begin{equation}
C_2 = (-12.5)(0.4)^2 + 20(0.4)^2 - 4 \log(0.4), C_1 = C_2 - 20 (0.2) + 4 \log(0.2).
\end{equation}
The main difficulty with this benchmark is that velocity is non-smooth at $r = 0.2  $ and $r = 0.4$, which will make $\DIV \bu \neq 0$ even when $h \rightarrow 0$. This means that the problem will test if the numerical scheme can conserve the properties considered in this work even when the divergence-free condition is severely violated. To be able to run the simulations longer without code crashes due to positivity issues, the density profile was chosen to be smooth.  It is important to emphasize that a stabilized scheme for the incompressible variable density Euler equations would be able to solve this problem without numerical instabilities. Since the purpose of this test is to verify Table \ref{table:emac_mom} numerically, no stabilization is used.  We use a coarse mesh (4705 $\polP_3$ and 2113 $\polP_2$ nodes). The time step is computed using CFL = 0.5 and the initial time step is used throughout the whole simulation. \commentA{Periodic boundary conditions are enforced. The periodic boundary mapping is done via built-in functions in FEniCS \citep{fenics2015}.}%The slip boundary condition is used on all boundaries\commentA{ which are imposed strongly through the linear system.}

% In space $\commentA{  \polP_3 \CROSS \polP_3 \CROSS \polP_2 }$ are used for all formulations except the SI-MEDMAC formulation where we use $\commentA{  \polP_2 \CROSS \polP_3 \CROSS \polP_2 }$ instead.

The results are presented in Figure \ref{fig:gresho} and show that the \commentA{proved} properties agree with the numerical results. The minimum of $\rho$ is computed node-wise and the $L^2$-errors are not normalized. More specifically, we observe that kinetic energy, momentum and angular momentum are conserved for all EMAC formulations. \commentC{There are small violations of momentum and angular momentum conservation. This is due to cancellation errors and because $\alpha_{\uu} = 1$. We note that the errors} are similar to what has been obtained by other researchers for the constant density Gresho problem, \eg \citep[Sec 4.1]{Charnyi2017} and  \citep[Sec 6.2]{Ingimarson2023}. In Figure \ref{fig:gresho}, by `\textit{Squared density*}' we refer to the modified energy of the advection equation given by $  \int_\Omega   \rho^2\ud \bx  - \frac{1}{ | \Omega | } \l(  \int_\Omega \rho  \ud\bx  \r)^2 $.

One important observation is that the shift-invariant property plays a big role, where the shift-invariant formulations are the most robust. For instance, compare SI-MEMAC and SI-EDMAC to MEMAC and EDMAC. The SI-EDMAC and SI-MEDMAC formulations are the most robust and the only formulations that avoided code crash. The SI-MEDMAC formulation is the most accurate for density even though it used fewer degrees of freedom, \ie 2113 $\polP_2$ nodes compared to 4705 $\polP_3$ nodes.

% \newcommand \greshoFigWidth{.5}
% \newcommand \greshoScale{.47}

% \lukas{explain rho energy *. Perhaps change name on y axis.}

% \begin{figure}[H]
% \begin{subfigure}{.47 \linewidth} 
% \centering
% \includegraphics[scale=.5]{Figures/gresho/su-pc/gresho_analytic_velocity-eps-converted-to.pdf}
% \caption{Velocity}
% \end{subfigure} 
% \begin{subfigure}{.5 \linewidth}
% \centering 
% \includegraphics[scale=.5]{Figures/gresho/su-pc/gresho_analytic_density-eps-converted-to.pdf}
% \caption{Density}
% \end{subfigure}% 
% \caption{Analytic solution to the Gresho problem with variable density. (a) The second component of velocity at $y=0$. (b) Density profile at $y = 0$.}
% \label{fig:gresho_analytic}
% \end{figure}

\begin{figure}[H]
\begin{subfigure}{\greshoFigWidth \linewidth}
\centering 
\includegraphics[scale=\greshoScale]{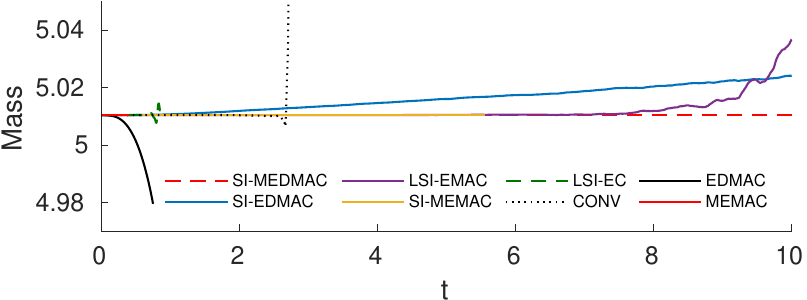}
% \caption{}
\end{subfigure}% 
\begin{subfigure}{\greshoFigWidth \linewidth} 
\centering
\includegraphics[scale=\greshoScale]{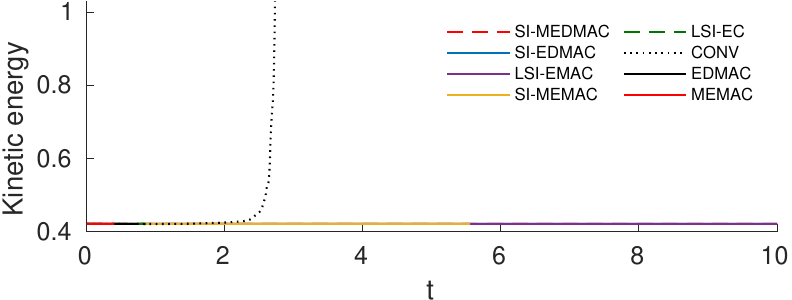}
% \caption{}
\end{subfigure} \\
\begin{subfigure}{\greshoFigWidth \linewidth}
\centering 
\includegraphics[scale=\greshoScale]{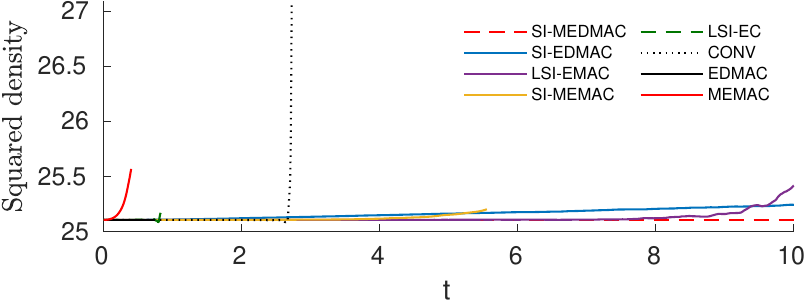}
% \caption{}
\end{subfigure}% 
\begin{subfigure}{\greshoFigWidth \linewidth} 
\centering
\includegraphics[scale=\greshoScale]{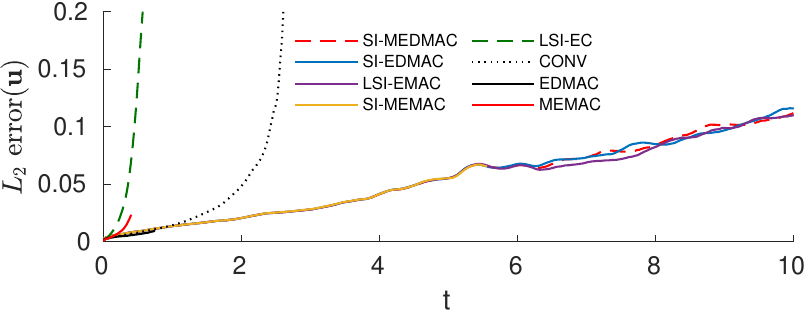}
% \caption{}
\end{subfigure} \\
\begin{subfigure}{\greshoFigWidth \linewidth}
\centering 
\includegraphics[scale=\greshoScale]{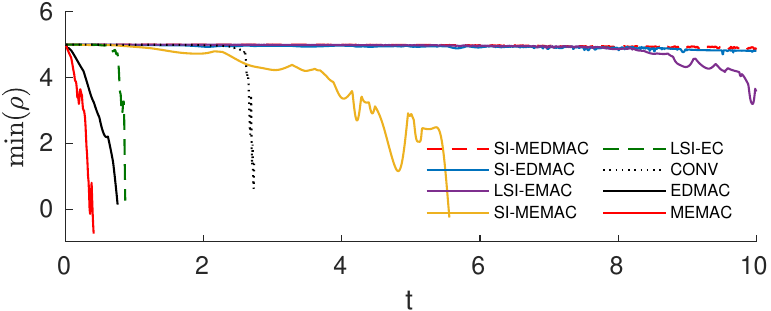}
% \caption{}
\end{subfigure}% 
\begin{subfigure}{\greshoFigWidth \linewidth} 
\centering
\includegraphics[scale=\greshoScale]{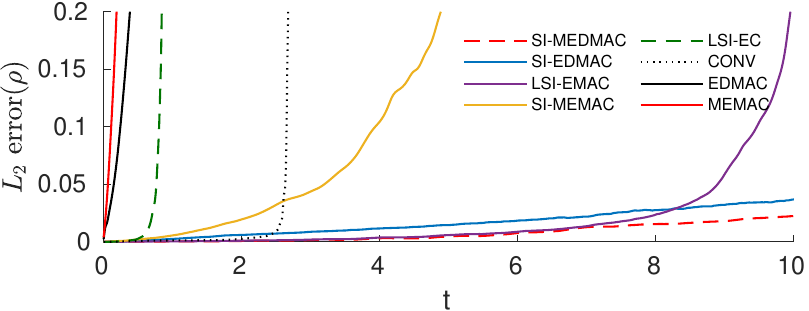}
% \caption{}
\end{subfigure} \\
\begin{subfigure}{\greshoFigWidth \linewidth}
\centering 
\includegraphics[scale=\greshoScale]{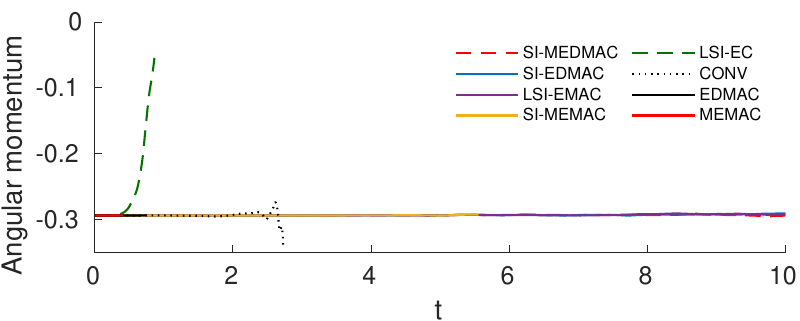}
% \caption{}
\end{subfigure}% 
\begin{subfigure}{\greshoFigWidth \linewidth} 
\centering
\includegraphics[scale=\greshoScale]{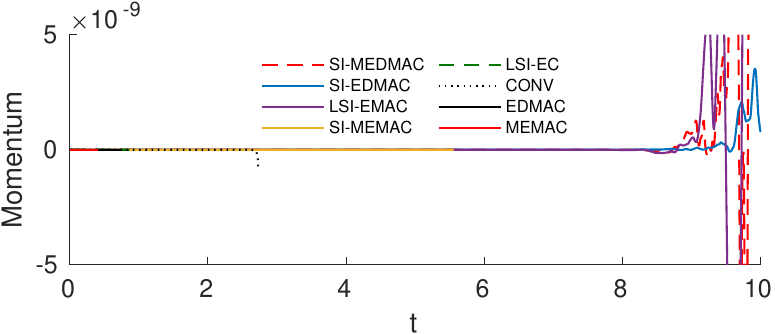}
% \caption{}
\end{subfigure}\\ 
\begin{subfigure}{\greshoFigWidth \linewidth} 
\centering
\includegraphics[scale=\greshoScale]{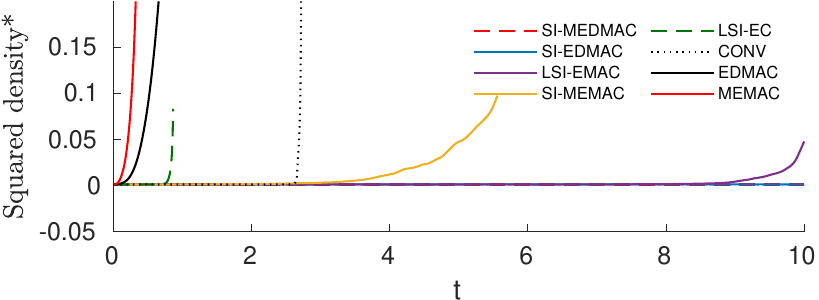}
% \caption{}
\end{subfigure}
\begin{subfigure}{ \greshoFigWidth \linewidth} 
 \raggedright \hspace*{50pt} \includegraphics[scale=\legendScale]{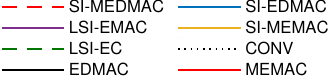}
% \caption{}
\end{subfigure}
\caption{\commentC{Gresho problem with variable density and periodic boundary conditions using $\text{CFL} = 0.5$ and 4705 $\polP_3$ nodes. $\kappa = \nu = \mu = 0$. No stabilization is used. All formulations used $(\rho, \bu, P)\in\commentA{  \polP_3 \CROSS \polP_3 \CROSS \polP_2 }$, except SI-MEDMAC which used $(\rho, \bu, P)\in\commentA{  \polP_2 \CROSS \polP_3 \CROSS \polP_2 }$.} }
\label{fig:gresho}
\end{figure}

\subsection{Viscous test.}
In this test, we numerically verify the properties of the viscous regularization. Since we are only interested in the properties of the model and not the discretization, we modify the Gresho problem from Section \ref{Sec:gresho} so that $\bu$ is smooth, meaning that $\DIV \bu \rightarrow 0$ as $h\rightarrow 0$. This means that for a fine mesh, the divergence-free condition will not affect the conserved properties much. The only factor that will meaningfully affect the computed properties is viscous regularization. We consider the following initial condition
\begin{equation} \label{eq:smooth_gresho}
\begin{split}
&\rho = 1 + 0.5  \l( 1 - \tanh \l(  \frac{x^2  + y^2}{r_0^2} - 1\r) \r), \\
& \bu = 0.1  \l( 1 - \tanh \l(  \frac{x^2  + y^2}{r_0^2} - 1 \r) \r) \begin{pmatrix}
-5y \\ 5x
\end{pmatrix},
\end{split}
\end{equation} 
on the domain $\Omega = \{ (x,y) \in (-1.25 , -1.25) \CROSS (1.25, 1.25) \} $  using periodic boundary conditions. The initial condition \eqref{eq:smooth_gresho} is an exact solution iff $\kappa = \mu = 0$. For $\mu \neq 0$ or $\kappa \neq 0$, the analytic solution is not known. We are interested in the behavior of the solution when $\kappa > 0$ and $\mu = 0$ since we want to verify Theorem \ref{theorem:momentum-viscous-flux}. If $\mu = 0$, the first condition of Theorem \ref{theorem:momentum-viscous-flux} is equivalent to kinetic energy conservation. 

We set $r_0 = 0.125$, $\kappa = 0.01$, $\mu = 0$. The time step is computed using CFL = 0.2 and the initial time step is used throughout the whole simulation. We use a fine mesh (73728 $\polP_3$ nodes) so that the discretization error is small since we are only interested in the viscous contribution. Since the discretization error is small, the formulations in Table \ref{table:formulations_of_interest} will produce very similar results. We use the SI-MEDMAC formulation. 

The results are presented in Figure \ref{fig:gresho-smooth} using three different viscous regularizations \eqref{eq:guermond_popov_flux}-\eqref{eq:GP_S}. The results agree with the theory, \ie 

\begin{itemize}
\item $\bef_{GP}$ conserves kinetic energy and momentum but not angular momentum.
\item \new{$\bef_{KS}$} and $\bef_0$ conserve momentum and angular momentum but not kinetic energy.
\end{itemize}

  %In the Figure, these are named (GP) which corresponds to $\bef_{GP}$, (ANG) which corresponds to $\bef_{KS}$ and (empty) which corresponds to $\bef_{0}$.

% $\Omega = 2.5*(-0.5 , -0.5) \CROSS (0.5, 0.5)$
% \lukas{Perhaps remove some of the figures in this section since not all are relevant... What do you think?}

% \lukas{Should I plot velocity here as well. Similar to Figure \ref{fig:gresho_analytic}?}

\begin{figure}[H]
\begin{subfigure}{\greshoFigWidth \linewidth}
\centering 
\includegraphics[scale=\greshoScale]{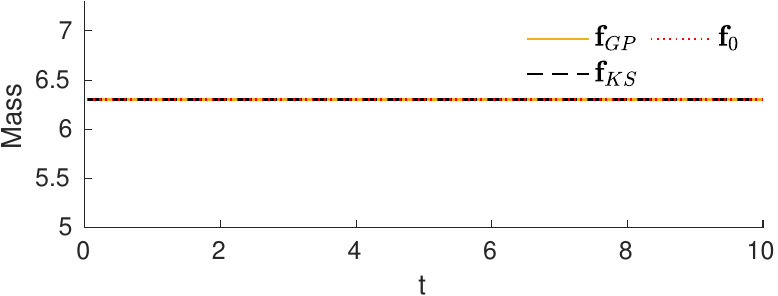}
% \caption{}
\end{subfigure}% 
\begin{subfigure}{\greshoFigWidth \linewidth} 
\centering
\includegraphics[scale=\greshoScale]{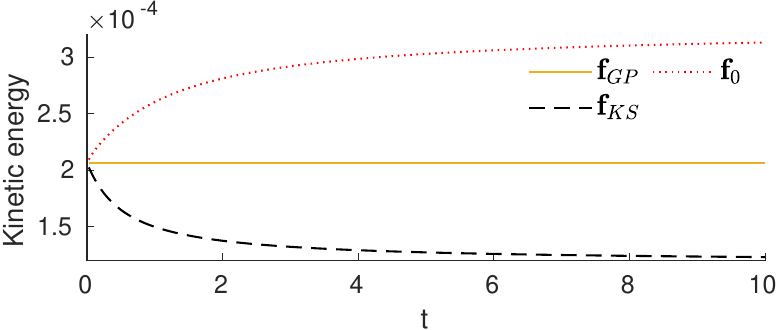}
% \caption{}
\end{subfigure} \\
\begin{subfigure}{\greshoFigWidth \linewidth}
\centering 
\includegraphics[scale=\greshoScale]{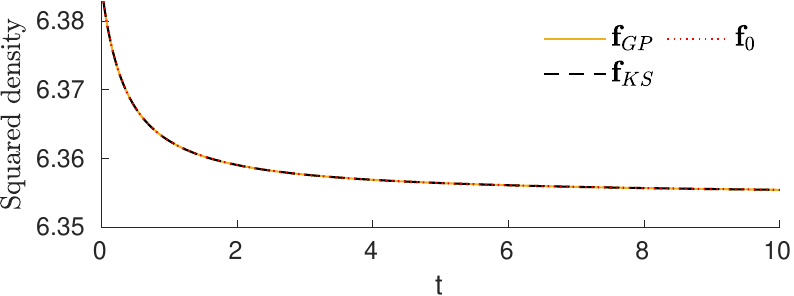}
% \caption{}
\end{subfigure}% 
\begin{subfigure}{\greshoFigWidth \linewidth} 
\centering
\includegraphics[scale=\greshoScale]{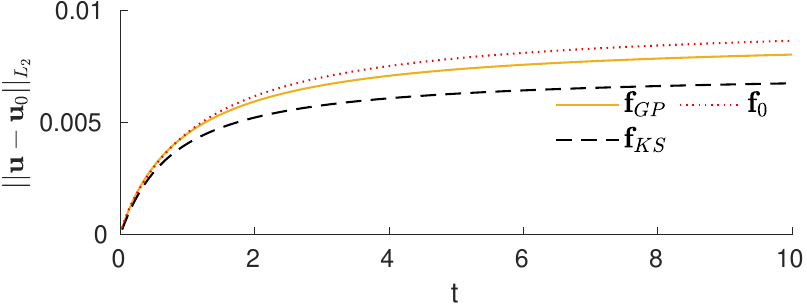}
% \caption{}
\end{subfigure} \\
\begin{subfigure}{\greshoFigWidth \linewidth}
\centering 
\includegraphics[scale=\greshoScale]{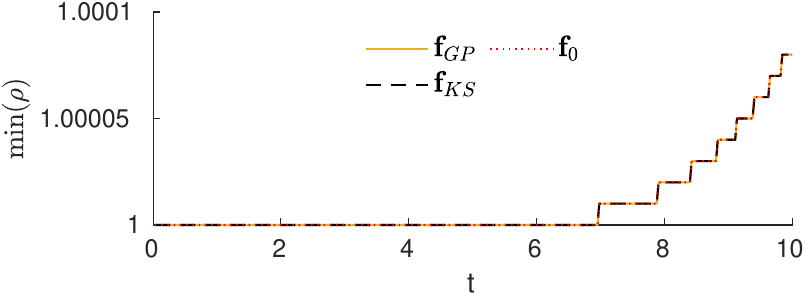}
% \caption{}
\end{subfigure}% 
\begin{subfigure}{\greshoFigWidth \linewidth} 
\centering
\includegraphics[scale=\greshoScale]{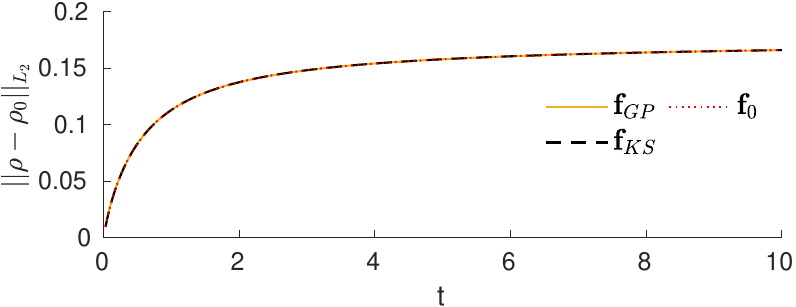}
% \caption{}
\end{subfigure} \\
\begin{subfigure}{\greshoFigWidth \linewidth}
\centering 
\includegraphics[scale=\greshoScale]{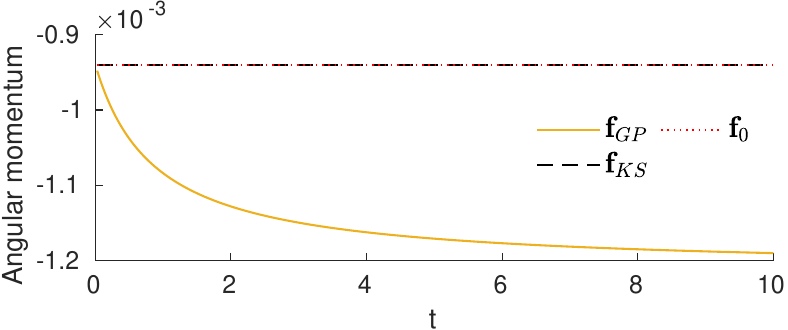}
% \caption{}
\end{subfigure}% 
\begin{subfigure}{\greshoFigWidth \linewidth} 
\centering
\includegraphics[scale=\greshoScale]{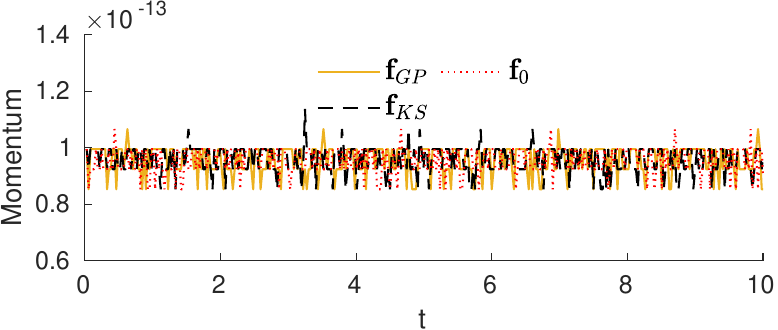}
% \caption{}
\end{subfigure}
\caption{\new{Comparison of viscous regularizations \eqref{eq:guermond_popov_flux}-\eqref{eq:GP_S} using $\mu = \nu = 0$, $\kappa = 0.01$ for a smooth variant of the Gresho problem with variable density and periodic boundary conditions. The solution was obtained using $\text{CFL} = 0.2$, 73728 $\polP_3$ nodes and no stabilization was used. }}
\label{fig:gresho-smooth}
\end{figure}

\subsection{Lock-exchange 2D.}

In this section, we compare some different viscous regularizations for a more realistic problem. The problem we consider is the so-called lock-exchange problem which was \commentB{numerically} studied extensively by \cite{Bartholomew_2019} and \cite{Birman_2005}. \commentB{Notably, this configuration has also been the subject of physical experimental exploration by \cite{lowe_rottman_linden_2005}, yielding observations that align with the findings put forth by \citep{Birman_2005}.} The reason why this problem is of interest to us is that it is one of the few case studies in the variable density flow literature which includes mass diffusivity in the model, \ie $\kappa > 0$. The classical setup is a heavy fluid with density $\rho_1$ and a lighter fluid with density $\rho_2$ separated by a vertical barrier. The heavy fluid is located to the left of the barrier and the lighter fluid is to the right. The barrier is removed at $t = 0$ and then the heavy gas moves to the right and the lighter fluid moves to the left. Following the setup from \citep{Bartholomew_2019,Birman_2005}, the computational domain is set to $\Omega = \{ (x,y) \in \commentD{(0, 32L)  \CROSS  (-L/2, L/2)} \} $ with a characteristic velocity scale set to $u = \sqrt{ g( \rho_{1} - \rho_2 )/\rho_1 L }$ and the Reynolds number is defined as $Re = \rho_1 L^{3/2} \sqrt{ g( \rho_{1} - \rho_2 )/\rho_1  }  / \mu$. Following the reference we set $L=1$, $\rho_2 = 0.7$ and $\rho_1 = 1.0$. The Reynolds number is set to $4000$ yielding a constant kinematic viscosity coefficient $\nu$. In summary, we set $\mu = \rho \nu$. The forcing function is set to $\bef = (0, -\rho g)$ with $g= 1$ to yield a downward gravitational force. The initial condition is initially regularized using the error function and is given by
\begin{equation}
\rho^0(\bx) = \frac{1}{2} \l(  \frac{\rho_2}{\rho_1} + 1 \r) - \frac{1}{2} \l( 1 - \frac{\rho_2}{\rho_1} \r) \text{erf}\l(   x_0 \sqrt{Re}  \r),
\end{equation}
where $x_0 = 14$ is the location of the barrier at $t = 0$. Following the reference, the mass diffusivity coefficient is set to $\kappa = \nu$ which means that the so-called Schmidt number which governs the relation of molecular diffusivity and kinematic viscosity is set to 1. \commentA{Also following the reference, we set slip boundary conditions on all boundaries which are imposed strongly through the linear system. A Neumann condition $\bn \SCAL \GRAD \rho|_{\partial \Omega} = 0,$ is automatically imposed weakly on all boundaries through the weak form \eqref{eq:full_time_disctetization}.}

The computed densities are presented in Figures \ref{fig:le-empty} - \ref{fig:le-gp} at times $T = 1, 3, 5, 7, 10$ (at time-scale $\\L^{1/2} \l( g( \rho_{1} - \rho_2 )/\rho_1 \r)^{-1/2}$). The contour lines at the final time step are plotted in Figure \ref{fig:contour}. The results were computed using the SI-MEDMAC formulation using 1851121 $\polP_3$ nodes on a uniform mesh. The parameters for the time-stepping are set to $\text{CFL} = 0.1$, $s_{max} = 1.02$, $s_{min} = 0.98$. The results using the momentum and angular momentum conserving viscous flux $\bef_0$ match with the reference \citep[Fig 4]{Bartholomew_2019} since we use the same viscous flux on a fine mesh. In \ref{Sec:lock exchange 3D} we extend the benchmark to 3D. %However, \citep{Birman_2005} uses a different viscous flux ($\bef = 0$ but in primitive variables) and thus produces a different result.

% \lukas{perhaps compare again Birman viscous flux.}

\newcommand \widthLockGlob{0.75}

\newcommand \xWidthlock{ \fpeval{ 0.5 / \widthLockGlob  } }

\newcommand \Widthlockcolorbar{0.091}
\newcommand \Widthlockcolorbarres{\fpeval{1.375*\Widthlockcolorbar}}

% \newcommand \xWidthbig{.12}
% \newcommand \xWidth{.06}
% \newcommand \bubbleWidth{.3}

% \newcommand \comparisonWidth{2cm}

% %\newcommand \xWidthcolorbar{.1114}
% \newcommand \xWidthcolorbar{\fpeval{1.47*\xWidth}}
% \newcommand \xWidthcolorbarr{\fpeval{0.85*\xWidth}}

% \newcommand \widthBubbleColorbar{\fpeval{0.38*\xWidthbubble}}
% \newcommand \widthBubbleColorbarr{\fpeval{0.35*\xWidthbubble}}

\begin{figure}[H] 
\begin{tabular}{cc}
\begin{subfigure}{\widthLockGlob \textwidth } 
  \raggedleft
    \includegraphics[width=\xWidthlock \textwidth]{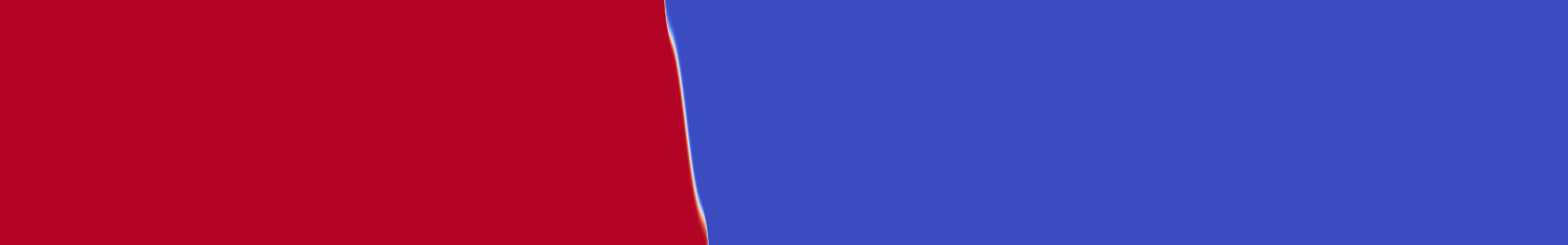} 
    \includegraphics[width=\xWidthlock \textwidth]{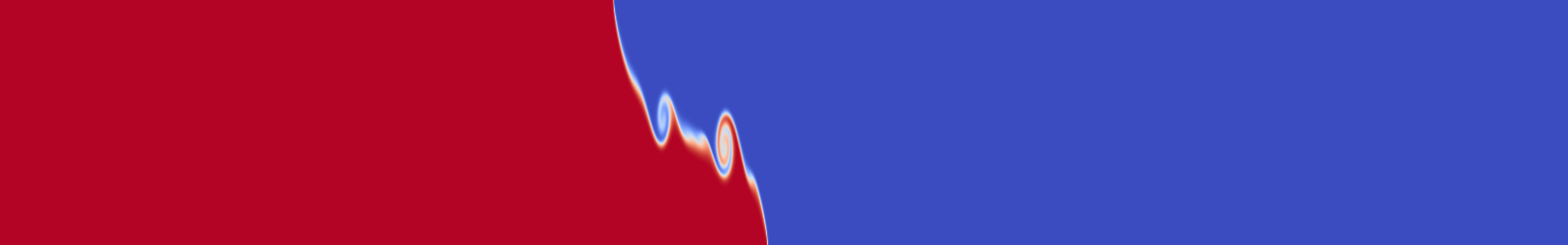} 
    \includegraphics[width=\xWidthlock \textwidth]{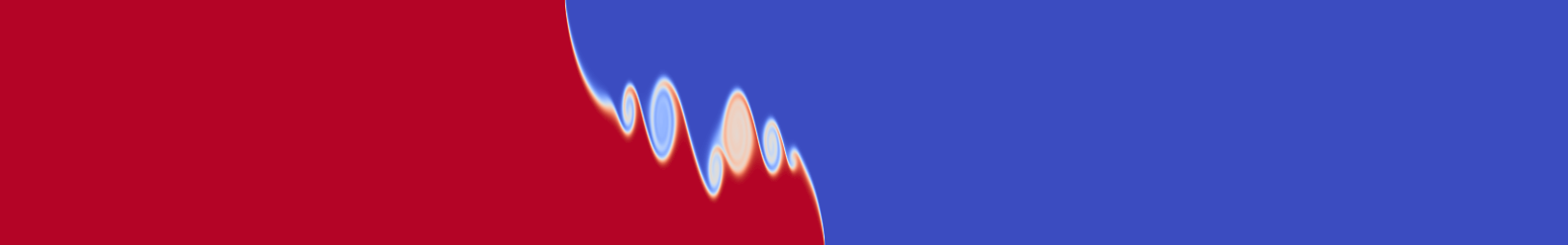} 
    \includegraphics[width=\xWidthlock \textwidth]{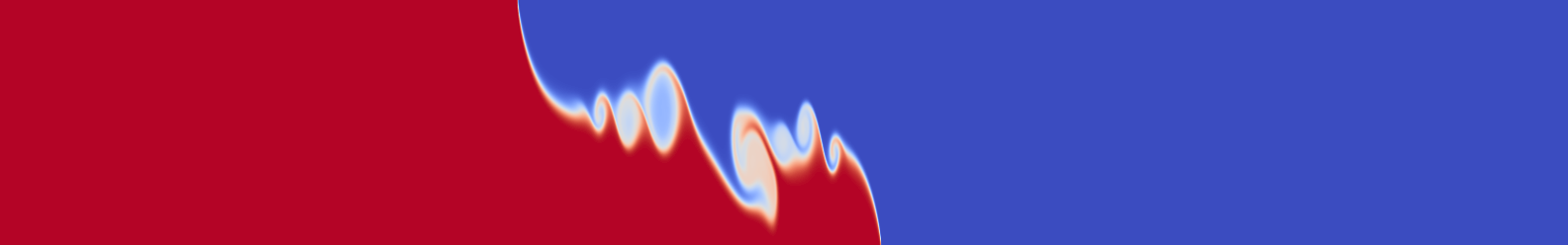} 
    \includegraphics[width=\xWidthlock \textwidth]{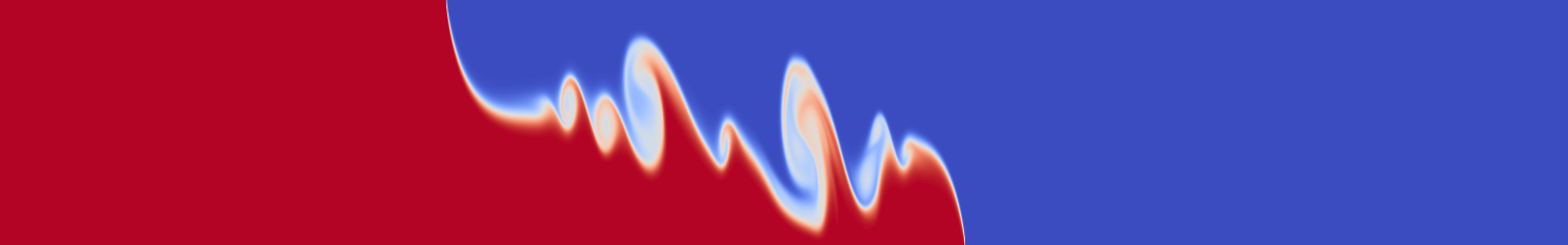} 
    \end{subfigure}
    & 
    \begin{subfigure}{0.5\textwidth}
    \includegraphics[width=\Widthlockcolorbar \textwidth]{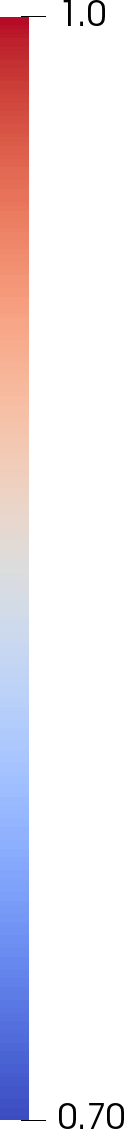} 
    \end{subfigure}
\end{tabular}
\caption{Momentum and angular momentum conserving viscous flux: $\bef_{0}$. Lock-exchange problem at $Re = 4000$, Schmidt number 1 and density ratio $1/0.7$. Density $\rho$ at times $t = 1, 3, 5, 7, 10$.}
\label{fig:le-empty}
\end{figure}

\begin{figure}[H] 
\begin{tabular}{cc}
\begin{subfigure}{\widthLockGlob \textwidth}
  \raggedleft
    \includegraphics[width=\xWidthlock \textwidth]{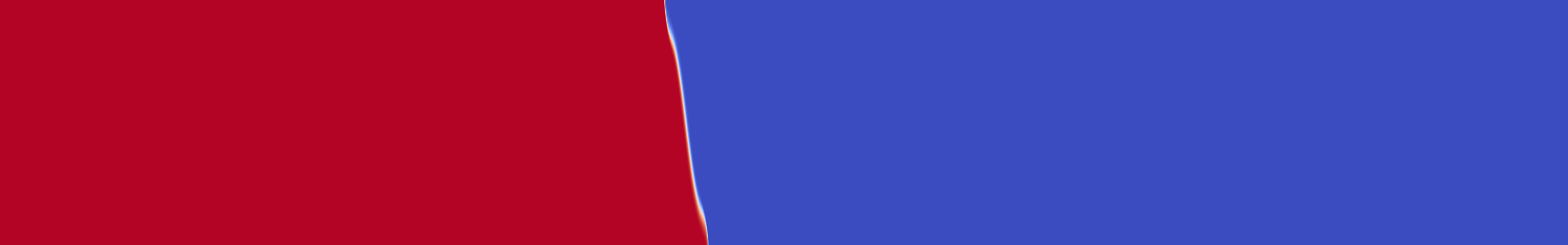} 
    \includegraphics[width=\xWidthlock \textwidth]{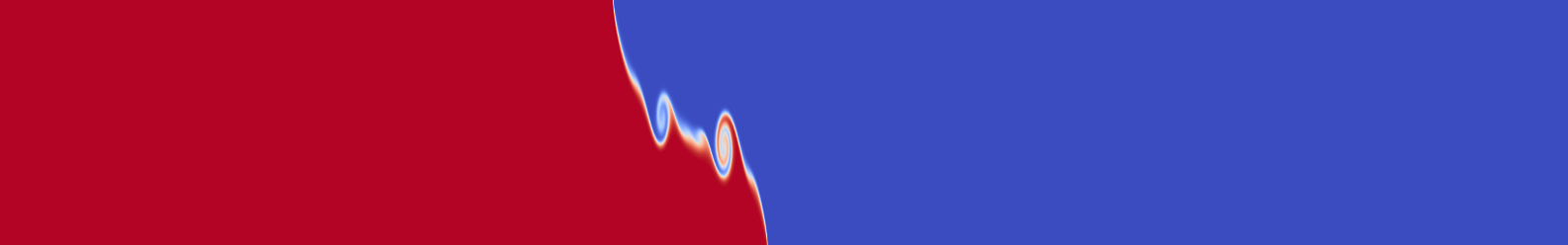} 
    \includegraphics[width=\xWidthlock \textwidth]{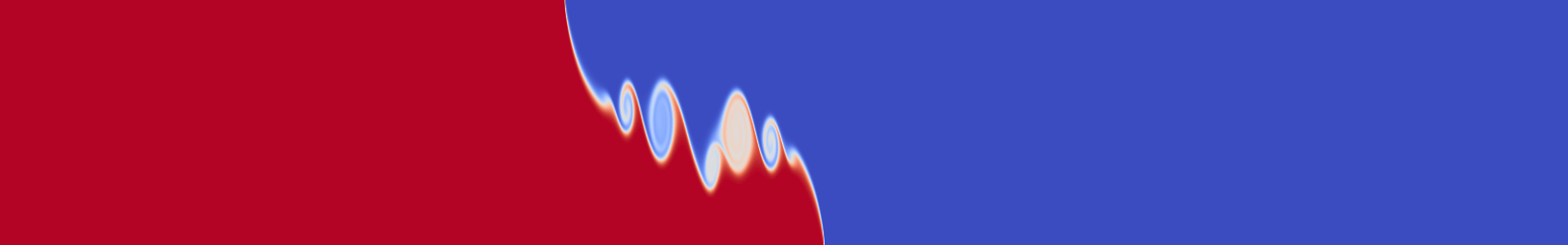} 
    \includegraphics[width=\xWidthlock \textwidth]{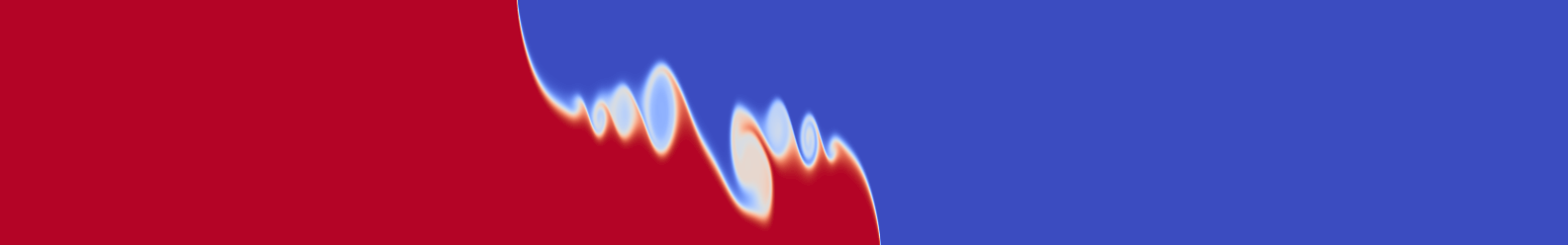} 
    \includegraphics[width=\xWidthlock \textwidth]{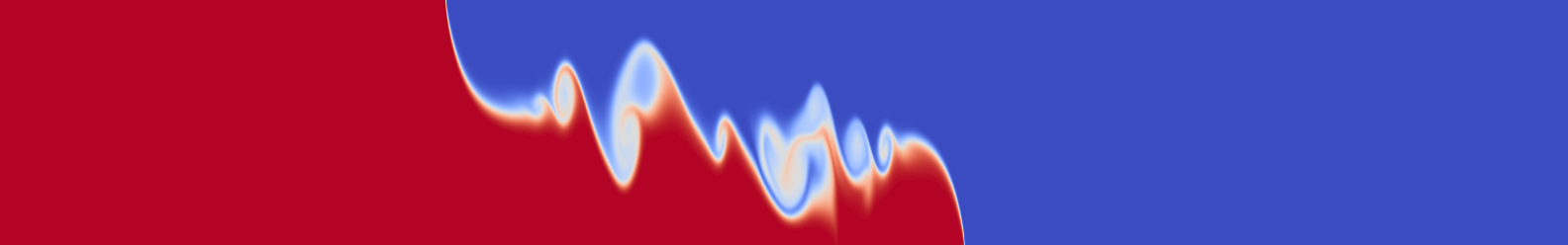} 
    \end{subfigure}
    & \raggedright
    \begin{subfigure}{0.5\textwidth}  
    \includegraphics[width=\Widthlockcolorbar \textwidth]{Figures/lock-exchange/le-n80-no-colorbar.png} 
    \end{subfigure}
\end{tabular}
\caption{Momentum and angular momentum conserving viscous flux: $\bef_{KS}$. Lock-exchange problem at $Re = 4000$, Schmidt number 1 and density ratio $1/0.7$. Density $\rho$ at times $t = 1, 3, 5, 7, 10$.}
\label{fig:le-ang}
\end{figure}

\begin{figure}[H] 
\begin{tabular}{cc}
\begin{subfigure}{\widthLockGlob \textwidth}
  \raggedleft
    \includegraphics[width=\xWidthlock \textwidth]{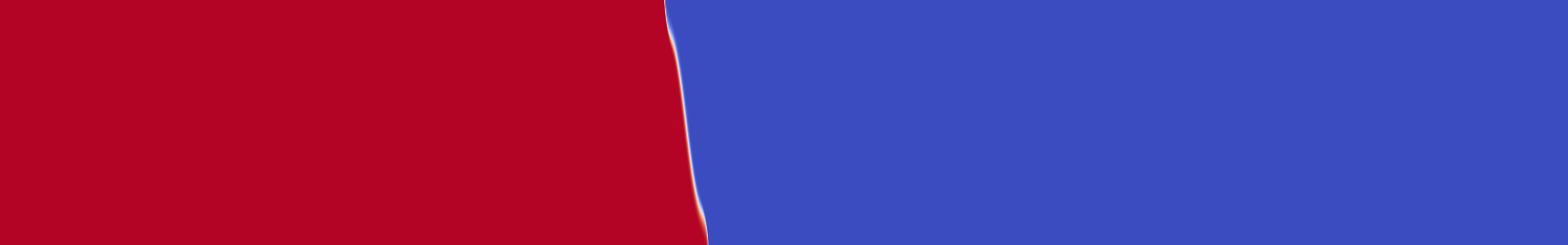} 
    \includegraphics[width=\xWidthlock \textwidth]{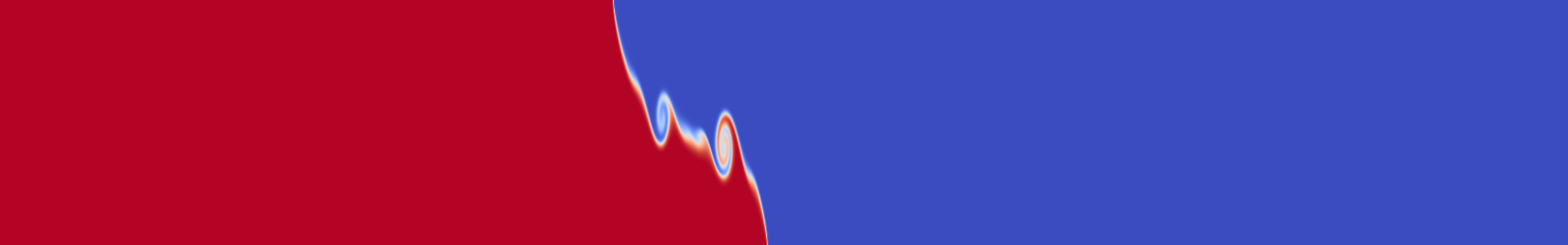} 
    \includegraphics[width=\xWidthlock \textwidth]{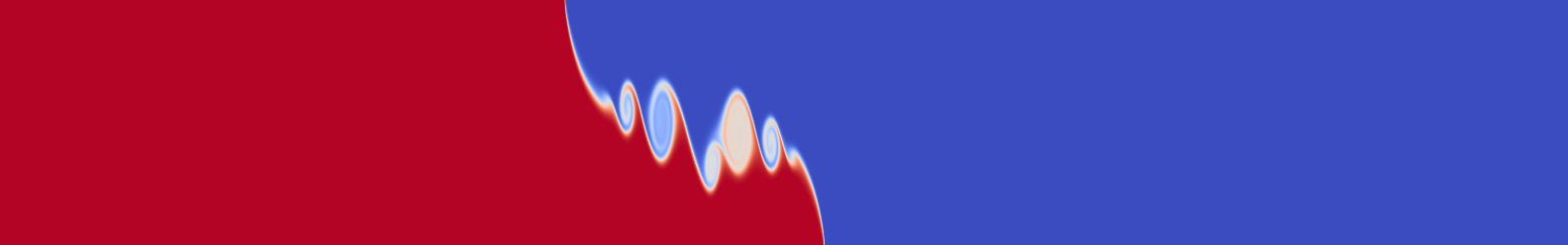} 
    \includegraphics[width=\xWidthlock \textwidth]{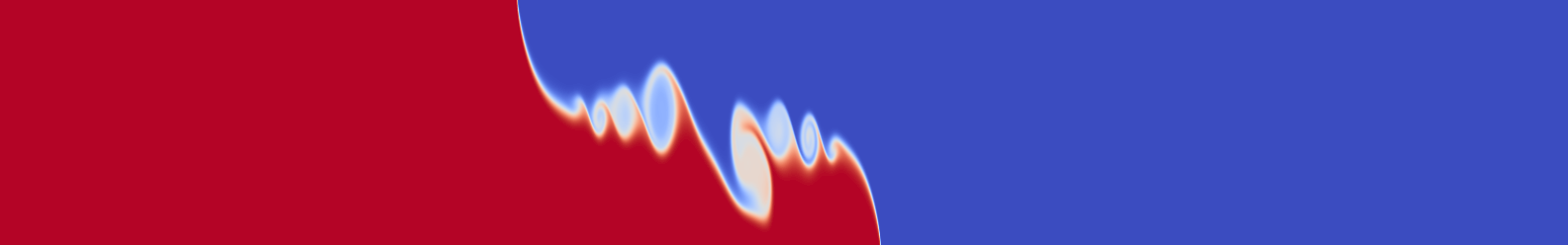} 
    \includegraphics[width=\xWidthlock \textwidth]{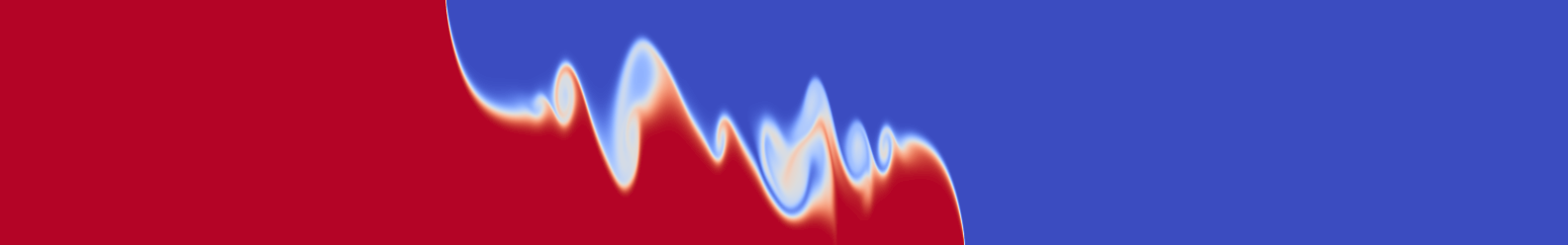} 
    \end{subfigure}
    & \raggedright
    \begin{subfigure}{0.5\textwidth}  
    \includegraphics[width=\Widthlockcolorbar \textwidth]{Figures/lock-exchange/le-n80-no-colorbar.png} 
    \end{subfigure}
\end{tabular}
\caption{Guermond-Popov viscous flux: $\bef_{GP}$. Lock-exchange problem at $Re = 4000$, Schmidt number 1 and density ratio $1/0.7$. Density $\rho$ at times $t = 1, 3, 5, 7, 10$.}
\label{fig:le-gp}
\end{figure}

\newcommand \widthLockGlobC{0.31}
\newcommand \xWidthlockC{ \fpeval{ 0.31 / \widthLockGlobC  } } 

\begin{figure}[H] 
\begin{tabular}{cccc}
\begin{subfigure}{\widthLockGlobC \textwidth}
    \includegraphics[width=\xWidthlockC \textwidth]{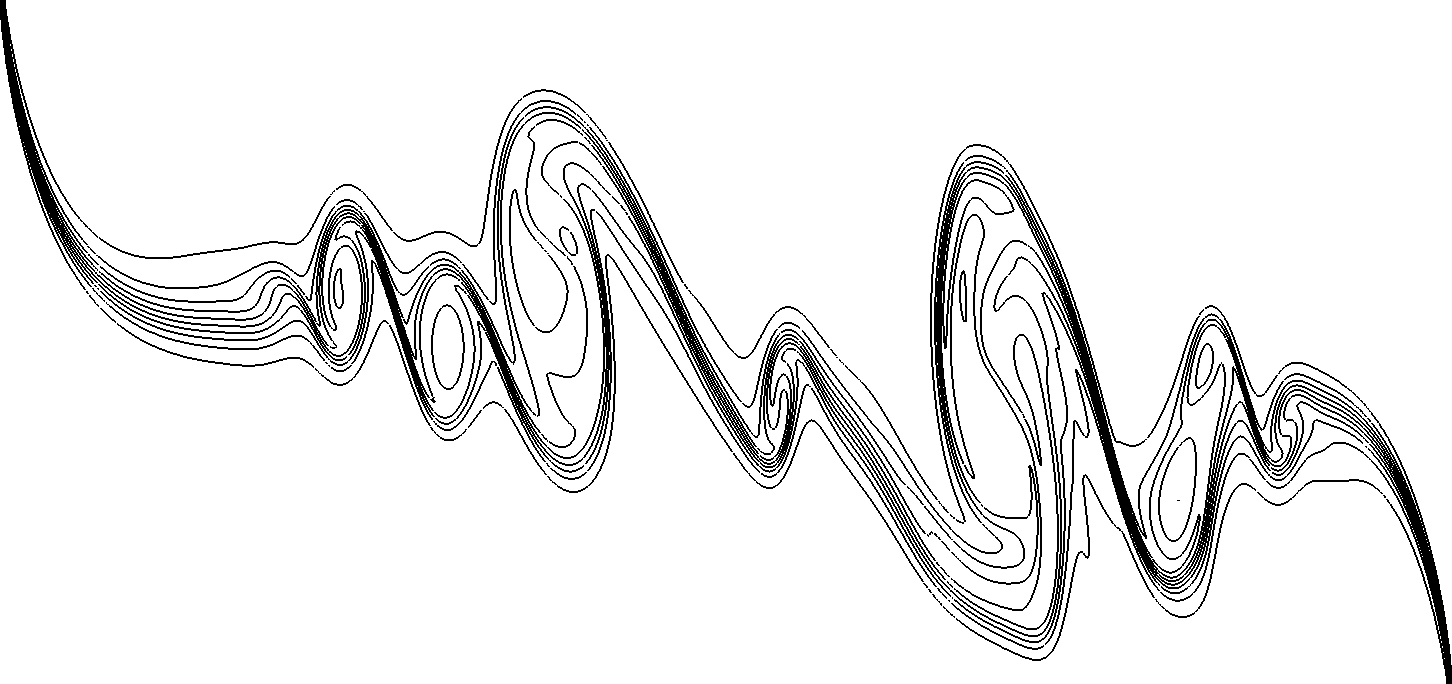} %no
    \caption{Momentum and angular momentum conserving viscous flux: $\bef_{0}$.}
    \end{subfigure}
    & 
    \begin{subfigure}{\widthLockGlobC \textwidth}  
    \includegraphics[width=\xWidthlockC \textwidth]{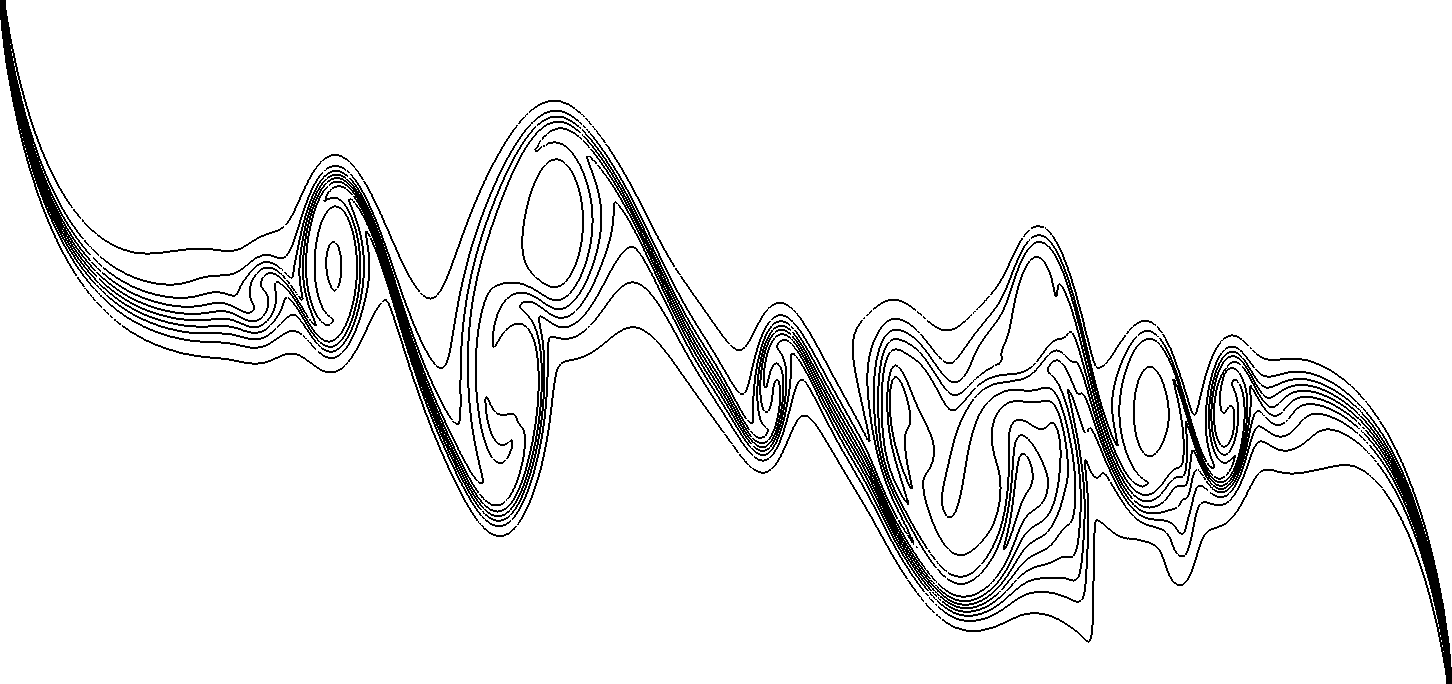} 
    \caption{\new{\cite{Kaz_Smag_1977} viscous model: $\bef_{KS}$.}}
    \end{subfigure}
    & 
    \begin{subfigure}{\widthLockGlobC \textwidth}  
    \includegraphics[width=\xWidthlockC \textwidth]{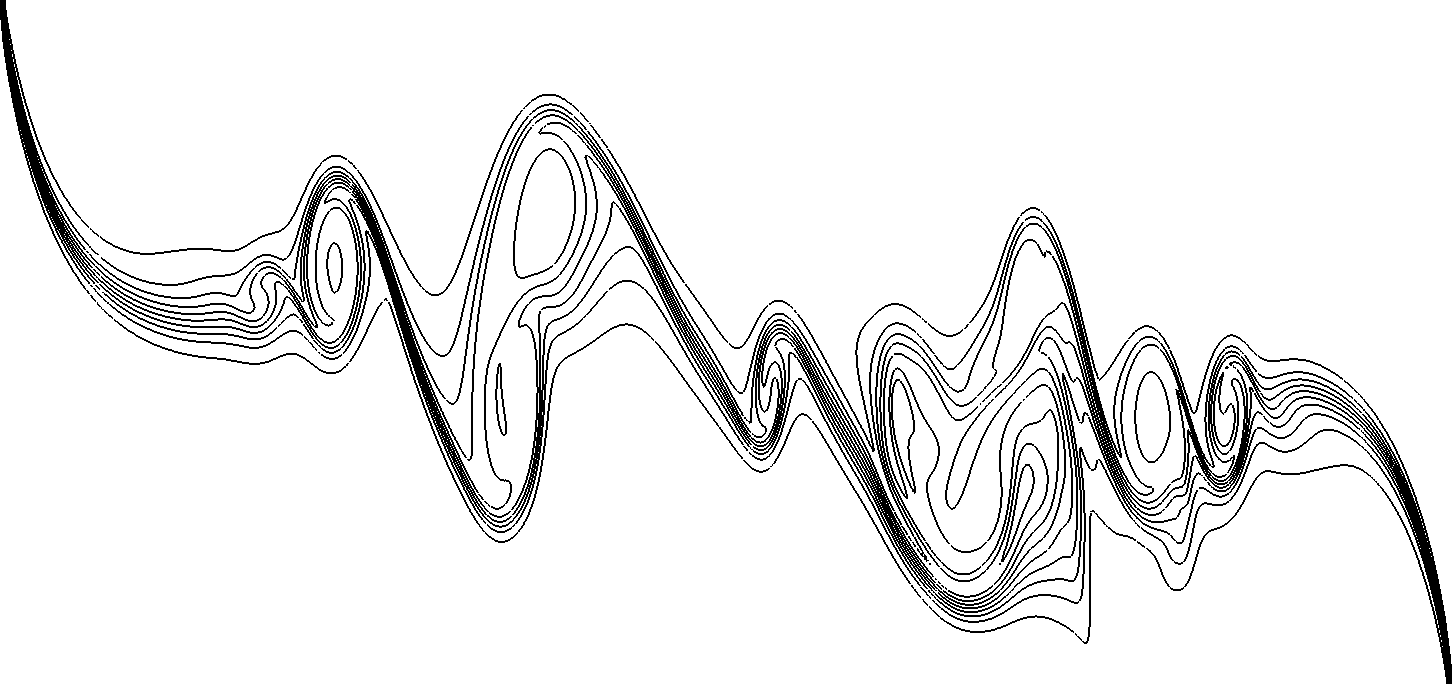} 
    \caption{Guermond-Popov viscous flux: $\bef_{GP}$.}
    \end{subfigure}
    %  & 
    % \begin{subfigure}{\widthLockGlobC \textwidth}  
    % \includegraphics[width=\xWidthlockC \textwidth]{Figures/lock-exchange/le-contour-n80-birman-10.png} 
    % \caption{Birman.}
    % \end{subfigure}
\end{tabular}
\caption{Lock-exchange problem at $Re = 4000$, Schmidt number 1 and density ratio $1/0.7$. Contour lines of density at $T = 10$ for three different viscous regularizations.}
\label{fig:contour}
\end{figure}

\section{Conclusion.} \label{Sec:conclusion}
\commentC{By consistently modifying the governing equations, we introduce a new formulation for the incompressible variable density Navier-Stokes equations. When discretized by a Galerkin method, the formulation allows the method to be shift-invariant and mass, squared density, kinetic energy, momentum and angular momentum conserving without the divergence-free constraint being strongly enforced.} The formulation we propose for the inviscid case \eqref{eq:cons_law_primitive} is
\new{\begin{equation}\label{eq:conclusion}
\begin{aligned}
        \p_t \rho + \bu \SCAL \GRAD \rho + \frac{1}{2}(\DIV \bu) (\rho - \Bar{\rho}) & = 0,\\
 \p_t \bbm + \bu \SCAL \GRAD  \bbm +  (\DIV \bu) \bbm  + \GRAD P  +      (\GRAD \bu) \bbm \\    + \frac{1}{2} \l(   \GRAD \rho  (\bu \SCAL \bu)   - \frac{1}{2} \GRAD (  (\rho - \Bar{\rho} ) \bu \SCAL \bu )  \r) 
     &= \bef,  
\end{aligned}
\end{equation}}
where $P = p - \frac{1}{4} \rho \bu \SCAL \bu - \frac{1}{4} \Bar{\rho} \bu \SCAL \bu $ is a modified pressure and $\Bar{\rho} = \frac{1}{| \Omega |} \int_\Omega \rho\ud \bx $. If $\rho = 1$, the formulation simplifies to the EMAC formulation by \cite{Charnyi2017}. \new{One important finding is that by ensuring that the polynomial degree of the mixed FEM formulation satisfies $k_\rho \leq k_P$, mass conservation is automatically achieved. Similarly, by setting $2 k_\rho \leq k_P$, squared density conservation can be automatically achieved.} We numerically compare the new formulation to some previously considered in the literature (see Table \ref{table:formulations_of_interest}). The new formulation performs the best on the benchmarks considered in this manuscript, both in terms of accuracy for smooth problems and in terms of robustness. 

Lastly, we consider the effect of viscous regularizations and our theoretical findings are summarized in Theorem \ref{theorem:momentum-viscous-flux}. One implication of Theorem \ref{theorem:momentum-viscous-flux} is that, if mass diffusion is included in the model, there is no viscous regularization that simultaneously dissipates kinetic energy and conserves angular momentum. Another implication is that a unique viscous regularization dissipates kinetic energy and conserves momentum.

% We show that when mass diffusion is added to the model, depending on which viscous regularization is used in the momentum equations, the model either conserves kinetic energy and momentum or momentum and angular momentum. There is no viscous regularization that conserves kinetic energy, momentum and angular momentum at the same time.

% This paper introduces several formulations of the variable density incompressible Navier-Stokes equations. The findings are summarized in Table \ref{table:emac_mom} and show that, when using a Galerkin method, improved conservation properties can be obtained even when $\DIV \bu \neq 0$ by modifying the nonlinear- and potential terms in a consistent way. By performing these modifications and choosing the polynomial degree of the density space to be less than or equal to the polynomial degree of the pressure space, it is possible to conserve mass, squared density, kinetic energy, momentum and angular momentum.

% We note that when $\DIV \bu$ is pointwise zero, all of the properties in Table \ref{table:emac_mom} are obtained by Galerkin methods regardless of which formulation is used.

As additional work, one could extend the formulation to more realistic multiphase flow models where surface tension is included such as Cahn-Hilliard or similar. One interesting research direction would be to consider positivity-preserving discretizations that are also shift-invariant. It would also be interesting to investigate which viscous regularization in Section \ref{Sec:viscous_regularization} gives the best match with experimental data from real-world experiments. For underresolved flows, \ie \commentB{$\kappa \ll |\bu| h$ or $\mu \ll \rho |\bu| h$,} we recommend that a stabilization technique \citep{Stiernstrom2021,Lundgren_2023,Nazarov_2013,Nazarov_Hoffman_2013} is used.

 % and also see how well the formulations performed together with stabilization

 % See how the formulations perform when combined with stabilization. Positivity preserving + shift-invariant.

\section*{Acknowledgments.}
The computations were enabled by resources in project SNIC 2022/22-428 provided by the Swedish National Infrastructure for Computing (SNIC) at UPPMAX, partially funded by the Swedish Research Council through grant agreement no. 2018-05973. The first author was partly supported by the Center for Interdisciplinary Mathematics, Uppsala University. \new{We extend our sincere gratitude to the anonymous reviewers whose constructive comments greatly increased the quality of this manuscript.}

\appendix

\section{Different formulations of the model problem} \label{Sec:alternative_formulations}

Since it is common to write the governing equations in primitive form, we reiterate the results of this manuscript for the governing equations below 

\begin{equation}\label{eq:nse_primitive_regularized}
      \begin{aligned}
        \p_t \rho + \bu \SCAL \GRAD \rho  & = \DIV   \bef_{\rho}  ,\\
 \rho \p_t \bu + \bbm \SCAL \GRAD  \bu  + \GRAD p  &= \bef + \DIV \l( \mu \l( \GRAD \bu + (\GRAD \bu)^\top \r)  \r)  + \bef_{\bbm} - (\DIV \bef_\rho) \bu , \quad &(\bx,t) \in  \Omega \CROSS (0,T],\\
           \DIV \bu &= 0,\\
           \bu(\bx,0) &= \bu_{0}(\bx), \\
           \rho(\bx,0) &= \rho_{0}(\bx), \quad  &\bx\in \Omega.  
      \end{aligned}
\end{equation}
Similarly, it is also common to write the time-derivative as $\sqrt{\rho} \p_t \l( \sqrt{\rho} \bu \r) $ \citep{Guermond2000} instead of $\p_t \bbm$ which leads to
\begin{equation}\label{eq:nse_sigma_regularized}
      \begin{aligned}
        \p_t \rho + \bu \SCAL \GRAD \rho   =& \DIV   \bef_{\rho}  ,\\
\sqrt{\rho} \p_t \l ( \sqrt{\rho} \bu \r) + \frac{1}{2} \bu \SCAL \GRAD  \bbm + \frac{1}{2} \bbm \SCAL \GRAD  \bu  + \GRAD p  =& \bef + \DIV \l( \mu \l( \GRAD \bu + (\GRAD \bu)^\top \r)  \r) \\  & + \bef_{\bbm} - \frac{1}{2} (\DIV \bef_\rho) \bu , \ &(\bx,t) \in  \Omega \CROSS (0,T],\\
           \DIV \bu =& 0,\\
           \bu(\bx,0) =& \bu_{0}(\bx), \\
           \rho(\bx,0) =& \rho_{0}(\bx),   &\bx\in \Omega.   
      \end{aligned}
\end{equation}
Many Navier-Stokes solvers use these formulations as their starting point and all the theoretical results of this manuscript directly translate to these forms provided that the formulations are modified accordingly. For instance, Theorem \ref{theorem:momentum-viscous-flux} holds for the above formulations since the term $ (\DIV \bef_\rho ) \bu $ is included. Additionally, all the conservativeness given when $\DIV \bu \neq 0$ holds provided the formulations are slightly modified, see the forthcoming sections.

\subsection{Velocity formulation.} 

If solving the problem in primitive form \eqref{eq:nse_primitive_regularized} is preferred we propose the following formulation:  Find $\\(\rho, \bbm, P) \in \commentA{ \calM \CROSS \mathbf{\bcalV} \CROSS \calQ }$ such that 
\begin{equation} \label{eq:velocity_general_formulation}
\begin{alignedat}{2}
( \new{\p_t} \rho , w) + (\bu \SCAL \GRAD \rho, w) + \alpha_\rho ((\DIV \bu) ( \rho - \Bar{\rho} ) , w) + ( \kappa \GRAD \rho , \GRAD w )  &= 0, \quad &&\forall w \in \calM, \\
( \rho  \new{\p_t} \bu, \bv)  + b( \bbm, \bu, \bv ) + \alpha_\bbm( ( \DIV \bu ) \bbm, \bu) + (\GRAD P, \bv) \\  + \l(  \alpha_P + \frac{1}{2} \r) b(\bv,\bbm,\bu) + \l(  \alpha_P  - \frac{1}{2} \r) b(\bv,\bu,\bbm) - \frac{1}{2} \alpha_\rho (   \GRAD( \bbm \SCAL \bu) , \bv ) \\ - \frac{1}{2} \l( (\bv \SCAL \GRAD \rho, \commentC{ \alpha_{\uu} \uuDG + (1 - \alpha_{\uu}) \uu }) -  \alpha_\rho \l(  \GRAD \l( \l(\rho - \Bar{\rho} \r) \commentC{ \l( \alpha_{\uu} \uuDG + (1 - \alpha_{\uu}) \uu \r) } \r) ,\bv  \r)  \r) \\
 +      \alpha_\rho (( \DIV \bu )     \Bar{\rho} \bu , \bv  ) +  \alpha_\rho (    \GRAD \l(    \Bar{\rho} \bu  \SCAL \bu \r) , \bv  )   
 + \l( \mu \l(   \GRAD \bu   + (\GRAD \bu)^\top  \r)  , \GRAD \bv \r)  \\ 
  -( \bef, \bv)  - ( \bef_m, \bv ) -  b( \kappa \GRAD \rho , \bu , \bv  ) -  b( \kappa \GRAD \rho , \bv , \bu  ) &= 0,    \quad &&\forall \bv \commentA{\in} \bcalV, \\
(\DIV \bu, q)&= 0, \quad&& \forall q \in \calQ,
\end{alignedat}
\end{equation}
where $P = p - \alpha_P \rho \bu \SCAL \bu - \frac{1}{2} \alpha_\rho \Bar{\rho} \bu \SCAL \bu $. Note that integration by parts was performed on $(\DIV \bef_\rho) \bu$ inside \eqref{eq:nse_primitive_regularized}. \commentC{ The conditions for kinetic energy and momentum conservation of \eqref{eq:velocity_general_formulation} are summarized in Tables \ref{table:emac_part1}, \ref{table:emac_part2} and \ref{table:emac_part3}. The remaining properties related to density are given in Table \ref{table:emac_mom}}.

% -( \kappa \bn \SCAL \GRAD \rho , w )_{\p \Omega} - \l( \l(  \kappa \bn \SCAL \GRAD \rho \r) \bu ,  \bv \r)_{\p \Omega}
% - \l( \mu \bn \SCAL \l(  \GRAD \bu + ( \GRAD \bu )^\top   \r) , \bv  \r)_{ \p \Omega }

\subsection{Fully discrete approximation.}
Similar to Section \ref{Sec:fully_discrete} a modified Crank-Nicolson method can be applied on \eqref{eq:velocity_general_formulation}. $\rho \new{\p_t} \bu$ is discretized as $\frac{\rho^{n+1/2} \l( \bu^{n+1} -  \bu^{n} \r)}{\Delta t_n}$ and $\uuDG$ is discretized as $ \frac{1}{2} ( \overline{\bu^n \SCAL \bu^n} + \overline{ \bu^{n+1} \SCAL \bu^{n+1}} ) $. % This leads to all properties at the fully discrete level except kinetic energy. By instead discretizing $\uu$ as $\uuN$ all properties except momentum and angular momentum are fulfilled.

\subsection{Mixed formulation.} 
If solving the problem in the mixed form \eqref{eq:nse_sigma_regularized} is preferred we propose the following formulation: Find $(\rho, \bbm, P) \in \commentA{ \calM \CROSS \mathbf{\bcalV} \CROSS \calQ }$ such that

\begin{equation} \label{eq:sigma_general_formulation}
  \begin{alignedat}{2}
( \new{\p_t} \rho , w) + (\bu \SCAL \GRAD \rho, w) + \alpha_\rho ((\DIV \bu) ( \rho - \Bar{\rho} ) , w)  &= 0, \quad &&\forall w \in \calM, \\
( \sqrt{\rho} \new{\p_t} (\sqrt{\rho} \bu ), \bv)   + \frac{1}{2} (  b( \bbm, \bu, \bv ) +  b( \bu, \bbm, \bv )) + \alpha_\bbm( ( \DIV \bu ) \bbm, \bu) + (\GRAD P, \bv) \\+ \alpha_P b(\bv,\bbm,\bu) + \alpha_P b(\bv,\bu,\bbm) + \frac{1}{2} \alpha_\rho ( ( \DIV \bu ) \bar{\rho} \bu, \bv) +  \frac{1}{2} \alpha_\rho (    \GRAD \l( \bar{\rho} \bu \SCAL \bu \r) , \bv ) \\ 
 + \l( \mu \l(   \GRAD \bu   + (\GRAD \bu)^\top  \r)  , \GRAD \bv \r)  \\ 
  = ( \bef, \bv) + ( \bef_m, \bv ) + \frac{1}{2} \l(  b( \kappa \GRAD \rho , \bu , \bv  ) + b( \kappa \GRAD \rho , \bv , \bu  )  \r)  &= 0, \quad&& \forall \bv \commentA{\in} \bcalV, \\
(\DIV \bu, q) &= 0,\quad && \forall q \in \calQ,
\end{alignedat}
\end{equation}

where $P = p - \alpha_P \bbm \SCAL \bu - \frac{1}{2} \alpha_\rho \Bar{\rho} \bu \SCAL \bu $. Note that integration by parts was performed on $(\DIV \bef_\rho) \bu$ inside \eqref{eq:nse_primitive_regularized}. \commentC{ The conditions for kinetic energy and momentum conservation of \eqref{eq:sigma_general_formulation} are summarized in Tables \ref{table:emac_part1}, \ref{table:emac_part2} and \ref{table:emac_part3}. The remaining properties related to density are given in Table \ref{table:emac_mom}}.

% - \l( \l(  \kappa \bn \SCAL \GRAD \rho \r) \bu ,  \bv \r)_{\p \Omega}
% -( \kappa \bn \SCAL \GRAD \rho , w )_{\p \Omega}
% - \l( \mu \bn \SCAL \l(  \GRAD \bu + ( \GRAD \bu )^\top   \r) , \bv  \r)_{ \p \Omega }

\subsection{Fully discrete approximation}

% \subsection{Modified Crank-Nicolson method for sigma formulation}

% \begin{equation}\label{eq:sigma_time_disctetization}
% \frac{ \sqrt{  \rho^{n+1/2} } (\sqrt{  \rho^{n+1} } \bu^{n+1} - \sqrt{  \rho^{n} } \bu^{n} )}{\Delta t_n}
% \end{equation}

Similar to Section \ref{Sec:fully_discrete} the standard Crank-Nicolson method can be applied on \eqref{eq:sigma_general_formulation}. $\sqrt{\rho} \new{\p_t} (\sqrt{\rho} \bu )$ is discretized as $\rho^{n+1/2} \frac{\bu^{n+1}-\bu^n}{\Delta t_n} + \frac{1}{2} \bu^{n+1/2} \frac{ \rho^{n+1} - \rho^n }{\Delta t_n}$ and this leads to all properties at the fully discrete level except kinetic energy.

\begin{table}[H]
  \centering
  \caption{\commentC{Conditions for semi-discrete kinetic energy conservation for the inviscid model problem. Note that the conditions on $\alpha_\rho, \alpha_\bbm, \alpha_P$ are different for the different formulations.}}
  \begin{tabular}{|c|c|}
  \hline
   & $ \p_t \int_\Omega \rho \bu \SCAL \bu \ud \bx = 0$ \\ \hline
  $ \new{\p_t} \bbm \ \eqref{eq:mom_update_general}  $ & \commentC{$\alpha_{\uu} = 1$ and} $\alpha_\bbm -\alpha_P - \alpha_\rho/2 = 1/2$ \\
  $ \sqrt{\rho} \new{\p_t} (\sqrt{\rho} \bu ) $ \eqref{eq:sigma_general_formulation}  & $\alpha_\bbm -\alpha_P  = 1/2$ \\
  $ \rho \new{\p_t} \bu$ \eqref{eq:velocity_general_formulation}  & \commentC{$\alpha_{\uu} = 1$ and} $\alpha_\bbm -\alpha_P + \alpha_\rho/2 = 1/2$ \\
  \hline 
  \end{tabular}
  \label{table:emac_part1}
  \end{table}
  
  % Table 2
  \begin{table}[H]
  \centering
  \caption{\commentC{Conditions for semi-discrete momentum conservation for the inviscid model problem. Note that the conditions on $\alpha_\rho, \alpha_\bbm, \alpha_P$ are different for the different formulations. Also note that the two latter formulations require setting $w = \bu \SCAL \be_i$ when \neww{proving} momentum conservation, which means that we require $k_\bu \leq k_\rho$.}}
  \begin{tabular}{|c|c|}
  \hline
   & $\p_t \int_\Omega \bbm \ud \bx = 0$ \\ \hline
  $ \new{\p_t} \bbm \ \eqref{eq:mom_update_general}  $ & \commentC{$\alpha_{\uu} = 0$ and} $\alpha_\bbm = 1$ \\
  $ \sqrt{\rho} \new{\p_t} (\sqrt{\rho} \bu ) $ \eqref{eq:sigma_general_formulation}  & \commentC{$k_\bu \leq k_\rho$ and} $\alpha_\rho/2 + \alpha_\bbm = 1$ \\
  $ \rho \new{\p_t} \bu$ \eqref{eq:velocity_general_formulation}  & \commentC{$\alpha_{\uu} = 0, k_\bu \leq k_\rho$ and} $\alpha_\rho + \alpha_\bbm = 1$ \\
  \hline 
  \end{tabular}
  \label{table:emac_part2}
  \end{table}
  
  % Table 3
  \begin{table}[H]
  \centering
  \caption{\commentC{Conditions for semi-discrete angular momentum conservation for the inviscid model problem. Note that the conditions on $\alpha_\rho, \alpha_\bbm, \alpha_P$ are different for the different formulations. Also note that the two latter formulations require setting $w = \bu \SCAL \bphi_i$ when \neww{proving} angular momentum conservation, which means that we require $(k_\bu + 1) \leq k_\rho$.}}
  \begin{tabular}{|c|c|}
  \hline
   & $ \p_t \int_\Omega \bbm \CROSS \bx \ud \bx = 0$ \\ \hline
  $ \new{\p_t} \bbm \eqref{eq:mom_update_general}  $ & \commentC{$\alpha_{\uu} = 0$ and} $\alpha_\bbm = 1$ \\
  $ \sqrt{\rho} \new{\p_t} (\sqrt{\rho} \bu ) $ \eqref{eq:sigma_general_formulation}  & \commentC{$(k_\bu + 1) \leq k_\rho$ and} $\alpha_\rho/2 + \alpha_\bbm = 1$ \\
  $ \rho \new{\p_t} \bu$ \eqref{eq:velocity_general_formulation}  & \commentC{$\alpha_{\uu} = 0, (k_\bu + 1)\leq k_\rho$ and} $\alpha_\rho + \alpha_\bbm = 1$ \\
  \hline 
  \end{tabular}
  \label{table:emac_part3}
  \end{table}

\section{Lock-exchange 3D.} \label{Sec:lock exchange 3D}

In this section, we consider the lock-exchange problem in 3D. Even though the theory is developed for the Crank-Nicolson scheme, one can use the viscous regularization and formulations developed in this work with different time-integration schemes such as high-order (backward differentiation formula) BDF methods. To this end, we consider the variable time-step BDF4 method from \citep{Lundgren_2023} which we modify to be fully implicit.

We follow the setup of \cite{Bartholomew_2019}. The setup is similar to the 2D case with a few exceptions. No-slip boundary conditions are used on the top and bottom boundaries and slip is used on the remaining boundaries. \commentA{These boundary conditions are imposed strongly through the linear systems. A Neumann condition $\bn \SCAL \GRAD \rho|_{\partial \Omega} = 0,$ is automatically imposed weakly on all boundaries through the weak form \eqref{eq:full_time_disctetization}.} The Reynolds number is computed in the same as in 2D, but now it is set to $Re = 2236$ instead. The domain is set to $\Omega = \{ (x,y \commentD{, z} ) \in (-1L, 17L) \CROSS (0, 2L) \CROSS (0, 2L) \}$. 

A random perturbation is added to the initial velocity
\begin{equation}
\bu(0, \bx) = X \cdot 0.05 \cdot \exp{(-25 x^2)} + (0 , 0, 0)^\top,
\end{equation}
where $X \sim U(-1,1)$, a uniform distribution of random numbers in $(-1,1)$. The authors of \citep{Bartholomew_2019} added this perturbation `to simulate that a physical barrier is removed'.

We compute the solution using 15,696,120 $\polP_3$ nodes on a uniform mesh which corresponds to a similar mesh resolution as \cite{Bartholomew_2019}. We also set $\text{CFL} = 0.1$ and use the LSI-EMAC formulation using $\commentA{  \polP_3 \CROSS \polP_3 \CROSS \polP_2 }$ elements. The nonlinear system is solved using the relative tolerance $10^{-7}$. The results are presented at $T = 20$ in Figure \ref{fig:le-3d}. In the figure, we compare the Guermond-Popov viscous flux ($\bef_{GP}$) and the momentum and angular momentum conserving flux that \cite{Bartholomew_2019} use ($\bef_0$). The simulations closely match \cite{Bartholomew_2019}, but there is no pointwise match due to the random perturbations in the initial condition and differences in the numerical method. %The simulations took roughly 50 000 core hours each on the Uppmax supercomputer.

% \lukas{Improve figures later.}

\begin{figure}[H]
\centering 
\includegraphics[scale=0.38]{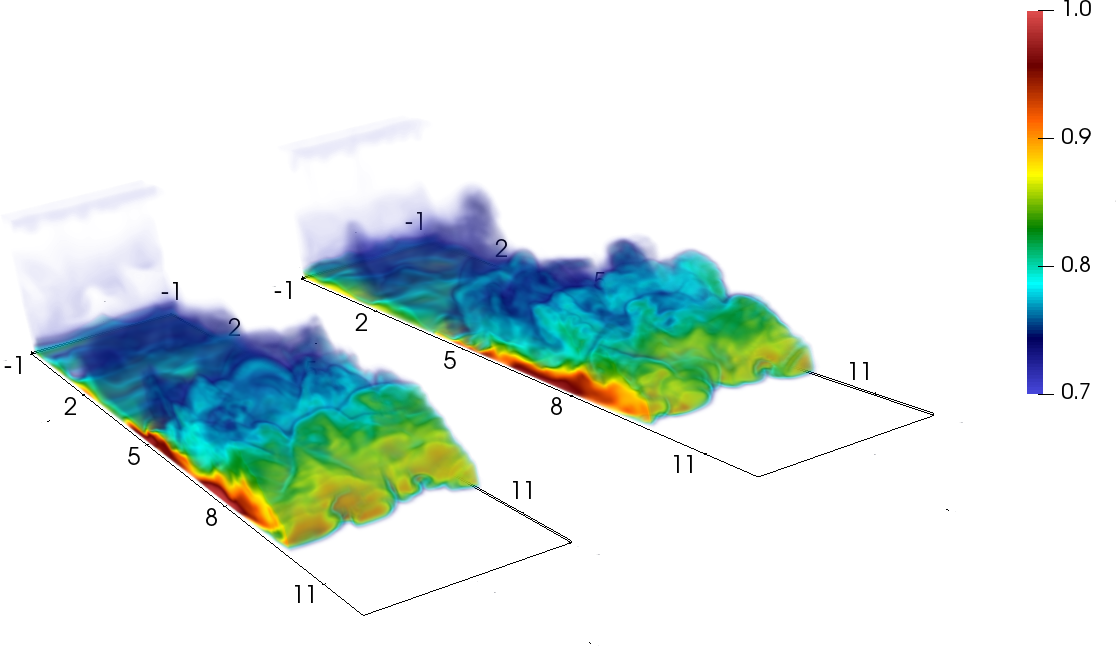}
\caption{Lock-exchange problem in 3D at $Re = 2236$ and Schmidt number 1. Density $\rho$ at $T = 20$. Left: Momentum and angular momentum conserving viscous flux $\bef_0$. Right: Guermond-Popov viscous flux $\bef_{GP}$.}
\label{fig:le-3d}
\end{figure}

% double noise = 0.05; 

%   IndataVelocity3D() : Expression(3) {}

%   void eval(Array<double> &values, const Array<double> &x) const {
%   float random_float;
    
%     random_float = Random();
%     values[0] = random_float * noise * exp(-25. * pow( x[0]  - 0.0 ,2));

%     random_float = Random();
%     values[1] = random_float * noise * exp(-25. * pow( x[0]  - 0.0 ,2));

%     random_float = Random();
%     values[2] = random_float * noise * exp(-25. * pow( x[0]  - 0.0 ,2));

%   }

%   private:
%   float Random() const {
%   float random_float = (float)rand() / RAND_MAX;
%     int minus = rand() % 2;
%     minus = minus * (-1) + (1 - minus) * 1;

%     random_float = random_float * minus;
%     return random_float;
%   }

% 3D. Barthemal used 16924800 mesh nodes. 2d nodes: 2889901, 3d nodes bubble rv paper: 1498861. will be expensive...

% \newpage
% \section*{References}

%\bibliographystyle{siam}
\bibliographystyle{abbrvnat} 
\bibliography{ref}
% \end{document}

% \bibliographystyle{siam}
% \bibliography{ref}

\end{document}